\RequirePackage{fix-cm}
\documentclass[smallextended,final]{svjour3}
\smartqed
\usepackage{setspace,url}
\usepackage{graphics,graphicx,epstopdf}
\usepackage{tabularx}
\usepackage{longtable}
\usepackage{booktabs,multirow,array,multicol}
\usepackage{dsfont}
\usepackage{enumitem}
\usepackage{subfig}
\usepackage[leqno]{amsmath}
\usepackage{amsxtra,amsfonts,amscd,amssymb,bm,epsf}
\usepackage{mathtools}
\usepackage{comment}
\usepackage[algo2e,linesnumbered,vlined,ruled]{algorithm2e}
\usepackage{amsopn}
\usepackage{colortbl}
\usepackage{xcolor}
\usepackage[normalem]{ulem}
\usepackage{times}

\ifpdf
  \DeclareGraphicsExtensions{.eps,.pdf,.png,.jpg}
\else
  \DeclareGraphicsExtensions{.eps}
\fi

\colorlet{color1}{blue}
\colorlet{color2}{red!50!black}

\usepackage{hyperref}[6.83]
\hypersetup{
  colorlinks=true,
  frenchlinks=false,
  pdfborder={0 0 0},
  naturalnames=false,
  hypertexnames=false,
  breaklinks,
  allcolors = color1,
  urlcolor = color2,
}

\numberwithin{theorem}{section}
\numberwithin{equation}{section}

\numberwithin{lemma}{section}
\numberwithin{corollary}{section}

\newtheorem{assumption}{Assumption}

\newcommand{\R}{\mathbb{R}}
\newcommand{\Rn}{\mathbb{R}^n}

\newcommand{\Rex}{(-\infty,+\infty]}

\newcommand{\N}{\mathbb{N}}

\newcommand{\Spp}{\mathbb{S}^n_{++}}
\newcommand{\dom}{\mathrm{dom}}

\newcommand{\vp}{\varphi}
\newcommand{\veps}{\varepsilon}

\newcommand{\cH}{\mathcal{H}}

\newcommand{\cS}{{\mathcal{S}}}

\newcommand{\proxt}[2]{{\mathrm{prox}}^{#1}_{#2}}

\newcommand{\envt}[2]{{\mathrm{env}}^{#1}_{#2}}

\newcommand{\Exp}{\mathbb{E}} 
\newcommand{\Prob}{\mathbb{P}}

\newcommand{\vhalf}{v_{k+{1\!\!\;/\!\!\;2}}}

\newcommand{\aev}{$\text{a.e. }$}
\newcommand{\as}{$\text{a.s. }$}

\newcommand{\half}{\frac{1}{2}}
\newcommand{\iprod}[2]{\langle #1, #2 \rangle}

\DeclareMathOperator*{\argmin}{arg\,min}

\newcommand{\be}{\begin{equation}}
\newcommand{\ee}{\end{equation}}
\newcommand{\bee}{\begin{equation*}}
\newcommand{\eee}{\end{equation*}}

\newcommand{\bea}{\begin{eqnarray}}
\newcommand{\eea}{\end{eqnarray}}
\newcommand{\beaa}{\begin{eqnarray*}}
\newcommand{\eeaa}{\end{eqnarray*}}

\definecolor{ivory}{RGB}{218,215,203}

\newcommand{\celliv}[1]{\cellcolor{ivory} {#1}}
\newcolumntype{C}[1]{>{\centering\arraybackslash}p{#1}}
\newcolumntype{M}{@{\hspace{.6\tabcolsep}}>{\columncolor{ivory!45}[.8pt][.4\tabcolsep]}C{1.2cm}@{\hspace{.8\tabcolsep}}}
\newcolumntype{R}{@{\hspace{.6\tabcolsep}}>{\columncolor{ivory!45}[.8pt][.4\tabcolsep]}C{2.0cm}@{\hspace{.8\tabcolsep}}}

\begin{document}
\title{A Stochastic Extra-Step Quasi-Newton Method for Nonsmooth Nonconvex Optimization
\thanks{Z$.$ Wen is supported in part by the NSFC grants 11831002 and 11421101. A$.$ Milzarek is partly supported by the Fundamental Research Fund -- Shenzhen Research Institute for Big Data (SRIBD) Startup Fund {JCYJ-AM20190601}.}}

\titlerunning{A Stochastic Extra-Step Quasi-Newton Method for Nonsmooth Problems}        

\author{Minghan Yang \and Andre Milzarek \and Zaiwen Wen \and Tong Zhang} 

\institute{Minghan Yang \at
              Beijing International Center for Mathematical Research, BICMR, Peking University, Beijing, CHINA \\
              \email{yangminghan@pku.edu.cn} \and
              Andre Milzarek \at
              Institute for Data and Decision Analytics, The Chinese University of Hong Kong, Shenzhen and Shenzhen Institute for Artificial Intelligence and Robotics for Society, AIRS, Shenzhen, CHINA \\
              \email{andremilzarek@cuhk.edu.cn} \and
              Zaiwen Wen \at
              Beijing International Center for Mathematical Research, BICMR, Peking University, Beijing, CHINA \\
              \email{wenzw@pku.edu.cn} \and
              Tong Zhang \at
              Hong Kong University of Science and Technology, Hong Kong, CHINA \\
              \email{tongzhang@tongzhang-ml.org} 
}

\date{Received: date / Accepted: date}

\maketitle


\begin{abstract}
In this paper, a novel stochastic extra-step quasi-Newton method is developed to solve a class of nonsmooth nonconvex composite optimization problems. We assume that the gradient of the smooth part of the objective function can only be approximated by stochastic oracles. The proposed method combines general stochastic higher order steps derived from an underlying proximal type fixed-point equation with additional stochastic proximal gradient steps to guarantee convergence. Based on suitable bounds on the step sizes, we establish global convergence to stationary points in expectation and an extension of the approach using variance reduction techniques is discussed. Motivated by large-scale and big data applications, we investigate a stochastic coordinate-type quasi-Newton scheme that allows to generate cheap and tractable stochastic higher order directions. Finally, the proposed algorithm is tested on large-scale logistic regression and deep learning problems and it is shown that it compares favorably with other state-of-the-art methods.
\end{abstract}


\section{Introduction} In this paper, we investigate a two-step, stochastic higher order method for composite-type optimization problems of the form:
\be \label{eq:prob} \min_{x \in \Rn}~\psi(x) := f(x) + \vp(x),\ee
where $f : \Rn \to \R$ is a continuously differentiable but not necessarily convex function and $\vp : \Rn \to \Rex$ is a convex, lower semicontinuous, and proper mapping. We assume that a full evaluation of the function $f$ and its gradient $\nabla f$ is either too expensive or not  possible and that approximate gradient information can only be accessed via calling \textit{stochastic oracles}. 
Indeed, such a situation occurs frequently in large-scale and big data applications where the function $f$ corresponds to a data term or statistical loss model that is parametrized by the vector $x \in \Rn$. 
%
For instance, $f$ can be chosen as an expected risk mapping:
\be \label{eq:exrm} f(x) := \Exp[F(x,\Xi)]=\int_{\Omega}F(x,\Xi(\omega))\,{\rm d}\Prob(\omega), \ee
where $(\Omega,\mathcal F,\Prob)$ is a given probability space, $\Xi : \Omega \to \Upsilon$ is a random variable, $\Upsilon$ is a measure space, and $F: \Rn \times \Upsilon \to \R$ is a model function. 
Since the distribution $\Prob$ is often not known and only partial information in form of finitely many data samples is available, the following empirical risk formulation is typically considered in practice:
%
\be \label{eq:erm} f(x):= \frac1N \sum_{i=1}^N f_i(x), \quad f_i : \Rn \to \R, \quad i = 1,...,N. \ee
%
The function $\vp$ is an additional regularization term that is introduced to appropriately handle the risk of overfitting or to promote a certain structure of the variable $x$, such as (group-) sparsity or low rank properties. 

Composite-type problems of the form \eqref{eq:prob} are common in machine learning \cite{Bis06,shalev2014understanding,BotCurNoc18}, statistical learning \cite{friedman2001elements,vapnik2013nature}, sparse logistic regression \cite{REF08a,shi2010fast,SST2011}, and image or signal processing \cite{ComPes11}. Examples for challenging large-scale and nonconvex applications that fit within our proposed framework comprise matrix decomposition \cite{ChaSanParWil09,CanRec09}, deep learning \cite{MasBaxBarFre99,DenYu14,simonyan2014very,lecun2015deep,Sch15,he2016deep}, and structured dictionary learning \cite{mairal2009online,BacJenMaiObo11}.

\subsection{Contribution} In this work, we develop a stochastic extra-step quasi-Newton method for the general composite problem \eqref{eq:prob}. Our basic idea is to utilize stochastic higher order Newton-type steps that are designed to solve the nonsmooth, nonlinear equation
\be \label{eq:nonlinear}F^\Lambda(x) := x - \proxt{\Lambda}{\vp}(x-\Lambda^{-1}\nabla f(x)) = 0, \quad \Lambda \in \Spp, \ee
which represents the associated first order optimality conditions of \eqref{eq:prob}. In order to guarantee global convergence of the proposed approach, we introduce and perform an additional stochastic proximal gradient step. For a generated stochastic process of iterates $({\sf X}^k)_k$, we can then establish the following convergence results
\be \label{eq:intro-conv} \Exp[ \| F^I({\sf X}^k) \|^2] \to 0   \quad \text{and} \quad F^I({\sf X}^k) \to 0 \quad \text{almost surely}\quad   k \to \infty. \ee
We now summarize our different contributions.
\begin{itemize}[leftmargin = 4 ex]
\item We propose a general stochastic extra-step scheme for nonsmooth and nonconvex optimization problems of the form \eqref{eq:prob} that allows to  incorporate stochastic higher order information in a natural and simple way. 
\item The introduced extra-step ensures that the sequence of function values $(\psi({\sf X}^k))_{k}$ is approximately decreasing without requiring expensive line search procedures or checking additional conditions. Specifically, global convergence is achieved if the step sizes and stochastic errors are correctly balanced and satisfy certain summability conditions. We present different step size strategies that put emphasis either on the efficient usage of stochastic higher order steps and information or on weaker assumptions on the variance of the stochastic approximations.
\item  A large variety of stochastic approximation methods, such as basic sub-sampling schemes or more elegant variance reduction techniques, \cite{JohZha13,DefBacSim14,XiaZha14,RedHefSraPocSmo16,nguyen2017sarah}, can be applied within our abstract framework. To demonstrate this versatility, we investigate a variance reduced (SVRG-type) version of our approach for empirical risk problems. We show that an $\veps$-accurate stationary point with $\Exp[\|F^I({\sf X})\|^2] \leq \veps$ can be recovered within $\mathcal O(N^{2/3}/\veps)$ iterations which is the same iteration complexity as Prox-SVRG, \cite{AllHaz16,RedHefSraPocSmo16}.
\item We propose and analyze a stochastic coordinate-based quasi-Newton strategy to generate higher order-type directions for large-scale problems. Numerical experiments on $\ell_1$-regularized logistic regression and sparse deep learning problems illustrate the efficiency and promising performance of our two-step method and the high potential of stochastic higher order information in large-scale settings. 
\end{itemize}


\subsection{Related Work} 
In this section, we briefly review related work and research directions on large-scale nonsmooth and nonconvex optimization. 

The stochastic gradient descent method (SGD) originates from the seminal work \cite{RobMon51} by Robbins and Monro and uses single sample or mini-batch stochastic oracles to approximate gradient information at each iteration. The SGD method is a highly popular and flexible approach and many extensions, such as variance reduction techniques and acceleration schemes, have been proposed to further enhance its practical performance and theoretical properties, see, e.g., \cite{JohZha13,SSZ2013,DefBacSim14,nguyen2017sarah,LMH2015,All17-Kat,SchLeRBac17}. 
For nonsmooth problems, i.e., if $\vp \neq 0$, the proximal operator of $\vp$ can be utilized to develop stochastic proximal gradient methods, \cite{DefBacSim14,XiaZha14,SSZ2016,pham2019proxsarah}, that have similar convergence properties as SGD.
The convergence analysis of SGD and its corresponding variance reduced versions can also be extended to the class of nonconvex nonsmooth problems considered in this work, see, e.g., \cite{GhaLan13,GhaLan16,AllHaz16,RedHefSraPocSmo16,RedSraPocSmo16,lei2017non,fang2018spider,wang2018spiderboost,zhou2018stochastic,nguyen2019finite,pham2019proxsarah}. 

Most of the existing stochastic second order algorithms are designed for smooth and convex finite-sum problems and are based on various sub-sampling strategies to build cheap oracles for approximating the gradient and Hessian of $f$, see, e.g., \cite{byrd2011use,ErdMon15,xu2016sub,BolByrNoc18,RooMah18}. In \cite{XuRooMah19}, Xu et al. analyze a trust-region framework for nonconvex problems with inexact or stochastic Hessian information. The Newton sketch method by Pilanci and Wainwright, \cite{pilanci2017newton}, uses randomly projected Hessians as oracles for smooth convex problems with decomposable Hessians and a comparison of Newton sketch and sub-sampled Newton methods is presented in \cite{BerBolNoc17}. Stochastic quasi-Newton methods are another class of stochastic higher order algorithms. Typically, these methods combine different stochastic oracles for the gradient and (L)-BFGS-type mechanisms to construct tractable stochastic quasi-Newton updates for both convex and nonconvex problems. 
In \cite{schraudolph2007stochastic}, Schraudolph et al. develop an online version of the BFGS method. In  \cite{ByrHanNocSin16}, Byrd et al. propose a stochastic L-BFGS algorithm that utilizes sub-sampled Hessian information to build the BFGS-type updates. A linearly convergent stochastic L-BFGS method with variance reduction is discussed in \cite{moritz2016linearly}. Zhao et al., \cite{zhao2017stochastic}, use an adaptive sampling technique to further improve the iteration complexity. Wang et al., \cite{WanMaGolLiu17}, consider a stochastic L-BFGS  approach for nonconvex problems and introduce a Powell-damping scheme to guarantee positive definiteness of the updates. In \cite{gower2016stochastic}, Gower et al. propose a stochastic block L-BFGS method that incorporates enriched curvature information. Moreover, Mokhtari et al., \cite{mokhtari2018iqn}, analyze an incremental BFGS method with local superlinear rate of convergence. For nonsmooth problems, deterministic proximal Newton-type methods are studied in \cite{LeeSunSau14,hsieh2014quic,yuan2012improved} and stochastic versions have been considered in \cite{ShiLiu15,WanWanYua18,WanZha19}. At each iteration, these methods need to solve a nonsmooth quadratic subproblem to obtain a new direction.
%
Semismooth Newton-type algorithms are another well-known class of second order methods that can be applied to solve the problem \eqref{eq:prob} or the nonsmooth equation \eqref{eq:nonlinear}, \cite{QiSun93,Qi93,PanQi93,SteThePat17,XiaLiWenZha18}. In contrast to proximal Newton methods, a semismooth Newton step is generated more directly and only involves the (approximate) solution of a linear system of equations. In \cite{MilXiaCenWenUlb18}, a stochastic variant of the semismooth Newton method is developed and convergence is established in expectation and almost surely. We refer to section \ref{sec:algorithm} and \ref{section:direction} for further discussions and more information.

Deep learning, \cite{simonyan2014very,he2016deep,lecun2015deep}, is a highly successful but computationally demanding machine learning methodology and it can take days to train a model to reach a desired level of accuracy. Although stochastic first order approaches \cite{kingma2014adam,duchi2011adaptive,sutskever2013importance,goyal2017accurate,you2018imagenet,AkiSuzFuk17} are the dominating methods of choice, stochastic second order schemes have recently gained more attention \cite{BotRitBar17,osawa2018second,martens2015optimizing,grosse2016kronecker,martens2010deep,BerNocTak16} -- especially in large mini-batch settings -- and show their own advantages. 

\subsection{Organization}
This paper is organized as follows. In section 2, we present the abstract algorithmic framework and introduce several preliminaries. In section 3, global convergence is investigated for different step size strategies. Moreover, an extended version of the algorithm using variance reduction is proposed and discussed. In section 4, we introduce a stochastic coordinate-type quasi-Newton method and other higher order-type schemes. Finally, in section 5 and 6, numerical experiments are performed to demonstrate the efficiency of the extra-step method.


\subsection{Notation}

By $\iprod{\cdot}{\cdot}$ and $\|\cdot\| := \|\cdot\|_2$ we denote the standard Euclidean inner product and norm. For matrices, the norm $\|\cdot\|$ is the standard spectral norm. The set of symmetric and positive definite $n \times n$ matrices is denoted by $\Spp$. For a given matrix $\Lambda \in \Spp$, we define the inner product $\iprod{x}{y}_\Lambda := \iprod{x}{\Lambda y} = \iprod{\Lambda x}{y}$ and $\|x\|_\Lambda := \sqrt{\iprod{x}{x}_\Lambda}$.
For any $n \in \N$, we set $[n] := \{1,...,n\}$ and $[n]_0 = \{0\} \cup [n]$. Let $(\Omega,\mathcal F,\Prob)$ be a probability space. We will use uppercase letters and a sans-serif letterform to describe random variables ${\sf X} : \Omega \to \Rn$, while lowercase letters or letters with serifs are typically reserved for realizations of a random variable, $x = {\sf X}(\omega)$, or deterministic parameters. We use $L^p(\Omega) := L^p(\Omega,\Prob)$, $p \in [1,\infty]$, to denote the standard $L^p$ space on $\Omega$. We write ${\sf X} \in \mathcal F$ for ``${\sf X}$ is $\mathcal F$-measurable''. Moreover, $\sigma({\sf X}^1,...,{\sf X}^k)$ denotes the $\sigma$-algebra generated by the family of random variables ${\sf X}^1,...,{\sf X}^k$. For a random variable ${\sf X} \in L^1(\Omega)$ and a sub-$\sigma$-algebra $\cH \subseteq \mathcal F$, the conditional expectation of ${\sf X}$ given $\cH$ is denoted by $\Exp[{\sf X} \mid \cH]$. We use the abbreviations ``a.e.'' and ``a.s.'' for ``almost everywhere'' and ``almost surely'', respectively.

\section{A Stochastic Extra-Step Quasi-Newton Method} We first discuss the underlying first order optimality conditions of the composite-type problem \eqref{eq:prob}. In subsection \ref{sec:algorithm}, we then present the proposed approach in detail.


\subsection{Preliminaries and First Order Optimality} It is well-known that the stationarity conditions of problem \eqref{eq:prob} can be equivalently rewritten as a proximal fixed-point equation, i.e., a point $x \in \dom~\vp$ is a critical point of \eqref{eq:prob} if and only if
\be \label{eq:opt} F^\Lambda(x) := x - \proxt{\Lambda}{\vp}(x-\Lambda^{-1}\nabla f(x)) = 0, \quad \Lambda \in \Spp, \ee
where $\proxt{\Lambda}{\vp} : \Rn \to \Rn$, $\proxt{\Lambda}{\vp}(x) := \argmin_y \vp(y) + \frac{1}{2}\|x-y\|^2_\Lambda$ denotes the proximity operator of the mapping $\vp$. The proximity operator is a $\Lambda$-firmly nonexpansive function, i.e., it is globally Lipschitz continuous and satisfies
\be \label{eq:prox-nonexp} \|\proxt{\Lambda}{\vp}(x) - \proxt{\Lambda}{\vp}(y)\|_\Lambda^2 \leq \iprod{\Lambda(x-y)}{\proxt{\Lambda}{\vp}(x) - \proxt{\Lambda}{\vp}(y)}, \quad \forall~x,y. \ee
The proximity operator $\proxt{\Lambda}{\vp}$ can also be alternatively characterized via the underlying optimality condition
\be \label{eq:prox-opt} \proxt{\Lambda}{\vp}(x) \in x - \Lambda^{-1} \partial \vp(\proxt{\Lambda}{\vp}(x)), \ee
where $\partial \vp$ is the classical subdifferential of the convex function $\vp$. The so-called Moreau envelope of $\vp$ is given by $\envt{\Lambda}{\vp} : \Rn \to \R$, $\envt{\Lambda}{\vp}(x) := \min_y \vp(y) + \half \|x-y\|_\Lambda^2$. The mapping $\envt{\Lambda}{\vp}$ is real-valued, continuously differentiable, and convex, and its gradient satisfies $\nabla \envt{\Lambda}{\vp}(x) = \Lambda(x - \proxt{\Lambda}{\vp}(x))$. Let us refer to \cite{Mor65,ComWaj05,BacJenMaiObo11,BauCom11,ParBoy14} for further details and background on the proximity operator and Moreau envelope.
%
%
%
%
%

Given a generic (stochastic) oracle represented by a direction $v \in \Rn$, we consider inexact variants of the nonsmooth residual \eqref{eq:opt} of the type
\begin{equation}
\label{sto_app_optimal}  F^\Lambda_v(x) := x - \proxt{\Lambda}{\vp}(x- \Lambda^{-1} v). 
\end{equation}
We will also use $u^\Lambda_v(x) := x - \Lambda^{-1} v$ and $p^\Lambda_v := \proxt{\Lambda}{\vp}(u^\Lambda_v(x))$ to denote the inexact (proximal) gradient steps. In the case $v \equiv \nabla f(x)$, the terms $u^\Lambda$ and $p^\Lambda$ are used to denote the exact or full gradient and proximal gradient step.


\subsection{Algorithmic Framework} \label{sec:algorithm} Our algorithmic idea is to utilize stochastic higher order information and to generate stochastic second order-type directions $d $ based on the optimality condition \eqref{eq:opt} and its approximate variant \eqref{sto_app_optimal}. Specifically, we are interested in directions of the form
\be \label{dirup} d = -W F^\Lambda_v(x), \ee
where the matrix $W \in \R^{n \times n}$ is chosen to refine and improve the basic stochastic proximal gradient direction $-F^\Lambda_v(x)$ and $v \approx \nabla f(x)$ is a corresponding stochastic approximation of the gradient of $f$. In each iteration, we first calculate a new trial point via $z = x + \beta d$ and then perform an additional proximal gradient step to obtain the next iterate $x_+$: 
\be \label{eq:algo-scheme} \left[ \begin{array}{rl} z & \hspace{-.25ex} = x + \beta d, \\
x_{+} & \hspace{-.25ex}= \proxt{\Lambda_+}{\vp}(x + \alpha d - \Lambda_+^{-1}v_+), \end{array} \right. \ee
where $\alpha, \beta \geq 0$, and $\Lambda, \Lambda_+ \in \Spp$ are suitable step size parameters and $v_+ \in \Rn$ is a stochastic approximation of the trial gradient value $\nabla f(z)$. The full method is presented in Algorithm \ref{alg:seqn}.

\LinesNumberedHidden
\begin{algorithm2e}[t]
\caption{\textbf{S}tochastic \textbf{E}xtra-Step \textbf{Q}uasi-\textbf{N}ewton Method}    
\label{alg:seqn}
\lnlset{alg:seqn-1}{1}{Initialization: ~~Select the parameter matrices $(\Lambda_k)_k, (\Lambda_{k,+})_k \in \Spp$. Choose the step sizes $(\alpha_k)_k$, $(\beta_k)_k$ and the initial point  $x^0 \in \dom~\vp$ .} \\ 
\For{$k = 0,1,... $}{
\lnlset{alg:seqn-2}{2}{Based on $v^k \approx \nabla f(x^k)$, compute a direction $d^k$ and set $z^k = x^k + \beta_k d^k$.} \\
\lnlset{alg:seqn-3}{3}{Calculate the new oracle $v^k_+ \approx \nabla f(z^k)$ and perform the update \\
\begin{center} $x^{k+1} = \proxt{\Lambda_{k,+}}{\vp}(x^k + \alpha_k d^k - \Lambda_{k,+}^{-1}v^k_+)$ \vspace{-1ex}.\end{center}}  
} 
\end{algorithm2e}

The scheme \eqref{eq:algo-scheme} is a general stochastic two-step framework that supports a large variety of stochastic oracles and mechanisms to construct the approximations $v$, $v_+$. In particular, in section \ref{section:extra-vr}, we discuss a more sophisticated version of our algorithm using the variance reduction technique proposed in \cite{JohZha13,XiaZha14,RedHefSraPocSmo16}. The two step sizes $\alpha_k$ and $\beta_k$ in Algorithm \ref{alg:seqn} can be selected in a very flexible way (e.g., it is possible to set $\alpha_k = \beta_k = 1$ for all $k$). Hence, several known stochastic algorithms can be regarded and treated as special cases of our method: 
\begin{itemize}[leftmargin=4ex]
\item In the case $\alpha_k = \beta_k = 0$, $W_k = I$, the approach reduces to the standard stochastic proximal gradient descent method or to related variants (depending on the stochastic approximation schemes).
\item If we set $\alpha_k = 0$ and $W_k = I$, our algorithm coincides with a stochastic version of the extragradient method studied by Nguyen et al. in \cite{NguPauRicSut18}, see also \cite{LuoTse93}.
\end{itemize}

There is a vast plethora of possible choices for the matrices $W$ and $W_k$ in \eqref{dirup}. In this work, we will mainly focus on stochastic L-BFGS-type and coordinate-based Quasi-Newton updates which provide low cost and (block) sparse approximations of the generalized (or Fr\'echet) derivative of $F^\Lambda$. Other strategies and more details are presented in section \ref{section:direction}. 
We note that the direction $d$ is not restricted to the format \eqref{dirup}. In the next section, we will verify that more general directions can be used as long as they are appropriately related to the stochastic proximal gradient direction $F^\Lambda_v(x)$. 

Following \cite[Theorem 4.5]{PatSteBem14} and \cite[Section 4]{SteThePat17}, we now briefly demonstrate and motivate the potential higher order nature of Algorithm \ref{alg:seqn}. Let us suppose that the mapping $\nabla f$ is Lipschitz continuous with constant $L_f > 0$ and let $x^*$ be a stationary point of problem \eqref{eq:prob}. Then, choosing $\Lambda_+ = \lambda_+^{-1} I \succ 0$ and using \eqref{eq:prox-nonexp}, we obtain
\begin{align*} \|x_+ - x^*\| &= \|\proxt{\Lambda_+}{\vp}(x+\alpha d - \Lambda_+^{-1}v_+) - \proxt{\Lambda_+}{\vp}(x^* - \Lambda_+^{-1}\nabla f(x^*))\| \\ & \leq \|x+\alpha d - x^*\| + L_f \lambda_+^{-1} \|x+\beta d - x^*\| + \lambda_+^{-1} \|v_+ - \nabla f(z)\|.
\end{align*}
Based on this estimate and under certain regularity conditions (such as, e.g., \textit{semismoothness} of the residual $F^\Lambda$) and appropriate choices of $d$ and $\alpha, \beta$, the extra-step scheme \eqref{eq:algo-scheme} can yield fast local convergence with high probability if the stochastic errors are suitably controlled and decrease sufficiently fast throughout the iterative process. For recent and related discussions of local convergence results for stochastic higher order methods for smooth problems we refer to \cite{ErdMon15,Mut16,AgaBulHaz17,KohLuc17,YeLuoZha17,BolByrNoc18,RooMah18}. We will investigate the latter observations and study the local behavior of Algorithm \ref{alg:seqn} in detail in a companion paper. 

The additional proximal gradient step in \eqref{eq:algo-scheme} was initially introduced by Patrinos et al. in \cite{PatSteBem14} (and later extended in \cite{SteThePat17}) to improve the global convergence properties and convergence rates of a family of nonsmooth Newton-type methods for a merit function formulation of problem \eqref{eq:prob} based on the so-called \textit{forward-backward envelope}. In this work, we extend and generalize this strategy to the inexact and stochastic setting when the full function and gradient values $f(x)$ and $\nabla f(x)$ are not available or intractable to compute. In contrast to \cite{PatSteBem14,SteThePat17,TheStePat18}, we do not need to perform line-search on the forward-backward envelope (which depends on the full objective function $\psi$) to ensure global convergence. Instead, we will show that the simple extra-step scheme \eqref{eq:algo-scheme} achieves martingale-type descent and converges in the sense of \eqref{eq:intro-conv}, if the two step size matrices $\Lambda$ and $\Lambda_+$ are chosen in a specific way and, as already mentioned, the direction $d$ is related to the stochastic residual $F^\Lambda_v(x)$. 

Compared to stochastic proximal Newton-type methods, \cite{ShiLiu15,RodKro16,WanWanYua18,WanZha19}, our approach generates stochastic higher order-type steps more directly. In particular, we do not need to solve potentially expensive subproblems (to obtain the proximal Newton direction) in each iteration. Let us also mention that the stochastic second order method studied by Milzarek et al. in \cite{MilXiaCenWenUlb18} considers similar stochastic semismooth Newton steps of the form $x_+ = x - W F^\Lambda_v(x)$. However, the authors apply a more intricate globalization mechanism based on growth conditions and, in general, evaluation of the full objective function $\psi$ is required to ensure global convergence.

\section{Global Convergence Analysis}


\subsection{Assumptions, First Properties, and Stochastic Tools}
Throughout this paper, we assume that $f : \Rn \to \R$ is continuously differentiable on $\Rn$ and $\vp : \Rn \to \Rex$ is convex, lower semicontinuous, and proper function. 
In the following, we further specify the assumptions on the mappings $f$ and $\vp$. 

\begin{assumption} \label{ass:lip} Let $f : \Rn \to \R$ be given. We assume: 
\setlength{\leftmargini}{7ex}
\begin{itemize}
\item[{\rm(A.1)}] The gradient mapping $\nabla f$ is Lipschitzian on $\Rn$ with modulus $L_f \geq 1$. 
\item[{\rm(A.2)}] The objective function $\psi$ is bounded from below on $\dom~\vp$. \vspace{.5ex}
\end{itemize}
\end{assumption}

%
%
%
%
%
%


We now verify that the oracle-based extra-step scheme can yield approximate $\psi$-descent whenever the parameter matrices $\Lambda$ and $\Lambda_+$ are chosen appropriately. 
Our result can be seen as an inexact or oracle-based generalization of the basic properties of the forward-backward envelope presented in \cite{PatSteBem14,SteThePat17}. Let us note that similar results for proximal gradient steps were also shown in \cite{XiaZha14,XuYin15,GhaLanZha16,RedSraPocSmo16} and that the proof of Lemma \ref{lemma:prox-descent} mainly relies on the well-known descent lemma
\be \label{eq:lip-ineq} f(y) \leq f(x) + \iprod{\nabla f(x)}{y-x} + \frac{L_f}{2}  \|y-x\|^2, \quad \forall~x,y \in \Rn, \ee  
which is a direct consequence of assumption {\rm(A.1)}.

\begin{lemma} \label{lemma:prox-descent} Let $x \in \dom~\vp$, $d \in \Rn$, $\Lambda = \lambda^{-1} I $, $\Lambda_+ = \lambda_+^{-1} I$, $\alpha, \beta, \rho > 0$, and the approximations $v, v_+ \in \Rn$ be given. Under assumption {\rm(A.1)} it holds that
\begin{align*} 2[\psi(p_+) - \psi(x)] & \\ & \hspace{-13.5ex} \leq \frac{1}{\rho} \|\nabla f(x) - v\|^2 +  {\lambda_+} \|\nabla f(x+\beta d) - v_+\|^2 +  \left[ \frac{1}{\lambda_+} - \frac{1}{\lambda} \right]  \|F^\Lambda_v(x)\|^2 \\ÃÂ & \hspace{-11.5ex} + \left[ L_f - \frac{1}{\lambda_+}\right] \|p_+ - x\|^2  +  \left[ \rho - \frac{1}{\lambda} \right]  \|p^\Lambda_v(x) - p^\Lambda(x)\|^2 + {\ell(\lambda_+,\alpha,\beta)} \|d\|^2   \\ & \hspace{-11.5ex} - \frac{1}{\lambda} \|F^\Lambda(x)\|^2 +  2\iprod{\nabla f(x+\beta d) - v_+}{\lambda_+(\nabla f(x) - \nabla f(x+\beta d)) +\alpha d}, \end{align*}
where $p_+ := \proxt{\Lambda_+}{\vp}(x+\alpha d - \lambda_+v_+)$ and $\ell(\lambda_+,\alpha,\beta) := \lambda_+(\frac{\alpha}{\lambda_+} + L_f\beta)^2$. 
\end{lemma}
\begin{proof} First, setting $p = \proxt{\Gamma}{\vp}(y+u-\Gamma^{-1}w)$ and applying the optimality condition of the proximity operator \eqref{eq:prox-opt}, we have
\begin{align*} \vp(p) - \vp(z) \leq \iprod{\Gamma u - w}{p-z} + \half \|y-z\|^2_\Gamma - \half \|p-y\|^2_\Gamma - \half \|p-z\|_\Gamma^2 \end{align*}
%
for all $y,u,w \in \Rn$, $z \in \dom~\vp$, and $\Gamma \in \Spp$. Utilizing this estimate three times for different choices of $y$, $z \in \dom~\vp$, and $\Gamma$ and applying the descent lemma \eqref{eq:lip-ineq}, Young's inequality, and the Lipschitz continuity of $\nabla f$, we obtain
\begingroup
\allowdisplaybreaks
\begin{align*} \psi(p_+) - \psi(x) & \\& \hspace{-12ex} = f(p_+) - f(x) + \vp(p_+) - \vp(p^\Lambda_v(x)) \\ & \hspace{-9ex}  + \vp(p^\Lambda_v(x)) - \vp(p^\Lambda(x)) +\vp(p^\Lambda(x)) - \vp(x) \\ & \hspace{-12ex} \leq  \iprod{\nabla f(x)}{p_+-p^\Lambda_v(x)} + \frac{L_f}{2}  \|p_+-x\|^2 + \iprod{\alpha \Lambda_+ d - v_+}{p_+ - p^\Lambda_v(x)} \\ & \hspace{-9ex} + \half \left[ \frac{1}{\lambda_+} - \frac{1}{\lambda} \right] \|F^\Lambda_v(x)\|^2 - \frac{1}{2\lambda_+} \|p_+ - x\|^2 - \frac{1}{2\lambda_+} \|p_+ - p^\Lambda_v(x)\|^2 \\ & \hspace{-9ex} + \iprod{\nabla f(x)-v}{p^\Lambda_v(x) - p^\Lambda(x)} - \frac{1}{2\lambda} \left[ \|p^\Lambda_v(x) - p^\Lambda(x)\|^2 + \|F^\Lambda(x)\|^2 \right] \\ & \hspace{-12ex} \leq  \frac{\| \nabla f(x) - v\|^2}{2\rho}  +  \left[ \frac{\rho}{2} - \frac{1}{2\lambda}\right] \|p^\Lambda_v(x) - p^\Lambda(x) \|^2 +  \left[ \frac{L_f}{2} - \frac{1}{2\lambda_+} \right]ÃÂ \|p_+-x\|^2   \\  &\hspace{-9ex}  + \frac12\left[\frac{1}{\lambda_+} - \frac{1}{\lambda}  \right] \|F^\Lambda_v(x)\|^2 - \frac{1}{2\lambda} \|F^\Lambda(x)\|^2  + \frac{\lambda_+}{2} \|\nabla f(x) + \lambda_+^{-1}\alpha d - v_+\|^2. \end{align*}
%
\endgroup
Extending the penultimate term by adding and subtracting $\nabla f(x+\beta d)$ and using
\[  \| \nabla f(x) - \nabla f(x+\beta d) + \lambda_+^{-1}\alpha d \|^2 \leq \left( \frac{\alpha}{\lambda_+} + L_f\beta \right)^2 \|d\|^2 = \frac{\ell(\lambda_+,\alpha,\beta)}{\lambda_+} \|d\|^2
\]
we finally get the estimate stated in Lemma \ref{lemma:prox-descent}. 
\end{proof}

The following convergence result for supermartingale-type stochastic processes is due to Robbins and Siegmund, \cite{RobSie71}, and will be a fundamental tool in the analysis of our proposed approach. 

\begin{theorem} \label{theorem:conv-superm} Let $({\sf Y}_k)_k$, $({\sf U}_k)_k$, $({\sf A}_k)_k$, $({\sf B}_k)_k$ be sequences of nonnegative integrable random variables, adapted to the filtration $(\mathcal U_k)_k$ such that we have $\Exp[{\sf Y}_{k+1} \mid \mathcal U_k] \leq (1+{\sf A}_k) {\sf Y}_k - {\sf U}_k + {\sf B}_k$, $\sum {\sf A}_k < \infty$, and $\sum {\sf B}_k < \infty$ \as for all $k \in \N$. Then, $({\sf Y}_k)_k$ \as converges and it holds $\sum {\sf U}_k < \infty$. \end{theorem}

%
%

 
\subsection{Convergence Analysis of Algorithm \ref{alg:seqn}} In this section, we investigate the global convergence properties of Algorithm \ref{alg:seqn} in detail. Our results primarily rely on the abstract bound derived in Lemma \ref{lemma:prox-descent} 
and are inspired by similar strategies used in the analysis of nonconvex stochastic optimization methods, see, e.g., \cite{GhaLan13,XuYin15,GhaLanZha16,WanMaGolLiu17}. 

In particular, following \cite{GhaLanZha16,XuYin15}, we will first assume that the variance of the stochastic gradients can be controlled and decreases in a suitable way. (This can be achieved, e.g., by aggregating stochastic information or via utilizing so-called mini-batching schemes). Possible extensions and generalization are discussed later in section \ref{sec:var-alpha}.

The stochasticity in Algorithm \ref{alg:seqn} mainly results from the random selection of the oracles $v^k$ and $v^k_+$ in step \ref{alg:seqn-2} and \ref{alg:seqn-3}. In this work, we assume that the stochastic approximations $v^k$ and $v^k_+$ correspond to realizations of the random vectors ${\sf V}^k : \Omega \to \Rn$ and ${\sf V}^k_+ : \Omega \to \Rn$. Furthermore, we suppose that the underlying probability space $(\Omega,\mathcal F,\Prob)$ is sufficiently rich allowing us to model the involved random processes in a unified way. We now define the filtrations
\[ \mathcal F^k := \sigma({\sf V}^0, {\sf V}^0_+, ..., {\sf V}^k), \quad \mathcal F^k_+ := \sigma(\mathcal F^k \cup \sigma({\sf V}^k_+)). \]
By convention, let $({\sf D}^k)_k$ denote the stochastic process associated with the directions $(d^k)_k$ chosen in step \ref{alg:seqn-2}. We will work with the following stochastic assumptions. 

\begin{assumption} \label{ass:stoch-1} We assume: 
\setlength{\leftmargini}{7ex}
\begin{itemize}
\item[{\rm(B.1)}] The mapping ${\sf D}^k : \Omega \to \R^{n}$ is an $\mathcal F^k$-measurable function for all $k$. \vspace{.5ex} 
\item[{\rm(B.2)}] There is $\nu_k > 0$ such that we have $\Exp[\|{\sf D}^k\|^2 \mid \mathcal F^{k-1}_+] \leq \nu_k^2 \cdot \Exp[\|F^{\Lambda_k}_{{\sf V}^k}({\sf X}^k)\|^2 \mid \mathcal F^{k-1}_+]$ \aev and for all $k \in \N$.
\item[{\rm(B.3)}] For all $k \in \N$, it holds $\Exp[{\sf V}^k \mid \mathcal F^{k-1}_+] = \nabla f({\sf X}^k)$, $\Exp[{\sf V}^k_+ \mid \mathcal F^{k}] = \nabla f({\sf Z}^k)$ \aev and there exists $\sigma_k, \sigma_{k,+} > 0$ such that \aev 
\end{itemize}
\[ \Exp[\|\nabla f({\sf X}^k) - {\sf V}^k\|^2 \mid \mathcal F_+^{k-1}] \leq \sigma_k^2 \quad \text{and} \quad \Exp[\|\nabla f({\sf Z}^k) - {\sf V}_+^k\|^2 \mid \mathcal F^{k}] \leq \sigma_{k,+}^2. \]
\end{assumption}

The conditions on the oracles ${\sf V}^k$ and ${\sf V}_+^k$ in (B.3) are common in stochastic optimization, see \cite{GhaLan13,XuYin15,ByrHanNocSin16,WanMaGolLiu17,BolByrNoc18,BotCurNoc18}. The second assumption requires the chosen directions $(d^k)_k$ and the stochastic process $({\sf D}^k)_k$ to be related to the stochastic nonsmooth residual $F^{\Lambda_k}_{v^k}(x^k)$, $k \in \N$. A similar, deterministic variant of this condition is also utilized in the analysis of FBE-based algorithms in \cite{SteThePat17,TheStePat18}. In section \ref{section:direction}, we present several specific examples for choice of the directions $(d^k)_k$ and we construct a family of stochastic quasi-Newton-type directions that satisfy the conditions stated in (B.1) and (B.2). Under assumption (B.1), the design of Algorithm \ref{alg:seqn} implies that the processes $({\sf X}^k)_{k}$ and $({\sf Z}^k)_{k}$ are adapted to the filtrations $\mathcal F^k$ and ${\mathcal F}^k_+$, i.e., we have
\be \label{eq:prop-stoch} {\sf Z}^k \in \mathcal F^k \quad \text{and} \quad {\sf X}^{k+1} \in {\mathcal F}^{k}_+ , \quad \forall~k \geq 0. \ee

We now present one of our main convergence results of this section.

\begin{theorem} \label{theorem:conv-gen} Let the random process $({\sf X}^k)_k$ be generated by Algorithm \ref{alg:seqn} using parameter matrices of the form $\Lambda_k = \lambda_k^{-1} I$ and $\Lambda_{k,+} = \lambda_{k,+}^{-1} I$. Suppose that the assumptions {\rm(A.1)}--{\rm(A.2)} and {\rm(B.1)}--{\rm(B.3)} are satisfied and let us assume that the step sizes $(\lambda_k)_k$, $(\lambda_{k,+})_k$, $(\alpha_k)_k$, and $(\beta_k)_k$ are chosen as follows:   
\be \label{eq:choice-step-1} \lambda_{k,+} \leq \frac{1}{L_f}, \quad \lambda_k \leq \frac{(1-\bar{\rho})\lambda_{k,+}}{1+\nu_k^2 (\alpha_k + L_f \beta_k \lambda_{k,+})^2}, \ee
for all $k \in \N$ and some $\bar{\rho} \in (0,1)$. Then, under the additional conditions 
\be \label{eq:choice-step-2} \sum \lambda_k = \infty, \quad \sum \lambda_k \sigma_k^2 < \infty, \quad \sum \lambda_{k,+} \sigma_{k,+}^2 < \infty \ee
it follows $\liminf_{k \to \infty} \Exp[\|F^I({\sf X}^k)\|^2] = 0$ and $\liminf_{k \to \infty} \| F^I({\sf X}^k) \| = 0$ \as and $(\psi({\sf X}^k))_k$ \as converges to a random variable $\sf Y^*$ with $\lim_{k\to \infty} \Exp[\psi({\sf X}^k)] = \Exp[{\sf Y}^*]$.  
\end{theorem}

\begin{proof} First, due to the law of total expectation and $\mathcal F^{k-1}_+ \subset \mathcal F^k$ and ${\sf X}^k,{\sf D}^k, {\sf Z}^k \in \mathcal F^k$, it follows
\begin{align*} \Exp[\iprod{\nabla f({\sf Z}^k) - {\sf V}^k_+}{\lambda_{k,+}(\nabla f({\sf X}^k) - \nabla f({\sf Z}^k)) +\alpha_k {\sf D}^k} \mid \mathcal F^{k-1}_+] & \\ & \hspace{-55ex} = \Exp[\iprod{\Exp[\nabla f({\sf Z}^k) - {\sf V}^k_+ \mid \mathcal F^k]}{\lambda_{k,+}(\nabla f({\sf X}^k) - \nabla f({\sf Z}^k)) +\alpha_k {\sf D}^k} \mid \mathcal F^{k-1}_+] = 0, \end{align*}
$\text{a.e.}$, where we used the fact that ${\sf V}^k_+$ is an unbiased estimator of the gradient $\nabla f({\sf Z}^k)$. Next, taking conditional expectation in Lemma \ref{lemma:prox-descent}, applying assumptions (B.2) and (B.3), and choosing $\rho = \lambda_k^{-1}$, we obtain
\begin{align*} \Exp[\psi({\sf X}^{k+1}) \mid \mathcal F^{k-1}_+] - \psi({\sf X}^k) & \\ & \hspace{-24ex} \leq  \frac{\lambda_{k,+}\sigma_{k,+}^2}{2} + \half  \left[ L_f - \frac{1}{\lambda_{k,+}} \right] \Exp [\|{\sf X}^{k+1} - {\sf X}^k\|^2\mid \mathcal F^{k-1}_+ ] - \frac{1}{2\lambda_k} \|F^{\Lambda_k}({\sf X}^k)\|^2 \\ & \hspace{-20ex} + \frac{\lambda_k \sigma_k^2}{2}  + \half \left[ \frac{1}{\lambda_{k,+}} - \frac{1}{\lambda_k} + \nu_k^2 \cdot \ell(\lambda_{k,+},\alpha_k,\beta_k) \right] \Exp[\|F^{\Lambda_k}_{{\sf V}^k}({\sf X}^k)\|^2 \mid \mathcal F^{k-1}_+] 
\end{align*}
almost everywhere. Utilizing the bounds stated in \eqref{eq:choice-step-1}, we have
\begin{align} \nonumber \Exp[\psi({\sf X}^{k+1}) \mid \mathcal F^{k-1}_+] - \psi({\sf X}^k) & \\  \label{eq:est-martingale} & \hspace{-20ex} \leq \frac{\lambda_k \sigma_k^2}{2} + \frac{\lambda_{k,+}\sigma_{k,+}^2}{2} - \frac{1}{2\lambda_k} \|F^{\Lambda_k}({\sf X}^k)\|^2 - \frac{\bar{\rho}}{2\lambda_k} \Exp[ \|F^{\Lambda_k}_{{\sf V}^k}({\sf X}^k)\|^2 \mid \mathcal F^{k-1}_+ ].  \end{align}
Now, by assumption (A.2) there exists $\psi^* \in \R$ such that $\psi(x) \geq \psi^*$ for all $x \in \Rn$. Taking expectation and summing the latter inequality, we get
\begin{align*} \sum_{k= 0}^R \frac{\Exp[\|F^{\Lambda_k}({\sf X}^k)\|^2 + \bar{\rho} \cdot \|F^{\Lambda_k}_{{\sf V}^k}({\sf X}^k)\|^2]}{\lambda_k} \leq 2 [\psi(x^0) - \psi^*] + \sum_{k = 0}^R \lambda_k \sigma_k^2 + \lambda_{k,+}\sigma_{k,+}^2 \end{align*}
for all $R \in \N$. Due to $L_f \geq 1$, it holds $\lambda_{k,+} \leq 1$, $\lambda_k \leq 1$ and by \cite[Lemma 2]{Nes13}, we know that the mapping $\delta \mapsto {\delta}^{-1} \|F^{\frac{1}{\delta}I}(x)\|$ is a decreasing function of $\delta$. This yields $\Exp[\|F^I({\sf X}^k)\|^2] \leq  \lambda_k^{-2} \Exp[ \|F^{\Lambda_k}({\sf X}^k)\|^2]$ for all $k \in \N$ and thus, we can infer  $\sum \lambda_k \Exp[\|F^I({\sf X}^k)\|^2] < \infty$. Hence, the first statement in Theorem \ref{theorem:conv-gen} follows from the assumption $\sum \lambda_k = \infty$. Furthermore, applying the Borel-Cantelli lemma, it holds that $\sum \lambda_k \|F^I({\sf X}^k)\|^2 < \infty$ a.s., which finally implies $\liminf_{k\to \infty} \| F^I({\sf X}^k) \| = 0$ with probability 1. The third claim in Theorem \ref{theorem:conv-gen} is a direct consequence of assumption (A.2), \eqref{eq:est-martingale}, and Theorem \ref{theorem:conv-superm}. In order to establish convergence of $(\Exp[\psi({\sf X}^k)])_k$, we first notice that we have 
\begin{align} \nonumber \psi({\sf X}^{k+1}) & \leq \psi({\sf X}^k) + \frac{\lambda_k}{2} \|\nabla f({\sf X}^k) - {\sf V}^k\|^2  \\ÃÂÃÂ \label{eq:des-later}  & \hspace{3ex}  +\lambda_{k,+} \|\nabla f({\sf Z}^k) - {\sf V}_+^k\|^2 + (\nu_k^2\lambda_k)^{-1}\|{\sf D}^k\|^2, \end{align}
(almost) surely and for all $k \in \N$, which follows from Lemma \ref{lemma:prox-descent}, $\iprod{a}{b} \leq \half \|a\|^2 + \half \|b\|^2$, $a,b \in \Rn$, $\lambda_{k,+}^{-1} \leq \lambda_k^{-1}$, and $\ell(\lambda_{k,+},\alpha_k,\beta_k) < \nu_k^{-2}(\lambda_k^{-1} - \lambda_{k,+}^{-1}) \leq (\nu_k^2 \lambda_k)^{-1}$. By induction, this shows that the nonnegative random variable $\psi({\sf X}^k) - \psi^*$ is (almost) surely dominated by 
\[ {\sf U} := \psi(x^0) - \psi^* + \sum_{k=0}^\infty  \frac{\lambda_k}{2} \|\nabla f({\sf X}^k) - {\sf V}^k\|^2 + \lambda_{k,+} \|\nabla f({\sf Z}^k) - {\sf V}_+^k\|^2 + \frac{\|{\sf D}^k\|^2}{\nu_k^2\lambda_k}   \]
for all $k \in \N$. Moreover due to Fatou's lemma, we have $\Exp[{\sf U}] < \infty$. Hence, our last claim is a consequence of the dominated convergence theorem. 
\end{proof}

\begin{remark} \label{remark:ostrowski} Suppose that the assumptions stated in Theorem \ref{theorem:conv-gen} are fulfilled and it holds $\lambda_{k,+} \leq (1-\bar{\rho})L_f^{-1}$ for all $k$ and for $\bar{\rho} \in (0,1)$. Then, the following additional summability condition is satisfied: 
\[ \sum  \Exp[ \|{\sf X}^{k+1} - {\sf X}^k\|^2 ] \leq \sum \lambda_{k,+}^{-1} \Exp[\|{\sf X}^{k+1} - {\sf X}^k\|^2]  < \infty, \]
which implies $\|{\sf X}^{k+1} - {\sf X}^k\| \to 0$ a.s. This property can be interpreted as a probabilistic Ostrowski condition. 
\end{remark}

Next, we show that the liminf-convergence of $(\Exp[\|F^I({\sf X}^k)\|^2])_k$ and $(F^I({\sf X}^k))_k$ can be strengthened to strong convergence. Our result is motivated by \cite[Theorem 2.6]{WanMaGolLiu17}, where strong almost sure convergence is shown for a stochastic quasi-Newton-type method for nonconvex smooth problems. We note that the proof in \cite{WanMaGolLiu17} is based on a boundedness assumption on the gradient estimates that is not required here.  

\begin{theorem} \label{theorem:gen-conv-special} Suppose that the conditions stated in Theorem \ref{theorem:conv-gen} are satisfied with $\bar{\rho} \in (0,1)$ and $\lambda_{k,+} \leq (1-\bar{\rho})L_f^{-1}$ and assume that there exists $\underline{\rho} \in (0,1)$ such that 
\be \label{eq:choice-step-3}  \underline{\rho}\lambda_{k,+} \leq \lambda_k, \quad \forall~k. \ee
 Then, it holds that $\lim_{k \to \infty} \Exp[ \| F^I({\sf X}^k) \|] = 0$ and $\lim_{k\to\infty} F^I({\sf X}^k) = 0$ a.s.
\end{theorem}

\begin{proof} We first verify the strong convergence of the sequence $(\|F^I({\sf X}^k)\|)_k$. According to Remark \ref{remark:ostrowski} and using Borel-Cantelli, we obtain $\sum \lambda_{k,+}^{-1} \|{\sf X}^{k+1} - {\sf X}^k\|^2 < \infty$ almost surely. 
%
%
Let us define the events $S_1 = \{ \omega: \sum \lambda_{k,+}^{-1} \|{\sf X}^{k+1}(\omega) - {\sf X}^k(\omega)\|^2 < \infty \}$, $S_2 = \{\omega: \sum_k \lambda_k \|F^I({\sf X}^k(\omega))\|^2 < \infty \}$, and  
\[ S_3 = \{ \omega: (\|F^{I}({\sf X}^k(\omega))\|)_k \, \text{does not converge}\} \] 
and let us assume that the claim $\lim_k F^I({\sf X}^k) = 0$, a.s., does not hold, i.e., we have $\Prob(S_3) > 0$. Theorem \ref{theorem:conv-gen} and our last steps ensure $\Prob(S_1) = \Prob(S_2) = 1$ and thus, it follows $\Prob(S_1 \cap S_2 \cap S_3) > 0$. We now consider an arbitrary sample $\omega \in S_1 \cap S_2 \cap S_3$ with associated realizations $(x^k)_k \equiv ({\sf X}^k(\omega))_k$, etc. Following the strategies used in the convergence analysis of classical trust region-type methods, see, e.g., \cite[Theorem 6.4.6]{ConGouToi00}, the condition $\omega \in S_2 \cap S_3$ implies that there exists $\veps > 0$ and infinite, increasing sequences $(t_i)_i$ and $(\ell_i)_i$ such that $\ell_i > t_i$ for all $i \in \N_0$ and
\[ \|F^I(x^{t_i}) \| \geq 2\veps, \quad \| F^I(x^{\ell_i}) \| < \veps, \quad \text{and} \quad \| F^I(x^k) \| \geq \veps, \quad \;\;~k = t_i+1,...,\ell_i - 1. \]
(Note that $(t_i)_i$ and $(\ell_i)_i$ depend on $\omega$). Thus, due to $\omega \in S_2$, it follows 
\[ \infty > \sum_{k=0}^\infty \lambda_k \|F^I(x^k) \|^2 \geq \sum_{i=0}^\infty \sum_{k = t_i}^{\ell_i - 1} \lambda_k \|F^I(x^k) \|^2 \geq \veps^2 \sum_{i=0}^\infty \sum_{k= t_i}^{\ell_i - 1} \lambda_k \]
and consequently, setting $\gamma_i := \sum_{k= t_i}^{\ell_i - 1} \lambda_k$, $i \in \N_0$, we have $\gamma_i \to 0$ as $i \to \infty$. 
%
%
Applying H\"older's inequality and \eqref{eq:choice-step-3}, this yields
\begin{align*} \|x^{\ell_i} - x^{t_i}\| \leq \sum_{k=t_i}^{\ell_i-1} \frac{\sqrt{\lambda_k}}{\sqrt{\underline{\rho}\lambda_{k,+}}} \|x^{k+1} - x^k\| \leq \frac{\sqrt{\gamma_i}}{\sqrt{\underline{\rho}}} \left [ \sum_{k=t_i}^\infty \frac{1}{\lambda_{k,+}} \|x^{k+1} - x^k\|^2 \right]^\half. \end{align*}
Due to $\omega \in S_1$, the sum appearing on the right side of the latter estimate converges and hence, it follows $\|{x}^{\ell_i} - {x}^{t_i} \|  \to 0$. Finally, since $x \mapsto F^I(x)$ is Lipschitz continuous with constant $L_F := 2+L_f$, it holds that
\begin{align*} \veps < |\|F^I(x^{\ell_i}) \| - \|F^I(x^{t_i})\||  \leq  \| F^I({x}^{\ell_i}) - F^I({x}^{t_i})\| \leq L_F \|{x}^{\ell_i} - {x}^{t_i}\| \to 0, \end{align*}
as $i \to \infty$, which is a contradiction. Since the sample $\omega$ was chosen arbitrarily, this implies $\Prob(S_1 \cap S_2 \cap S_3) = 0$ and $\Prob(S_3) = 0$ and consequently, $(F^I({\sf X}^k))_k$ converges \as to zero. The strong convergence in mean can be shown in a similar way by taking expectation in the latter calculations and using Jensen's inequality, see Theorem \ref{theorem:complex}. 
\end{proof}


The different conditions on the step sizes and parameters in \eqref{eq:choice-step-1}, \eqref{eq:choice-step-2}, and \eqref{eq:choice-step-3} require $\lambda_k$ and $\lambda_{k,+}$ and the noise levels $\sigma_k$ and $\sigma_{k,+}$ to be of the same or of a similar order. Furthermore, the step sizes $\alpha_k$ and $\beta_k$ should be chosen to balance the possibly unbounded parameters $\nu_k$, $k \in \N$, in order to guarantee the non-summability condition $\sum \lambda_k = \infty$. 

\begin{remark} \label{remark:unbiased-1} Reconsidering the proof of Theorem \ref{theorem:conv-gen}, convergence can also established under the weaker stochastic conditions $\Exp[\|{\sf D}^k\|^2] \leq \nu_k^2 \cdot \Exp[\|F^{\Lambda_k}_{{\sf V}^k}({\sf X}^k)\|^2]$, and 
\[ \Exp[\|\nabla f({\sf X}^k) - {\sf V}^k\|^2] \leq \sigma_k^2, \quad \Exp[\|\nabla f({\sf Z}^k) - {\sf V}^k_+\|^2] \leq \sigma_{k,+}^2, \]
if the bound for the step size $\lambda_k$ is slightly adjusted. (See, e.g., \eqref{eq:des-later} for comparison). Although, we can not apply the martingale-type result in Theorem \ref{theorem:conv-superm} in this case, we can still guarantee almost sure convergence of $\|F^{I}({\sf X}^k)\|$ and convergence in mean. These weaker conditions allow the oracle ${\sf V}^k_+$ to be a \textit{biased} estimator of $\nabla f{({\sf Z}^k)}$. In particular, ${\sf V}^k$ and ${\sf V}^k_+$ can be dependent random variables.  
\end{remark}


\subsection{Convergence Under Weaker Conditions} \label{sec:var-alpha}

In this section, we show that global convergence of Algorithm \ref{alg:seqn} can be ensured under weaker assumptions on the variances $(\sigma_k)_k$ and $(\sigma_{k,+})_k$. Specifically, convergence can be established if the variances are simply bounded and the step size $\alpha_k$ is chosen to depend on $\lambda_{k,+}$ and decreases. 

This observation is based on two recent works by Davis and Drusvyatskiy \cite{DavDru18-1,DavDru18-2} and we will verify that the techniques developed in \cite{DavDru18-1,DavDru18-2} can also be transferred and generalized to our algorithmic framework. Let us notice that the theoretical results in \cite{DavDru18-1,DavDru18-2} are not immediately applicable since the direction ${\sf V}_+^k - \alpha_k\lambda_{k,+}^{-1}{\sf D}^k$ is typically not an unbiased estimator of $\nabla f({\sf Z}^k)$ (or $\nabla f({\sf X}^k)$). 
In contrast to the analysis in the last section, we will see that the parameters $\alpha_k$ and $\beta_k$ now need to be chosen sufficiently small in order to balance the possible bias. 



As in \cite{DavDru18-1,DavDru18-2}, the main idea is to study the approximate descent properties of the random process $({\sf X}^k)_k$ using the Moreau envelope $\envt{\Theta}{\psi}$,  $\Theta := \theta^{-1} I$, instead of the objective function $\psi$. Under assumption (A.1), $\psi$ is $\theta^{-1}$-weakly convex for $\theta < L_f^{-1}$, i.e., the mapping $\psi + \frac{1}{2\theta}\|\cdot\|^2$ is (strongly) convex, proper, and lower semicontinuous for all $\theta < L_f^{-1}$. In this case, the Moreau envelope is a well-defined and continuously differentiable function with gradient $\nabla \envt{\Theta}{\psi}(x) = \Theta(x - \proxt{\Theta}{\psi}(x))$ and as shown in \cite{Don09,DruLew18}, we have the following connection to the 
residual $F^{\Theta}(x)$:
\be \label{prop:gradient-relation}
(1-L_f\theta) \|F^{\Theta}(x)\| \leq   \|\nabla \envt{\Theta}{\psi}(x)\| \leq (1+L_f\theta) \|F^{\Theta}(x)\|, \quad \forall~x.
\ee
Furthermore, setting $\bar x^k = \proxt{\Theta}{\psi}(x^k)$, the definition of the Moreau envelope allows us to derive a fundamental descent-type inequality:
\be \label{prop:env-descent} \envt{\Theta}{\psi}(x^{k+1})  \leq \envt{\Theta}{\psi}(x^k) + \frac{1}{2\theta} \left[ \|\bar x^k - x^{k+1}\|^2 - \|\bar x^k - x^k\|^2 \right].
\ee
In the next lemma, we first estimate the term $\|\bar x^k - x^{k+1}\|^2$ in \eqref{prop:env-descent}. 

\begin{lemma}\label{lemma:point-diff}
Let $x \in \dom~\vp$, $d \in \Rn$, $\Lambda_+ = \lambda_+^{-1} I$, $\Theta=\theta^{-1}I$, $\alpha, \beta>0$, $\rho_1,\rho_2 >0$  and the approximations $v, v_+ \in \Rn$ be given. Under assumption {\rm(A.1)} it holds that
\begin{align*}
\| p_+ - \bar x \|^2  &\leq \left[ (1+\rho_1)\tau_+^2 + 2L_f \lambda_{+}\tau_+ +(1+\rho_2)L_f^2\lambda_{+}^2 \right] \|\bar x - x\|^2 \\ 
& \hspace{2ex} + \left[ 1+\frac{1}{\rho_{1}}+\frac{1}{\rho_{2}} \right]  \mu^2\|d\|^2
+2 \lambda_{+} \iprod{ \tau_+[\bar x - x] -p}{v_+ - \nabla f(x+\beta d)} \\
&\hspace{2ex} + \lambda_{+}^2 \|\nabla f(x+\beta d) - v_+\|^2, 
\end{align*}
where $\bar x = \proxt{\Theta}{\psi}(x)$ , $p_+ := \proxt{\Lambda_+}{\vp}(x+\alpha d - \Lambda^{-1}_+v_+)$, $p := \Lambda_{+}^{-1}[\nabla f(\bar x) - \nabla f(x+\beta d)] + \alpha d $,  $\mu :=\alpha + L_f \beta\lambda_{+}$, and $\tau_+ := 1- \lambda_+\theta^{-1}$.
\end{lemma}

A detailed proof of Lemma \ref{lemma:point-diff} can be found in the appendix in section \ref{sec:app-A1}. We now present the main convergence result of this section.
\begin{theorem} \label{theorem:conv-gen2} Let the random process $({\sf X}^k)_k$ be generated by Algorithm \ref{alg:seqn} using parameter matrices of the form $\Lambda_k = \lambda_k^{-1} I$ and $\Lambda_{k,+} = \lambda_{k,+}^{-1} I$. Suppose that the assumptions {\rm(A.1)}--{\rm(A.2)} and {\rm(B.1)}--{\rm(B.3)} are fulfilled and let us assume that the step sizes $(\lambda_k)_k$, $(\lambda_{k,+})_k$, $(\alpha_k)_k$, and $(\beta_k)_k$ and $\theta$ satisfy the following conditions:   
\be \label{eq:choice-step-21} \theta = \frac{1}{3L_f}, \quad \lambda_{k,+} \leq \frac{1}{6L_f},  \quad \beta_k\nu_k \leq \frac{1}{9L_f}, \quad \alpha_k = L_f\beta_k\lambda_{k,+}, \ee
for all $k \in \N$. Then, under the additional conditions 
\be \label{eq:choice-step-22} \sum \lambda_{k,+} = \infty, \quad \sum \lambda_{k,+} \lambda_k^2 \sigma_k^2 < \infty, \quad \sum \lambda^2_{k,+} \sigma_{k,+}^2 < \infty ,\ee
it follows $\liminf_{k \to \infty} \Exp[\|F^I({\sf X}^k)\|^2] = 0$, $\liminf_{k \to \infty} \| F^I({\sf X}^k) \| = 0$ \as 
Suppose that the output variable ${\sf X} = {\sf X}^{k_*}$ is sampled from the iterates $\{{\sf X}^k : k \in [T]_0\}$ with probability  $\Prob_T(k_* = k) = \frac{\lambda_{k,+}}{\sum_{i=0}^T \lambda_{i,+}}$. Then, we have
\be \label{eq:res-alpha-des}
\Exp[\|F^{I}({\sf X})\|^2] \leq  54 \cdot  \frac{\envt{\Theta}{\psi}(x^{0}) -\psi^* +  \sum_{k=0}^{T} [\frac{\lambda_{k,+}\lambda_k^2\sigma_k^2}{4} + \frac{3L_f\lambda_{k,+}^2\sigma_{k,+}^2}{2}]}{\sum_{k=0}^{T} \lambda_{k,+}}.
\ee
\end{theorem}

\begin{proof} Let us set $\bar{\sf X}^k := \proxt{\Theta}{\psi}({\sf X}^k)$ and $\tau_{k,+} := 1 - \lambda_{k,+}\theta^{-1}$. As before, using the law of total expectation, ${\sf D}^k, {\sf Z}^k \in \mathcal F^k$, ${\sf X}^k,\bar {\sf X}^k \in \mathcal F_+^{k-1}$, and (B.3), it follows
\begin{align*}
&\Exp [\iprod{\alpha_k {\sf D}^k -\Lambda_{k,+}^{-1}[\nabla f(\bar {\sf X}^k) - \nabla f({\sf Z}^k)]}{{\sf V}^k_+ - \nabla f({\sf Z}^k)}\mid \mathcal F^{k-1}_+ ]  \\ & \hspace{20ex} +  \tau_{k,+} \cdot \Exp [\iprod{\bar {\sf X}^k -{\sf X}^k}{{\sf V}^k_+ - \nabla f({\sf Z}^k)}\mid \mathcal F^{k-1}_+ ] = 0  
\end{align*}
almost everywhere. 
Next, utilizing the assumptions (B.2)--(B.3) and $(a+b)^2 \leq 2(a^2 + b^2)$, $a,b \in \R$, we obtain 
\begin{align} \nonumber
\Exp [\|{\sf D}^k \|^2 \mid \mathcal F^{k-1}_+ ] & \leq \nu_k^2 \cdot \Exp[\|F^{\Lambda_k}_{{\sf V}^k}({\sf X}^k)+F^{\Lambda_k}({\sf X}^k)-F^{\Lambda_k}({\sf X}^k)\|^2\mid \mathcal F^{k-1}_+ ] \\ 
&\leq 2 \nu_k^2 \cdot [ \Exp[\|F^{\Lambda_k}_{{\sf V}^k}({\sf X}^k) -F^{\Lambda_k}({\sf X}^k)\|^2\mid \mathcal F^{k-1}_+] +  \|F^{\Lambda_k}({\sf X}^k)\|^2 ] \nonumber \\ 
& \leq 2\nu_k^2 \cdot [ \lambda_k^2\sigma_k^2 +  \|F^{\Lambda_k}({\sf X}^k)\|^2]. \label{prop:upper-d}
\end{align}
We now apply Lemma \ref{lemma:point-diff} with $\rho_1 \equiv \rho_{k,1} = L_f \lambda_{k,+}$,  $\mu \equiv \mu_k = \alpha_k + L_f \beta_k \lambda_{k,+} = 2 L_f \beta_k \lambda_{k,+}$, $\rho_2 = 1$. Taking conditional expectation in Lemma \ref{lemma:point-diff}, it holds that
\begin{align*} \Exp[\|\bar{\sf X}^k - {\sf X}^{k+1}\|^2 \mid \mathcal F^{k-1}_+] & \\ &\hspace{-15ex} \leq c_{k,1} \|\bar {\sf X}^k - {\sf X}^k\|^2  + c_{k,2} [ \lambda_k^2\sigma_k^2 +  \|F^{\Lambda_k}({\sf X}^k)\|^2] + \lambda_{k,+}^2\sigma_{k,+}^2, \end{align*}
where $c_{k,1} = (1+\rho_{k,1})\tau_{k,+}^2 + 2L_f \lambda_{k,+}\tau_{k,+} + 2L_f^2\lambda_{k,+}^2$ and $c_{k,2} = 2 \nu_k^2 \mu_k^2 (2 + \rho_{k,1}^{-1})$. Moreover, the choice of the parameters and \eqref{eq:choice-step-21} imply
\[ c_{k,1} - 1 = L_f\lambda_{k,+} [-3 -L_f \lambda_{k,+} + 9L_f^2 \lambda_{k,+}^2] < 0.\] 
%
Thus, taking expectation in \eqref{prop:env-descent} and using $ \theta\|F^{I}(x)\| \leq \|F^{\Theta}(x)\|$ and $\|F^{\Lambda_k}(x)\| \leq \|F^{I}(x)\|$ (see again \cite{Nes13}) and \eqref{prop:gradient-relation}, this yields  
\begin{align*}
&\Exp[ \envt{\Theta}{\psi}({\sf X}^{k+1}) \mid \mathcal F^{k-1}_+ ] -\envt{\Theta}{\psi}({\sf X}^k) \\
&\hspace{2ex} \leq \frac{1}{2\theta} \left[ (c_{k,1}-1)  \|\bar {\sf X}^k - {\sf X}^k\|^2 + c_{k,2} [ \lambda_k^2\sigma_k^2 + \|F^{\Lambda_k}({\sf X}^k)\|^2] +\lambda_{k,+}^2 \sigma_{k,+}^2\right] \\ & \hspace{2ex} \leq \frac{1}{2\theta} \left[ \left[ (1-L_f\theta)^2\theta^2 (c_{k,1}-1) + c_{k,2} \right] \|F^I({\sf X}^k)\|^2 + c_{k,2} \lambda_k^2\sigma_k^2 +  \lambda_{k,+}^2\sigma_{k,+}^2 \right].
\end{align*}
Together with $c_{k,2} = 8L_f \beta_k^2\nu_k^2 \lambda_{k,+}(2L_f \lambda_{k,+}+1) \leq \frac{4\lambda_{k,+}}{81L_f} (4L_f \lambda_{k,+}+2)$, we have
\[ \frac{4}{81L_f^2} (c_{k,1}-1) + c_{2,k} \leq \frac{4\lambda_{k,+}}{81L_f} [ - 1  + 3 L_f \lambda_{k,+} + 9 L_f^2 \lambda_{k,+}^2 ] \leq -\frac{\lambda_{k,+}}{81L_f}. \]
Combining the latter estimate and $(2\theta)^{-1}c_{k,2} \leq \frac{\lambda_{k,+}}{4}$, we finally obtain

\begin{align*}
\Exp [ \envt{\Theta}{\psi}({\sf X}^{k+1}) \mid \mathcal F^{k-1}_+ ] -\envt{\Theta}{\psi}({\sf X}^k)  & \\ & \hspace{-15ex} \leq -\frac{\lambda_{k,+}}{54}\|F^{I}({\sf X}^k)\|^2+ \frac14 \lambda_{k,+}\lambda_k^2\sigma_k^2 + \frac{3L_f}{2}\lambda_{k,+}^2 \sigma_{k,+}^2.
\end{align*}
Taking expectation, summing the inequality, and using assumption (A.2), we have 
\begin{align*}
 \sum_{k=0}^{T} \lambda_{k,+}\Exp[\|F^{I}({\sf X}^k)\|^2]\leq 54  \left[ \envt{\Theta}{\psi}(x^{0}) - \psi^*+ \sum_{k=0}^{T} \frac{\lambda_{k,+}\lambda_k^2\sigma_k^2}{4}+ \frac{3L_f\lambda_{k,+}^2\sigma_{k,+}^2}{2} \right] 
\end{align*}
for all $T \in \N$. As in the proof of Theorem \ref{theorem:conv-gen}, this establishes the first and second statement in Theorem \ref{theorem:conv-gen2}. 
In addition, if the output variable ${\sf X} = {\sf X}^{k_*}$ is sampled from the iterates $\{{\sf X}^k : k \in [T]_0 \}$ with probability $\Prob_T(k = k_*) = \frac{\lambda_{k,+}}{\sum_{i=0}^T \lambda_{i,+}}$, we get
\begin{align*}
\Exp[\|F^{I}({\sf X})\|^2] = \left[{\sum_{k=0}^{T} \lambda_{k,+}} \right]^{-1} \sum_{k=0}^{T} \lambda_{k,+}\Exp[\|F^{I}({\sf X}^k)\|^2].
\end{align*}
The complexity bound \eqref{eq:res-alpha-des} now follows easily from the last two results.
\end{proof}

\subsection{The Stochastic Extragradient Method with Variance Reduction} \label{section:extra-vr}
In this section, motivated by the general success and recent algorithmic and theoretical advancements of SVRG-type stochastic optimization methods, see, e.g., \cite{AllHaz16,RedHefSraPocSmo16,RedSraPocSmo16,ZhaRedSra16,WanMaGolLiu17,LiuSoWu18,PooLiaSch18}, we investigate the convergence of a variant of the stochastic extra-step scheme using variance reduction. The approach is designed for empirical risk minimization problems, where $f$ takes the form
\[ f(x) = \frac{1}{N} \sum_{i=1}^N f_i(x), \quad x \in \Rn, \quad N \in \N. \]
Following the variance reduction strategies developed in \cite{JohZha13,XiaZha14,AllHaz16,RedHefSraPocSmo16,RedSraPocSmo16}, we utilize mini-batch-type stochastic gradients $\nabla f_{\mathcal S}(x) := |\mathcal S|^{-1} \sum_{i \in \mathcal S} \nabla f_i(x)$ and periodically updated full gradient information to build a stochastic oracle of $\nabla f$. Here, $\mathcal S$ is a randomly chosen subset $\mathcal S \subseteq [N]$. The complete method is shown in Algorithm \ref{alg:inex-qnm}.

\begin{algorithm2e}[t]
\caption{\textbf{S}toch. \textbf{E}xtra-Step \textbf{Q}uasi-\textbf{N}ewton Method with {\bf V}ar. {\bf R}eduction}    
\label{alg:inex-qnm}
\lnlset{alg:ssn-1}{1}{Initialization: ~~Choose $x^0 \in \dom~\vp$ and select $(\Lambda^m_k), (\Lambda^m_{k,+}) \subset \Spp$ and $(K_m) \subset \N$. Choose the step and mini-batch sizes $(\alpha_k^m)$ $(\beta_k^m)$, $(b_m)$, and $(b_{m,+})$ and set $x^{-1}_{K_{-1}} = x^0$.} \\ 
\For{$m = 0,1,..., M$}{
\lnlset{alg:ssn-2}{2}{Update $\tilde x^m = x^{m-1}_{K_{m-1}}$ and $g^m = \nabla f(\tilde x^{m})$ and set $x^m_0 = \tilde x^m$.} \\
\For{$k = 0,...,K_m-1$}{
\lnlset{alg:ssn-3}{3}{Choose a random sample set $\mathcal S_k^m \subset [N]$ with $|\mathcal S_k^m| = b_m$ and compute the oracle $v_k^m = \nabla f_{\mathcal S_k^m}(x^m_k) - \nabla f_{\mathcal S^m_k}(\tilde x^m) + g^m$.} \\
\lnlset{alg:ssn-4}{4}{Select $d^m_k$ and calculate $z^m_k = x^m_k + \beta_k^m d^m_k$ and the new oracle $v^m_{k,+} = \nabla f_{\mathcal S^m_{k,+}}(z^m_k) - \nabla f_{\mathcal S^m_{k,+}}(\tilde x^m) + g^m$ with rate $|\mathcal S^m_{k,+}| = b_{m,+}$.} \\ 
\lnlset{alg:ssn-5}{5}{Perform the update $x^m_{k+1} = \proxt{\Lambda^m_{k,+}}{\vp}(x^m_k + \alpha_k^m d^m_k - (\Lambda^m_{k,+})^{-1}v^m_{k,+})$.} \\
}} 
\end{algorithm2e}

We now discuss the convergence properties of Algorithm \ref{alg:inex-qnm} and of the generated stochastic process of iterates $({\sf X}^m_k)_{k,m}$. In the following, we interpret the sample sets $\mathcal S_k^m$ and $\mathcal S_{k,+}^m$, which are randomly selected  in step \ref{alg:ssn-3} and \ref{alg:ssn-4} of Algorithm \ref{alg:inex-qnm}, as realizations of the collections of random variables
\[ {\sf S}^m_k := \{({\sf S}^m_k)_1,...,({\sf S}^m_k)_{b_m}\} \quad \text{and} \quad {\sf S}_{k,+}^m := \{({\sf S}^m_{k,+})_1,...,({\sf S}^m_{k,+})_{b_{m,+}}\} \] 
with $({\sf S}^m_k)_i: \Omega \to [N]$, $({\sf S}^m_{k,+})_j: \Omega \to [N]$ for all $i \in [b_m]$, $j \in [b_{m,+}]$. We can then define the filtrations
\[ \mathcal F^m_k = \sigma({\sf S}^0_0, {\sf S}^0_{0,+},...,{\sf S}^0_{K_0-1,+}, {\sf S}^1_{0},...,{\sf S}^m_{k-1,+},{\sf S}^m_k) \;\; \text{and} \;\; {\mathcal F}_{k,+}^m = \sigma(\mathcal F_k \cup \sigma({\sf S}^m_{k,+})) \]
with $\mathcal F_{-1,+}^m := \mathcal F^{m-1}_{K_{m-1}-1,+}$. Next, we state a modified version of Assumption \ref{ass:lip} and \ref{ass:stoch-1} characterizing the stochastic behavior of the chosen directions $(d^m_k)_{k,m}$ and their associated stochastic process $({\sf D}^m_k)_{k,m}$.
\begin{assumption} \label{ass:wk} Let $({\sf D}_k^m)_{k,m}$ be generated by Algorithm \ref{alg:inex-qnm}. For all $k,m$, we assume: 
\setlength{\leftmargini}{7ex}
\begin{itemize}
\item[{\rm(C.1)}] Each of the component functions $\nabla f_i$, $i \in [N]$, is $L_i$-Lipschitz continuous. \vspace{.5ex} 
\item[{\rm(C.2)}] The mapping ${\sf D}^m_k : \Omega \to \R^{n}$ is an $\mathcal F^m_k$-measurable function. \vspace{.5ex} 
\item[{\rm(C.3)}] There exists $\nu_k^m > 0$ such that $\Exp[\|{\sf D}^m_k\|^2 \mid \mathcal F_{-1,+}^m] \leq (\nu_k^m)^2 \Exp[\|F^{\Lambda_k^m}_{v_k^m}({\sf X}^m_k)\|^2 \mid  \mathcal F_{-1,+}^m]$ almost everywhere.
\item[{\rm(C.4)}] It holds $\Exp[{\sf V}^m_k \mid \mathcal F^{m}_{k-1,+}] = \nabla f({\sf X}^m_k)$ and $\Exp[{\sf V}^m_{k,+} \mid \mathcal F^{m}_k] = \nabla f({\sf Z}^m_k)$ \aev
\end{itemize}
\end{assumption}

In addition to (A.1) and (C.1), we define the uniform Lipschitz constant $L := \max_{i \in [N]} L_i$. Then, it holds that $L_f \leq \frac{1}{N} \sum_{i} L_i \leq L$. Let us note that condition (C.3) is a slightly stronger variant of assumption (B.2) since a coarser $\sigma$-algebra is used in the conditioning. Furthermore, under (C.2), the structure of Algorithm \ref{alg:inex-qnm} implies that the processes $({\sf X}^m_k)_{k,m}$ and $({\sf Z}^m_k)_{k,m}$ are adapted to $\mathcal F^m_k$ and ${\mathcal F}^m_{k,+}$, i.e., we have
\be \label{eq:prop-proc} {\sf X}^m_0 \in \mathcal F_{-1,+}^m, \quad {\sf X}^m_k \in {\mathcal F}^m_{k-1,+}, \quad {\sf Z}^m_0 \in \mathcal F_0^m, \quad {\sf Z}^m_k \in \mathcal F^m_k, \ee
and $\tilde{\sf X}^m \in  \mathcal F_{-1,+}^m$ for all $k \in [K_m]$ and all $m \in \N_0$. 

We now present our main global convergence results for the case $M \to \infty $. 

\begin{theorem} \label{theorem:complex} Let the stochastic process of iterates $({\sf X}_k^m)_{k,m}$ be generated by Algorithm \ref{alg:inex-qnm} using the parameter matrices $\Lambda_k^m \equiv (\lambda^m_k)^{-1} I$, $\Lambda^m_{k,+} \equiv (\lambda^m_{k,+})^{-1}I $ 
and suppose that the assumptions {\rm(A.1)}--{\rm(A.2)} and {\rm (C.1)}--{\rm(C.4)} are satisfied. Let us further assume that the step sizes are chosen as follows: 
\[ \lambda^m_{k,+} \equiv \lambda^m_+ \leq \frac{1}{2L_f+L \cdot \tau_m}, \quad \lambda^m_k \leq \frac{(1-\bar\rho)\lambda_+^m}{1+(\mu_k^m)^2 + (L\beta_k^m\nu_k^m\lambda_+^m)^2 K_m b_{m,+}^{-1}}, \]
where $\tau_m := K_m(b_{m,+}^{-1} + b_m^{-1})^\half$,  $\mu_k^m := (\alpha_k^m + L_f \beta_k^m \lambda_{+}^m)\nu_k^m$,  and $\bar\rho \in [0,1)$ is a given constant. Then, $(\psi(\tilde{\sf X}^m))_m$ almost surely converges to a random variable ${\sf Y}^*$. Additionally, if there exist $ \underline{\rho} > 0$ and $\kappa > 0$ such that
\be \label{eq:step-choice-vr} \sum_{m=0}^\infty K_m^2 \left[ \sum_{k=0}^{K_m -1} \frac{1}{\lambda^m_k} \right]^{-1} = \infty,  \quad \underline{\rho} \lambda_+^m \leq K_m \left[ \sum_{k=0}^{K_m-1} \frac{1}{\lambda_k^m} \right]^{-1}, \ee
and $K_m\lambda_+^m \leq \kappa$ for all $m \in \N$, we have $ K_m^{-1} \sum_{k=0}^{K_m-1} \|F^I({\sf X}^m_k)\| \to 0$ almost surely and $ K_m^{-1} \sum_{k=0}^{K_m-1} \Exp[\|F^I({\sf X}^m_k)\|] \to 0$ as $m \to \infty$.
\end{theorem}

The proof of Theorem \ref{theorem:complex} combines the theoretical properties of the variance reduction technique discussed in \cite{RedHefSraPocSmo16,RedSraPocSmo16} and the strategies utilized in the proofs of Theorem \ref{theorem:conv-gen} and \ref{theorem:gen-conv-special}. 
We first present a preparatory lemma that was shown by Reddi et $\text{al.}$ in \cite{RedHefSraPocSmo16} and \cite[Lemma 3]{RedSraPocSmo16} and provides an upper bound for the error terms $\|\nabla f(x^m_k)-v^m_k\|^2$ and $\|\nabla f(z^m_k)-v_{k,+}^m\|^2$. Lemma \ref{lemma:sto-err-bound} is primarily a consequence of the measurability properties \eqref{eq:prop-proc} and the tower property of the expectation. 

\begin{lemma} \label{lemma:sto-err-bound} Let $({\sf X}^m_k)_{k,m}$, $({\sf Z}^m_k)_{k,m}$, and $(\tilde{\sf X}^m)_m$ be generated by Algorithm \ref{alg:inex-qnm} and suppose that {\rm(A.1)}, {\rm(C.1)}--{\rm(C.2)}, and {\rm(C.4)} are satisfied. For all $k \in [K_m-1]_0$ and $m \in \N_0$, we have 
\[ \Exp[\|\nabla f({\sf X}^m_k) - {\sf V}^m_k\|^2 \mid \mathcal F^m_{-1,+}] \leq \frac{L^2}{b_m} \Exp[ \| {\sf X}^m_k - \tilde{\sf X}^m \|^2 \mid \mathcal F_{-1,+}^m] \] 
and $\Exp[\|\nabla f({\sf Z}^m_k) - {\sf V}_{k,+}^m\|^2 \mid \mathcal F^m_{-1,+}] \leq \frac{L^2}{b_{m,+}} \Exp[ \| {\sf Z}^m_k - \tilde {\sf X}^m \|^2 \mid\mathcal F_{-1,+}^m]$ \aev
\end{lemma}

We now turn to the proof of Theorem \ref{theorem:complex}.

\begin{proof} We proceed as in the proof of Theorem \ref{theorem:conv-gen} and first derive an approximate descent condition for the iterates ${\sf X}_k^m$. Due to ${\sf X}_k^m,{\sf D}_k^m, {\sf Z}_k^m \in \mathcal F_k^m$, $\mathcal F_{-1,+}^m \subset \mathcal F_k^m$, the law of total expectation, and the unbiasedness of ${\sf V}^m_{k,+}$, it follows
\begin{align*} \Exp[\iprod{\Exp[\nabla f({\sf Z}_k^m) - {\sf V}_{k,+}^m \mid \mathcal F_{k}^m]}{\lambda^m_{+}(\nabla f({\sf X}_k^m) - \nabla f({\sf Z}_k^m)) +\alpha_k^m {\sf D}_k^m} \mid \mathcal F_{-1,+}^m] = 0.  
\end{align*}
%
Taking conditional expectation in Lemma \ref{lemma:prox-descent} and applying Lemma \ref{lemma:sto-err-bound} and $\|{\sf Z}^m_k-\tilde {\sf X}^m\|^2 \leq (1+\frac{1}{\nu})(\beta^m_k)^2 \|{\sf D}^m_k\|^2 + (1+\nu) \|{\sf X}^m_k-\tilde {\sf X}^m\|^2$ (for all $\omega$), we obtain
\begingroup
\allowdisplaybreaks
\begin{align*} \Exp[\psi({\sf X}^m_{k+1}) - \psi({\sf X}^m_k) \mid \mathcal F^m_{-1,+}] & \\  & \hspace{-22ex} \leq   \half  \Exp\left[ {\textstyle\left[ L_f - \frac{1}{\lambda^m_{+}}\right]}  \|{\sf X}^m_{k+1} - {\sf X}^m_k\|^2 + {\textstyle \Theta_k^m }  \|F^{\Lambda^m_k}_{{\sf V}^m_k}({\sf X}^m_k)\|^2 \mid \mathcal F_{-1,+}^m \right] \\ & \hspace{-19ex} + \Exp\left[{\textstyle \frac{L^2}{2} \left[ \frac{(1+\nu)\lambda^m_{+}}{b_{m,+}} + \frac{1}{b_m\rho} \right]} \|{\sf X}^m_k - \tilde {\sf X}^m \|^2  - {\textstyle\frac{1}{2\lambda^m_k}}  \|F^{\Lambda^m_k}({\sf X}^m_k)\|^2 \mid \mathcal F^m_{-1,+}\right ] \\ & \hspace{-19ex} + \frac12 {\textstyle\left[ \rho - \frac{1}{\lambda^m_k} \right]} \Exp[ \|p^{\Lambda^m_k}_{{\sf V}_k^m}({\sf X}_k^m) - p^{\Lambda^m_k}({\sf X}_k^m)\|^2 \mid \mathcal F^m_{-1,+}], 
\end{align*}
\endgroup
almost everywhere, where $\Theta_k^m =  \mathcal L_k^m - {(\lambda^m_k)}^{-1} $,
\[ \mathcal L_k^m := \left[ \ell(\lambda^m_{+},\alpha_k^m,\beta_k^m) + \frac{L^2(\beta^m_k)^2}{b_{m,+}}\left(1+\frac{1}{\nu}\right) \lambda^m_{+} \right] (\nu^m_k)^2 + \frac{1}{\lambda^m_{+}}, \] 
and $k \in [K_m-1]_0$, $m$, and $\nu, \rho > 0$ are arbitrary. Summing the estimate $\|{\sf X}_k^m - {\sf X}_0^m\|^2 \leq k \sum_{i=1}^k \|{\sf X}_{i}^m - {\sf X}_{i-1}^m\|^2$ for $k \in [K_m-1]$, we now get  
%
%
%
\begin{align*} {\sum}_{k=0}^{K_m-1} \|{\sf X}^m_k - \tilde {\sf X}^m \|^2 & \leq {\sum}_{i=1}^{K_m-1} \left[ {\sum}_{k=i}^{K_m-1} k \right] \|{\sf X}_i^m - {\sf X}_{i-1}^m\|^2 \\ÃÂ & \leq \frac{K_m(K_m-1)}{2} {\sum}_{i=1}^{K_m-1}  \|{\sf X}_i^m - {\sf X}_{i-1}^m\|^2
\end{align*}
for all $\omega \in \Omega$.
%
%
Combining the latter results, choosing $\nu = (K_m-1)^{-1}$ and $\rho := ((1+\nu)\lambda^m_{+})^{-1} \leq \mathcal L^m_k$, defining $\theta(\lambda^m_{+}) = L^2(1+\nu)\lambda^m_{+} [ \frac{1}{b_{m,+}} + \frac{1}{b_m}]$, and summing over $k$, it holds
\begin{align} \nonumber  2\Exp[\psi({\sf X}^m_{K_m})  \mid \mathcal F_{-1,+}^m] - 2 \psi({\sf X}^{m}_{0})  +  \sum_{k=0}^{K_m-1} \frac{1}{\lambda^m_k} \Exp[\|F^{\Lambda^m_k}({\sf X}_k^m)\|^2 \mid \mathcal F_{-1,+}^m]  & \\ \nonumber &\hspace{-61ex} \leq  + \sum_{i=1}^{K_m} {\textstyle\left[ L_f - \frac{1}{\lambda^m_{+}} + \half K_m (K_m-1)\theta(\lambda^m_{+}) \right]} \Exp[ \|{\sf X}^m_i - {\sf X}^m_{i-1} \|^2 \mid \mathcal F_{-1,+}^m] \\ \label{eq:sum-it-up} & \hspace{-58ex} +  \sum_{k=0}^{K_m-1}  {\Theta_k^m} \cdot  \Exp[ \|F^{\Lambda^m_k}_{{\sf V}^m_k}({\sf X}^m_k)\|^2 + \|p^{\Lambda^m_k}_{{\sf V}_k^m}({\sf X}_k^m) - p^{\Lambda^m_k}({\sf X}_k^m)\|^2 \mid \mathcal F_{-1,+}^m]. 
\end{align}
%
 Next, we want to choose the step sizes $\lambda_{+}^m$ and $\lambda^m_k$ such that 
\be \label{eq:def-lamplus-app} L_f - \frac{1}{2\lambda^m_{+}} + \frac{1}{2} K_m(K_m-1) \theta(\lambda^m_{+}) \leq 0 \quad \text{and} \quad \mathcal L_k^m - \frac{1-\bar\rho}{\lambda^m_k} \leq 0. \ee
In this case, the two last sum expressions in the preceding inequality \eqref{eq:sum-it-up} are nonpositive and can therefore be neglected. We first derive a bound for the parameter $\lambda_{+}^m$ using the first inequality in \eqref{eq:def-lamplus-app}. In particular, the (unique) positive solution $\bar \lambda^m_+$ of its associated quadratic equation is given by $\bar \lambda^m_+ = (L_f + ({L_f^2} + L^2\tau_m^2)^{1/2})^{-1}$. Hence, all $\lambda^m_{+}$ with
\[ \lambda^m_+ \leq \frac{1}{2L_f+ L \cdot \tau_m} \]
satisfy the first inequality \eqref{eq:def-lamplus-app}. Rearranging the terms and using \eqref{eq:def-lamplus-app} and the definition of $\mathcal L_k^m$ yields the bound stated in Theorem \ref{theorem:complex}. 
%
%
Summing the inequality \eqref{eq:sum-it-up} for $m$, taking expectation, and applying (A.2) and \cite[Lemma 2]{Nes13}, it follows
\be \label{eq:esti-for-complex} \half \sum_{m=0}^{M-1} \sum_{k=0}^{K_m-1} \frac{\Exp[ \|F^{I}({\sf X}^m_k)\|^2 ]}{({\lambda^m_k})^{-1}} \leq \sum_{m=0}^{M-1} \Exp[ \psi(\tilde{\sf X}^m) -\psi(\tilde{\sf X}^{m+1})] \leq \psi(x^0) -\psi^*  \ee
%
%
%
for all $M \in \N$. Using the bounds \eqref{eq:def-lamplus-app} in \eqref{eq:sum-it-up}, we further obtain 
\be \label{eq:another-eq} \sum_{m=0}^\infty \sum_{k=0}^{K_m-1} \frac{\Exp[\|{\sf X}^m_{k+1} - {\sf X}^m_{k}\|^2]}{\lambda_+^m}  < \infty, \, \sum_{m=0}^\infty \sum_{k=0}^{K_m-1} \frac{\overline{\rho}\,\Exp[\|F^{\Lambda_k^m}_{{\sf V}^m_k}({\sf X}^m_k)\|^2]}{\lambda_k^m} < \infty. \ee
Moreover, due to condition (A.2), Theorem \ref{theorem:conv-superm} is applicable, which establishes almost sure convergence of $(\psi(\tilde{\sf X}^m))_m$. By the reverse H\"older inequality, we have
\[ \sum_{m=0}^{\infty} \sum_{k=0}^{K_m-1} {\lambda^m_k}  \Exp[ \|F^{I}({\sf X}^m_k)\|^2 ] \geq  \sum_{m=0}^\infty {\left( \sum_{k=0}^{K_m-1} \frac{1}{\lambda^m_k} \right)^{-1}}  \Exp\left[\left( \sum_{k=0}^{K_m-1} \|F^{I}({\sf X}^m_k)\|  \right)^2 \right],  \]
which implies $\liminf_{m\to \infty} K_m^{-1} \sum_{k=0}^{K_m-1} \Exp[\|F^I({\sf X}^m_k)\|] = 0$ under the additional assumptions in \eqref{eq:step-choice-vr}. To prove convergence of the whole sequence, we proceed as in Theorem \ref{theorem:gen-conv-special}. In particular, let us define $y_m := K_m^{-1} \sum_{k=0}^{K_m-1} \Exp[\|F^I({\sf X}^m_k)\|]$ and suppose that $(y_m)_m$ does not converge to zero. Then, there exist increasing sequences $(t_i)_i$ and $(\ell_i)_i$ such that $\ell_i > t_i$ for all $i$ and 
\[ y_{t_i} \geq 2\veps, \quad y_{\ell_i} < \veps, \quad \text{and} \quad y_k \geq \veps, \quad \forall~k=t_i+1,...,\ell_i-1. \]
Hence, setting $\gamma_i := \sum_{r=t_i}^{\ell_i-1} K_r^2 [\sum_{k=0}^{K_r-1} (\lambda_k^r)^{-1} ]^{-1} $, it follows
\begin{align*}
\infty > \sum_{m=0}^\infty K_m^2 \left[ \sum_{k=0}^{K_m-1} \frac{1}{\lambda_k^m} \right]^{-1} y_m^2 \geq \sum_{i=0}^\infty \sum_{r=t_i}^{\ell_i-1} K_r^2 \left[ \sum_{k=0}^{K_r-1} \frac{1}{\lambda_k^r} \right]^{-1} y_r^2 \geq \veps^2 \sum_{i=0}^\infty \gamma_i
\end{align*}
and the sequence $(\gamma_i)_i$ must converge to zero. Utilizing Jensen's inequality for $| \Exp[\cdot ] |$ and $(\Exp[\cdot])^2$, the Lipschitz continuity of $x \mapsto F^I(x)$, the reverse triangle inequality, ${\sf X}^{r+1}_0 = {\sf X}^r_{K_r}$, H\"older's inequality, and the second condition in  \eqref{eq:step-choice-vr}, we obtain
\begin{align*} |y_{\ell_i} - y_{t_i}|  &  \leq |y_{\ell_i} - \Exp[\|F^I({\sf X}_0^{\ell_i})\|]| + {\sum}_{r=t_i+1}^{\ell_i-1} \Exp[| \| F^I({\sf X}_0^{r+1}) \| - \| F^I({\sf X}_0^{r}) \| | ] \\ & \hspace{4ex}+ |\Exp[\|F^I({\sf X}^{t_i+1}_0)\|] - y_{t_i}| \\ & \leq (2+L_f) \left[  K_{\ell_i}^{-1} \sum_{k=0}^{K_{\ell_i}-1} \Exp[ \|{\sf X}^{\ell_i}_k - {\sf X}^{\ell_i}_0\|] +  \sum_{r=t_i+1}^{\ell_i-1} \Exp[\|{\sf X}^{r+1}_0 - {\sf X}^r_0\|] \right.  \\ & \hspace{4ex} \left. +K_{t_i}^{-1} {\sum}_{k=0}^{K_{t_i}-1} \Exp[\|{\sf X}_0^{t_i+1}-{\sf X}_k^{t_i}\|] \right] \\
& \leq (2+L_f) {\sum}_{r=t_i}^{\ell_i} {\sum}_{k=0}^{K_r-1}\Exp[\|{\sf X}^r_{k+1} -{\sf X}^r_{k}\|] & \\ &\leq (2+L_f) \left[ K_{\ell_i} \lambda_+^{\ell_i} + \underline{\rho}^{-1}\gamma_i \right]^\half \left[ \sum^{\ell_i}_{r = t_i} \sum_{k=0}^{K_r - 1} \frac{1}{\lambda_+^r} \Exp[\|{\sf X}^r_{k+1} - {\sf X}^r_k\|^2] \right]^\half.
 \end{align*}
 Due to $K_{\ell_i}\lambda_+^{\ell_i} \leq \kappa$ and \eqref{eq:another-eq}, this implies $|y_{\ell_i} - y_{t_i}| \to 0$ as $i \to \infty$ which is a contradiction to $|y_{\ell_i} - y_{t_i}| \geq \veps$. Almost sure convergence can be shown in an analogous way. (See also the proof of Theorem \ref{theorem:gen-conv-special} for comparison).  
\end{proof}

%
%

The statements in Theorem \ref{theorem:complex} and the assumptions on the step sizes $\lambda_k^m$ and $\lambda_+^m$ resemble our previous results and requirements. However, since the number of inner iterations $K_m$ is allowed to increase or change, the conditions for the choice of $\lambda_k^m$ are more complex. In general, in order to ensure strong ergodic-type convergence of the stationarity measure $\|F^I({\sf X}^m_k)\|$, the step sizes $\lambda_k^m$ and $\lambda_+^m$ again need to be of similar order. This can be achieved by balancing and bounding the parameters $\nu_k^m$ and by performing a fixed number of inner iterations $K_m \equiv K$, $m \in \N$, for instance. If the step sizes $\lambda_k^m \equiv \lambda^m$ do not depend on $k$, then the conditions in \eqref{eq:step-choice-vr} reduce to $\sum_{m=0}^\infty K_m \lambda^m = \infty$ and $\underline{\rho}\lambda^m_+ \leq \lambda^m$ for all $m$. Furthermore, if only the first condition in \eqref{eq:step-choice-vr} is satisfied, we can still show a weaker convergence result, i.e., in this situation it holds that $ \liminf_m K_m^{-1} \sum_{k=0}^{K_m-1} \Exp[\|F^I({\sf X}^m_k)\|] = 0$ and $\liminf_m K_m^{-1} \sum_{k=0}^{K_m-1} \|F^I({\sf X}^m_k)\| = 0$ \as  

\begin{remark} \label{remark:vr}Computationally, it would be highly attractive to use the same sample set $\mathcal S^m_k$ to generate the oracle $v_{k,+}^m$. Similar to the observations made in Remark \ref{remark:unbiased-1}, it is still possible to show convergence in this case, if the bounds on the step sizes $\lambda_{+}^m$ and $\lambda_{k}^m$ are adjusted accordingly. (Notice that Lemma \ref{lemma:sto-err-bound} is not directly applicable).  Specifically, utilizing the estimate  
\begin{align*} \|\nabla f(z_k^m) - v_{k,+}^m\|^2 \leq \frac{1+\nu}{\nu}[(L+L_f)\beta_k^m \|d^m_k\|]^2 + (1+\nu) \|\nabla f(x_k^m) - v_k^m\|^2
\end{align*}
for $v_{k,+}^m = \nabla f_{\mathcal S^m_k}(z^m_k) - \nabla f_{\mathcal S^m_k}(\tilde x^m) + \nabla f(\tilde x^m)$, $\nu > 0$, we can extend the statements and calculations in the proof of Theorem \ref{theorem:complex} to cover this situation. 
\end{remark}
 
As in \cite{RedHefSraPocSmo16,RedSraPocSmo16,WanMaGolLiu17}, our convergence analysis allows to derive complexity results for reaching an $\veps$-accurate stationary point in terms of the number of gradient component evaluations (IFO) when the step sizes $\lambda_k^m$ and $\lambda_+^m$ are constant. 
%
%
%
%
%
%
We summarize our observations in the following corollary. A brief proof of Corollary \ref{corollary:complexity} is presented in the appendix.  

\begin{corollary} \label{corollary:complexity} Suppose that Algorithm \ref{alg:inex-qnm} is run with $K_m \equiv K$ and constant parameter matrices $\Lambda_k^m \equiv \lambda^{-1}I$, $\Lambda_{k,+}^m \equiv \lambda_+^{-1}I$ and batch sizes $b_m \equiv b$, $b_{m,+} \equiv b_+$. Let us assume that the conditions {\rm(A.1)}--{\rm(A.2)} and {\rm(C.1)}--{\rm(C.4)} are satisfied and that there is $\bar \nu > 0$ such that $\alpha_k^m \nu_k^m \leq \bar \nu$ and $\beta^m_k\nu_k^m \leq \bar \nu$ for all $m$, $k$. Furthermore, let us set
\[ K = \lceil N^{\frac13}\rceil, \quad b = b_+ = K^2, \quad \lambda_+ = \frac{\gamma}{L}, \quad \lambda = \frac{\gamma}{L(1+3\bar\nu^2)}, \quad \gamma = \frac{\sqrt{5}-1}{2}, \]
and suppose that the output variable ${\sf X}$ is sampled uniformly at random from the iterates $\{{\sf X}^m_k : k \in [K-1]_0, \, m \in [M-1]_0\}$. Then, it holds that 
\be \label{eq:cor-complex} \Exp[\|F^I({\sf X})\|^2] \leq \frac{2L(1+3\bar\nu^2)[\psi(x^0)-\psi^*]}{\gamma MK} \ee
and the total number of component gradient evaluations to achieve an $\veps$-accurate stationary point with $\Exp[\|F^I({\sf X})\|^2] \leq \veps$ is $\mathcal O(N^{2/3}/\veps)$.
\end{corollary}

\section{Higher Order-Type and Quasi-Newton Directions} \label{section:direction}

We now describe possible strategies for choosing and constructing the directions $d^k$ and $d^m_k$ used in Algorithm \ref{alg:seqn} and \ref{alg:inex-qnm}. 

\subsection{Stochastic Semismooth Newton-type Directions} In order to motivate the quasi-Newton directions proposed in the next subsections, we first consider a ``full'' stochastic second order step. The principal idea of semismooth Newton-type methods is to build a suitable generalized derivative $M$ of the nonsmooth residual $F^\Lambda$ at $x$ and to iteratively perform updates of the form
\[ M d = - F^\Lambda(x) \quad \text{and} \quad x \leftarrow x + d, \]
see, e.g., \cite{QiSun93,Qi93,PanQi93}. In \cite{MilXiaCenWenUlb18}, Milzarek et al. propose a stochastic variant of this procedure that is based on the following stochastic generalized derivative
\[ \mathcal M_{v,H}^\Lambda(x) := \{M: M = (I-D) + D\Lambda^{-1} H, \, D \in \partial \proxt{\Lambda}{\vp}(u^\Lambda_v(x)) \} \]
where $H \approx \nabla^2 f(x)$ is a stochastic approximation of the Hessian and $\partial \proxt{\Lambda}{\vp}$ denotes the Clarke subdifferential of the proximity operator. After choosing $M \in \mathcal M_{v,H}^{\Lambda}(x)$, a stochastic direction can be generated via $d = - W F^{\Lambda}_{v}(x)$ where $W \approx M^{-1}$. In practice, $d$ may be obtained by using an inexact, iterative solver. The theoretical results presented in the last sections are applicable to directions of this form as long as an appropriate bound of the (inexact) inverse is available or if the output of the utilized iterative solver can be suitably controlled. Since the computation of such a stochastic second order-type step can be time-consuming and expensive in large-scale settings, we will focus on more straightforward ways to generate $d^k$.


\subsection{Stochastic L-BFGS Steps} \label{sto-full-qn}
The limited memory BFGS method (L-BFGS) is a classical and ubiquitous algorithmic scheme to generate cheap and tractable quasi-Newton steps. 
Based on a set of recent curvature pairs $\{ U_{k},Y_{k}\}$,
\[ U_{k} = [u_{k-p},\dots,u_{k-1}] \in \R^{n \times p}, \quad Y_{k} =[y_{k-p},\dots,y_{k-1}] \in \R^{n \times p}, \]
a matrix $W_k$ is built via the recursion: $W_k^0 = \gamma_k I$, 
\beÂ \label{eq:bfgs-up} W_k^{i+1} = (I - \rho_i u_{k_p(i)} y_{k_p(i)}^\top) W_k^{i} (I- \rho_i y_{k_p(i)} u_{k_p(i)}^\top) + \rho_i u_{k_p(i)} u_{k_p(i)}^\top, \ee
for $k_p(i) := k-p+i$, $i = 0,...,p-1$, and $W_k = W_k^p$, where 
\be \label{eq:bfgs-para} \rho_i = \frac{1}{\iprod{y_{k_p(i)}}{u_{k_p(i)}}} \quad \text{and} \quad \gamma_k = \frac{\iprod{u_{k-1}}{y_{k-1}}}{\iprod{y_{k-1}}{y_{k-1}}}. \ee 
%
If $u_{k-1}$ and $y_{k-1}$ are chosen as $u_{k-1} = x^{k} - x^{k-1}$ and $y_{k-1} = \nabla f(x^{k}) - \nabla f(x^{k-1})$, then the process \eqref{eq:bfgs-up} coincides with the standard L-BFGS method, \cite{Noc80,LiuNoc89}, to approximate the inverse of the Hessian $\nabla^2 f$ at $x^k$. Recently, various stochastic L-BFGS strategies have been developed for convex and nonconvex smooth problems, see, e.g., \cite{schraudolph2007stochastic,MokRib14,MokRib15,ByrHanNocSin16,gower2016stochastic,BerNocTak16,WanMaGolLiu17}. Since the full gradient $\nabla f$ is not accessible, the key differences between those stochastic L-BFGS techniques are how the pairs $\{ U_{k},Y_{k}\}$ are updated and constructed. In this paper, we utilize the following iterate and stochastic residual differences: 
%
\begin{equation}
\label{curve}
u_k = z^{k} - x^k \quad \text{and} \quad y_k =F^{\Lambda_{k}}_{v^k_z}(z^{k})-F^{\Lambda_k}_{v^k}(x^k).
\end{equation} 
Furthermore, an update of the pairs $\{U_k, Y_k\}$ is only performed if the condition
\begin{equation}
\label{curvecond}
\iprod{u_k}{y_k} \geq \delta \|u_k\|^2.
\end{equation} 
holds for some constant $\delta > 0$. We also assume that the same stochastic oracle $v^k$ can be evaluated at the point $z^k$ to generate $v_z^k$. This is specified in more detail in Assumption \ref{assumption:four} in the next section. 
We note that a deterministic variant of this L-BFGS scheme was used successfully in \cite{XiaLiWenZha18} for convex composite programs.  
Moreover, other (stochastic or deterministic) strategies can be also be incorporated in our framework, see, e.g., \cite{CheQi94,Qi97,SunHan97,WanMaLi11,ManRun18} for related BFGS techniques for nonsmooth problems.


\subsection{ Coordinate Quasi-Newton Method}

The cost of computing a stochastic L-BFGS direction is $\mathcal O(pn)$. 
This cost is generally dominated by $n$ and can still be noticeably high when the dimension $n$ is large. In addition, the direction generated by L-BFGS is based on a dense approximation while in some applications the associated generalized derivative of $F^{\Lambda}$ is a sparse block matrix, as in, e.g., $\ell_1$-problems, \cite{XiaLiWenZha18}. 

Next, to further reduce the computational costs, we propose a coordinate-type L-BFGS method. Based on two disjoint index sets $\mathcal{I}\equiv \mathcal{I}(x^k)$ and $\mathcal{A} \equiv \mathcal{A}(x^k)$, we consider directions $d^k$ of the form:
 \be \label{eq:bfgs-coor-def}
 d^k = 
-W_k F_{v^k}^{\Lambda_k}(x^k)= - 
\begin{bmatrix}
W_{\mathcal{I}\mathcal{I}} & 0 \\[-.5ex] \\
0 & \zeta_k I\\
\end{bmatrix}
\begin{bmatrix}
(F_{v^k}^\Lambda(x^k))_{\mathcal{I}} \\[-0.5ex] \\
(F_{v^k}^{\Lambda_k}(x^k))_{\mathcal A} 
\end{bmatrix},
\ee
where $\zeta_k >0$ is a parameter and $W_{\mathcal{I}\mathcal{I}}$ is determined by the standard L-BFGS method using the lower dimensional curvature pairs $(U_k)_{[\mathcal{I}\cdot]}$ and $(Y_k)_{[\mathcal{I}\cdot]}$. A similar block-wise BFGS scheme was also investigated by Janka et al. \cite{JanKirSagWae16} for SQPs. 

The curvature pairs $\{U_k,Y_k\}$ are generated as in \eqref{curve} and every pair $\{u_i,y_i\}$, $i = k-p,...,k-1$, is supposed to satisfy the condition \eqref{curvecond}. Let $\mathcal Q \subseteq [p-1]_0$ denote the set of indices such that 
\be \label{eq:bfgs-coor} | \iprod{u_{k_p(q),\mathcal{I}}}{y_{k_p(q),\mathcal{I}}} | \geq \delta_1 \|u_{k_p(q)}\|^2, \quad \forall~q \in \mathcal Q, \ee
for some fixed $\delta_1 > 0$. The matrix $W_{\mathcal{I}\mathcal{I}}$ is now constructed via the L-BFGS recursion \eqref{eq:bfgs-up} using the adjusted curvature pairs 
%
\[ \bar U_k = [u_{k_p(q),\mathcal I}]_{q \in \mathcal Q} \in \R^{|\mathcal I| \times |\mathcal Q|} \quad \text{and} \quad \bar Y_k = [y_{k_p(q),\mathcal I}]_{q \in \mathcal Q} \in \R^{|\mathcal I| \times |\mathcal Q|}. \]
%
%
If $\mathcal Q$ is empty, we reset $\mathcal I$ to $[n]$. In such a case, $W_{\mathcal I\mathcal I} = W_k$ coincides with the full stochastic L-BFGS update presented in section \ref{sto-full-qn}.
In this paper, we propose to select the set as follows:
\begin{equation}
\label{coordinate-choose}
\mathcal{I}(x^k) := \{i\in [n] : |(F^{\Lambda_k}_{v^k}(x^k))_i| \geq \delta_2 \}, \quad  \mathcal{A}(x^k):=[n]\setminus \mathcal{I}(x^k),
\end{equation}
and $\delta_2  = 10^{-6}$. The choice of the index sets $\mathcal I$ and $\mathcal A$ is very flexible. In particular, we can consider multiple blocks and split the set $\mathcal I$ into several smaller disjoint sets. The coordinate selection strategy can also be tailored to specific optimization problems. 

We now present the stochastic setup and conditions that will allow us to formalize the definition of $v_z^k$ in \eqref{curve} and to show that the coordinate L-BFGS-type direction $d^k$ satisfies the assumptions (B.1)--(B.2) or (C.2)--(C.3). 

%




\begin{assumption} \label{assumption:four} Suppose that Algorithm \ref{alg:seqn} uses parameter matrices of the form $\Lambda_k = (\lambda_k)^{-1} I$. We consider the stochastic conditions:
\setlength{\leftmargini}{7ex}
\begin{itemize}
\item[{\rm(D.1)}] Let $(\Upsilon,\mathcal U)$ be a measurable space and ${\sf V} : \Rn \times \Upsilon \to \Rn$ is a Carath\'eodory function. For all $k \in \N$, we assume that ${\sf V}^k$ and ${\sf V}^k_+$ are generated as follows: 
\end{itemize}
\[ {\sf V}^k(\omega) := {\sf V}({\sf X}^k(\omega),\Xi^k(\omega)), \quad {\sf V}^k_+ = {\sf V}({\sf Z}^k(\omega),\Xi^k_+(\omega)), \quad \forall~\omega \in \Omega, \]
\begin{itemize}
\item[] where $\Xi^k,\Xi^k_+ : \Omega \to \Upsilon$ are $(\mathcal F,\mathcal U)$-measurable random variables.
\item[{\rm(D.2)}] There exists a random variable ${\sf L} : \Upsilon \to \R$ and $\ell > 0$ such that
\end{itemize}
\[ \|{\sf V}(x,u) - {\sf V}(y,u)\| \leq {\sf L}(u) \|x - y\|, \quad \forall~u \in \Upsilon, \quad \forall~x,y \in \Rn \]
\begin{itemize}
\item[] and we have $\sup_{k \in \N} \lambda_k {\sf L}(\Xi^k(\omega)) \leq \bar\ell$ almost surely. 
\end{itemize}
\end{assumption}

Condition (D.1) implies that the stochastic oracles $v^k$ and $v^k_+$ are generated by selecting two samples $\xi^k$ and $\xi^k_+$ and by setting $v^k = {\sf V}(x^k,\xi^k)$ and $v^k_+ = {\sf V}(z^k,\xi^k_+)$. This finally allows us to define the approximation $v^k_z$, used in \eqref{curve}, as follows 
\be \label{eq:vkz} v^k_z := {\sf V}(z^k,\xi^k). \ee 
Assumption (D.2) is a Lipschitz-type condition that appears frequently (in different variants) in the analysis of stochastic optimization methods, see, e.g., \cite{IusJofOliTho17,WanMaGolLiu17,DavDru18-2}.

\begin{example} Let us consider the mini-batch-type stochastic oracle $\nabla f_{\mathcal S_k}$ introduced in section \ref{section:extra-vr} for empirical risk problems. Suppose that $\mathcal S_k \subset [N]$ is a sub-sample that is chosen uniformly at random and without replacement from $[N]$ and let us set $\Upsilon := \{0,1\}^N$ and $\mathcal U := \mathcal P(\Upsilon)$. Then, we can define
\[ {\sf V}: \Rn \times \Upsilon \to \Rn, \quad {\sf V}(x,u) := {\iprod{\mathds 1}{u}}^{-1} \cdot {\sum}_{i\in [N]} \, u_i \nabla f_i(x). \]
and for $\xi^k_i := \mathds 1_{\mathcal S_k}(i)$, $i \in [N]$, it follows $\nabla f_{\mathcal S_k}(x^k) = {\sf V}(x^k,\xi^k)$. Consequently, if the mappings $\nabla f_i$ are all Lipschitz continuous, then the mini-batch stochastic gradient $\nabla f_{\mathcal S_k}$ satisfies the conditions (D.1)--(D.2). \end{example}


Clearly, the random variable ${\sf D}^k$ associated with the direction $d^k$ defined in \eqref{eq:bfgs-coor-def} is $\mathcal F^k$-measurable. We now verify the boundedness property (B.2).
\begin{lemma}
\label{bfgs_upper}
Suppose that the matrices $(W_k)_k$ and the corresponding stochastic process $({\sf W}_k)_k$ are generated via the block coordinate scheme \eqref{eq:bfgs-coor-def} with curvature pairs $\{U_k,Y_k\}$ as specified in \eqref{curve}--\eqref{curvecond} and \eqref{eq:bfgs-coor}--\eqref{eq:vkz}. Let us assume that the conditions {\rm (D.1)}--{\rm(D.2)} are satisfied and it holds that $\Lambda_k = \lambda_k^{-1}I$ and $|\zeta_k| \leq \bar \zeta $ for some $\bar \zeta > 0$ and all $k$. Then there exists $\bar \nu \equiv \bar\nu(\bar\ell,p,\delta_1,\bar\zeta)$ such that, almost surely, we have
\[
\|{\sf W}_k\| \leq \bar\nu \quad \text{and} \quad \|{\sf D}^k\| \leq \bar \nu \cdot \|F^{\Lambda_k}_{{\sf V}^k}({\sf X}^k)\|, \quad \forall~k.
\]
\end{lemma}
\begin{proof}
As usual, the associated stochastic processes of the curvature pairs $(u_k)_k$ and $(y_k)_k$ are denoted by $({\sf U}_k)_k$ and $({\sf Y}_k)_k$. Then, due to assumption (D.2) and using the nonexpansiveness of the proximity operator, we obtain
\begin{align*}  \|{\sf Y}_k\| & \leq \|{\sf U}_k\| + \|\proxt{\Lambda_k}{\vp}({\sf Z}^k -\Lambda_k^{-1}{\sf V}_z^k) - \proxt{\Lambda_k}{\vp}({\sf X}^k - \Lambda_k^{-1}{\sf V}^k) \| \\ & \leq 2 \|{\sf U}_k\| + \lambda_k \|{\sf V}({\sf Z}^k,\Xi^k) - {\sf V}({\sf X}^k,\Xi^k)\| \leq (2+\bar\ell) \|{\sf U}_k\|
\end{align*}
for all $k$ and for almost every $\omega \in \Omega$. Next, let us consider an arbitrary realization of $({\sf Y}_k)_k$ and $({\sf U}_k)_k$ with $\|y_k\| \leq (2+\bar\ell)\|u_k\|$ for all $k \in \N$. Since $W_{\mathcal I\mathcal I}$ is generated by \eqref{eq:bfgs-up} using the pairs $\{\bar U_k, \bar Y_k\}$, we can write $W_{\mathcal I\mathcal I} = \bar W^{|\mathcal Q|}_k$ where $\bar W_k^0 = \bar \gamma_k I$, 
\[ \bar W_k^{i+1} = (I - \bar \rho_i \bar u_i \bar y_i^\top) \bar W_k^i (I - \bar \rho_i \bar u_i \bar y_i^\top) + \bar \rho_i \bar u_i \bar u_i^\top, \quad \bar\rho_i = \iprod{\bar u_i}{\bar y_i}^{-1}, \]
and $\bar \gamma_k = \iprod{\bar u_{|\mathcal Q|}}{\bar y_{|\mathcal Q|}} \|\bar y_{|\mathcal Q|}\|^{-2}$. We can now proceed as in \cite[Lemma 3.3]{WanMaGolLiu17}. Specifically, due to $\|uu^\top\| = \|u\|^2$, $\|uy^\top\| = \|u\|\|y\|$, and $|\bar \rho_i| \leq [\delta_1 \| \bar u_{i}\|^2]^{-1}$, we have
\begin{align*} \|\bar W_k^{i+1}\| & \leq \|\bar W_k^i\| + 2 |\bar \rho_i| \|\bar W_k^i\| \|\bar u_{i}\bar y_{i}^\top\| + \bar \rho_i^2 \|\bar W_k^i\| \|\bar u_{i}\bar y_{i}^\top\|^2 + |\bar \rho_i| \|\bar u_i\|^2 \\ & \leq \left[1 + {(2+\bar\ell)}\delta_1^{-1} \right]^2 \|\bar W_k^i\| + \delta_1^{-1} 
 \end{align*}
for $i = 0,...,|\mathcal Q|-1$. Similarly, the condition \eqref{eq:bfgs-coor} yields $|\iprod{\bar u_{|\mathcal Q|}}{\bar y_{|\mathcal Q|}}| \leq \delta_1^{-1} \|\bar y_{|\mathcal Q|}\|^2$ and $\bar \gamma_k \leq \delta_1^{-1}$. Together, this implies
\[ \|W_{\mathcal I\mathcal I}\| \leq \frac{1}{\delta_1} \sum_{i=0}^{|\mathcal Q|+1} \left[ \frac{2+\bar\ell + \delta_1}{\delta_1} \right]^{2i} \leq \frac{1}{2+\bar\ell} \left[ \left(\frac{2+\bar\ell+\delta_1}{\delta_1}\right)^{2(p+2)} - 1 \right]. \]
This shows that there exists a constant $\bar \nu$ (that only depends on $\bar\ell$, $p$, $\delta_1$, and $\bar \zeta$) such that $\|W_k\| \leq \bar \nu$ and finishes the proof of Lemma \ref{bfgs_upper}. 
\end{proof}

We note that by adjusting the curvature condition \eqref{eq:bfgs-coor} or by introducing an additional Powell damping strategy, see, e.g., \cite{JanKirSagWae16,WanMaGolLiu17}, the matrices $(W_k)_k$ can also be guaranteed to be uniformly positive definite.

\section{Numerical Results: Sparse Logistic Regression}
We first consider $\ell_1$-regularized logistic regression problems for binary classification:
\begin{equation}
\label{eq:LR} \min_{x \in \mathbb{R}^n}  \psi(x) = \frac{1}{N}\sum_{i=1}^{N} \log(1+\exp(-b_i \iprod{a_i}{x}) + \mu\|x\|_1, \end{equation}
where the data pairs $\{a_i,b_i\} \in \mathbb{R}^n \times \{-1,1\}$, $i\in [N]$, correspond to a given dataset. The parameter $\mu > 0 $ controls the level of sparsity and is set to $\mu = \frac{1}{N}$. 


\subsection{Datasets and Experimental Setting}
\label{relatedalgorithm}
\begin{table}[t]
\centering
\begin{tabular}{|p{10ex}p{8ex}p{8ex}|}
\hline 
data set & $N$ &  $n$ \\ 
\hline 
$\mathtt{cina}$ 	& 16033 		& 132 \\	
$\mathtt{a9a}$ 	& 32561		& 123 \\	
$\mathtt{ijcnn1}$ 	& 49990 		& 22 \\ 		
$\mathtt{covtype}$ 	& 581012 	& 54 \\		
$\mathtt{rcv1}$ 	& 20242 		& 47236 \\ 
\hline
\end{tabular}
\begin{tabular}{|p{10ex}p{10ex}p{10ex}|}
\hline 
data set & $N$ & $n$ \\ 
\hline 
$\mathtt{url}$ 		& 2396130	& 3231961 \\	
$\mathtt{susy}$ 	& 5000000	&18	\\		
$\mathtt{higgs}$ 	& 11000000	& 28	\\ 
$\mathtt{news20}$ 	&19996		& 1355191 \\ 
$\mathtt{kdda}$ 	& 8407752	& 20216830	 \\ 
\hline
\end{tabular}
\caption{A description of the datasets used in the numerical comparison.}
\label{table:datasets}
\end{table}

In our numerical experiments, we evaluate the performance of various solvers for the $\ell_1$-problem \eqref{eq:LR} on different datasets which are summarized in Table \ref{table:datasets}. All datasets except $\mathtt{cina}$ are downloaded from the LIBSVM website \cite{ChaLin11}. The tested datasets are large-scale and cover different numbers of data points $N$ and  features $n$. 
In addition to the training loss, 
we also keep track of the prediction accuracy during the training process. For $\mathtt{a9a}$, $\mathtt{ijcnn1}$, $\mathtt{kdda}$, and $\mathtt{rcv1}$, the training and testing sets are predefined in the original datasets. For the other datasets, we randomly select 80\% of the samples for training and use the remaining data for testing.



\subsection{Comparison for Different Variants}
\label{subsec:comparison-variants}
In the following, we conduct a preliminary comparison of different variants of Algorithm \ref{alg:inex-qnm} on the datasets $\mathtt{a9a}$ and $\mathtt{covtype}$. The total number of inner iterations $K_m$ and the size of the mini-batch stochastic gradient are chosen as follows:
\[ K_m \equiv K = 10, \quad b_m \equiv b = \min\{ 300, \lfloor 0.01N \rfloor \} \]
In order to simplify the notation, we use a single iteration counter for the inner and outer loops, i.e., we define $j := Km +k$ and 
\[ x^{|j} := x^{\lfloor j/K \rfloor}_{j - \lfloor j/K \rfloor K}, \quad \tilde x^{|j} := x^{\lfloor j/K \rfloor}_0, \quad \lambda_{|j} := \lambda^{\lfloor j/K \rfloor}_{j - \lfloor j/K \rfloor K}, \quad \text{etc.} \]
The oracle $v^{|j}$ is generated as specified in Algorithm \ref{alg:inex-qnm} by selecting a sample set $\cS^{|j}$ uniformly at random and without replacement from the index set $[N]$. As suggested and discussed in Remark \ref{remark:vr} and section \ref{section:direction}, we reuse the same sample set $\cS^{|j}$ to calculate the new oracles $v^{|j}_+$ and $v^{|j}_z$:
\be \label{eq:num-vr} v^{|j}_+ \equiv v^{|j}_z := \nabla f_{\cS^{|j}}(z^{|j}) - \nabla f_{\cS^{|j}}(\tilde x^{|j}) + \nabla f(\tilde x^{|j}). \ee
We consider the following four variants of Algorithm \ref{alg:inex-qnm}:
 \begin{itemize}[leftmargin=4ex]
\item SEQN-VR-1: We set $\alpha_{|j}=\beta_{|j} = 1$ and we utilize the coordinate quasi-Newton strategy to build the matrix $W_{|j}$ with $\zeta_{|j} = 1$ and $\mathcal I(x^{|j})$, $\mathcal A(x^{|j})$ as in \eqref{coordinate-choose}. 
The L-BFGS memory is set to $p = 10$.   
\item SEQN-VR-2:  The full stochastic L-BFGS method is applied with $\mathcal I(x^{|j}) = [n]$ for all $j$. Other settings are chosen as in SEQN-VR-1. 
\item SEQN-VR-3:  Different from \eqref{eq:num-vr}, we do not introduce a new oracle and set $v_+^{|j} \equiv v^{|j}$. Other parameters are as in SEQN-VR-1.
\item SEQN-VR-4 is a variant of SEQN-VR-2 with $v_+^{|j} \equiv v^{|j}$ for all $j$.
\end{itemize}

All four variants use matrices of the form $\Lambda_{|j,+}:=\lambda_{|j,+}^{-1}I$ and $\lambda_{|j,+}$ is calculated adaptively to estimate the Lipschitz constant of the gradient. Specifically, we compute 
\[ \lambda_{|j,+}^1 = \frac{\|z^{|j} - x^{|j}\|\cdot\min\{1,\lambda_{|j}\}}{\|F^{\Lambda_{|j}}_{v^{|j}_{z}}(z^{|j}) - F^{\Lambda_{|j}}_{v^{|j}}(x^{|j})\|}, \quad \lambda_{|j,+}^2 = \max\{10^{-3},\min\{10^3,\lambda_{|j,+}^1\}\}. \] 
The new step size $\lambda_{|j+1,+}$ is then defined as a weighted combination of $\lambda_{|j,+}^2$ and of the previous step sizes $\lambda_{|i,+}$, $i \in [j-1]$. The second step size is chosen via $\Lambda_{|j} =\lambda_{|j}I $ and $\lambda_{|j} = 0.5 \lambda_{|j,+}$. 
\begin{figure}[t]
	\centering
	\setlength{\belowcaptionskip}{-6pt}
	\begin{tabular}{cc}
		\hspace{-1.5ex}
		\subfloat[$\mathtt{a9a}$]{
			\includegraphics[width=5.5cm]{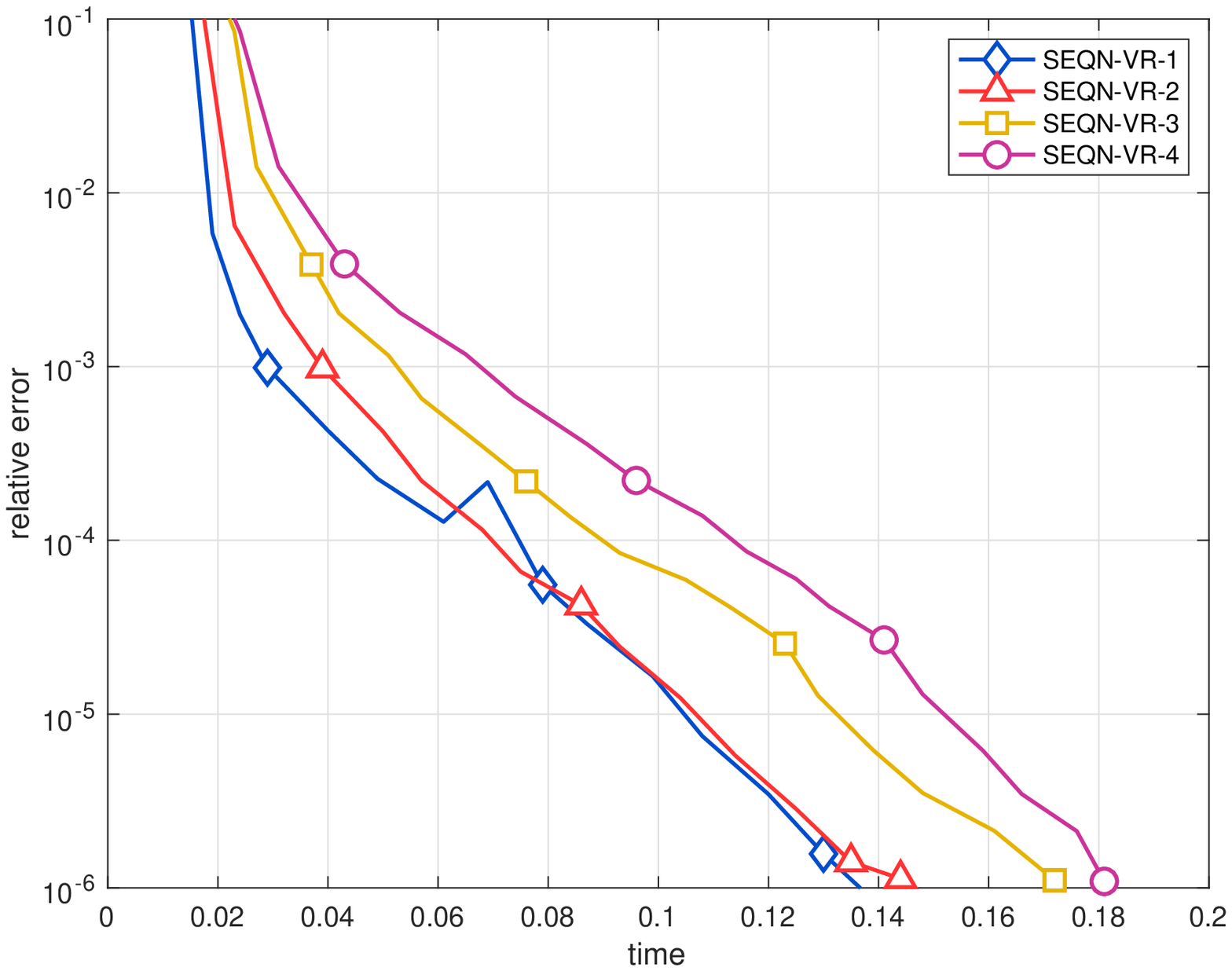}} &
			\subfloat[$\mathtt{covtype}$]{
			\includegraphics[width=5.5cm]{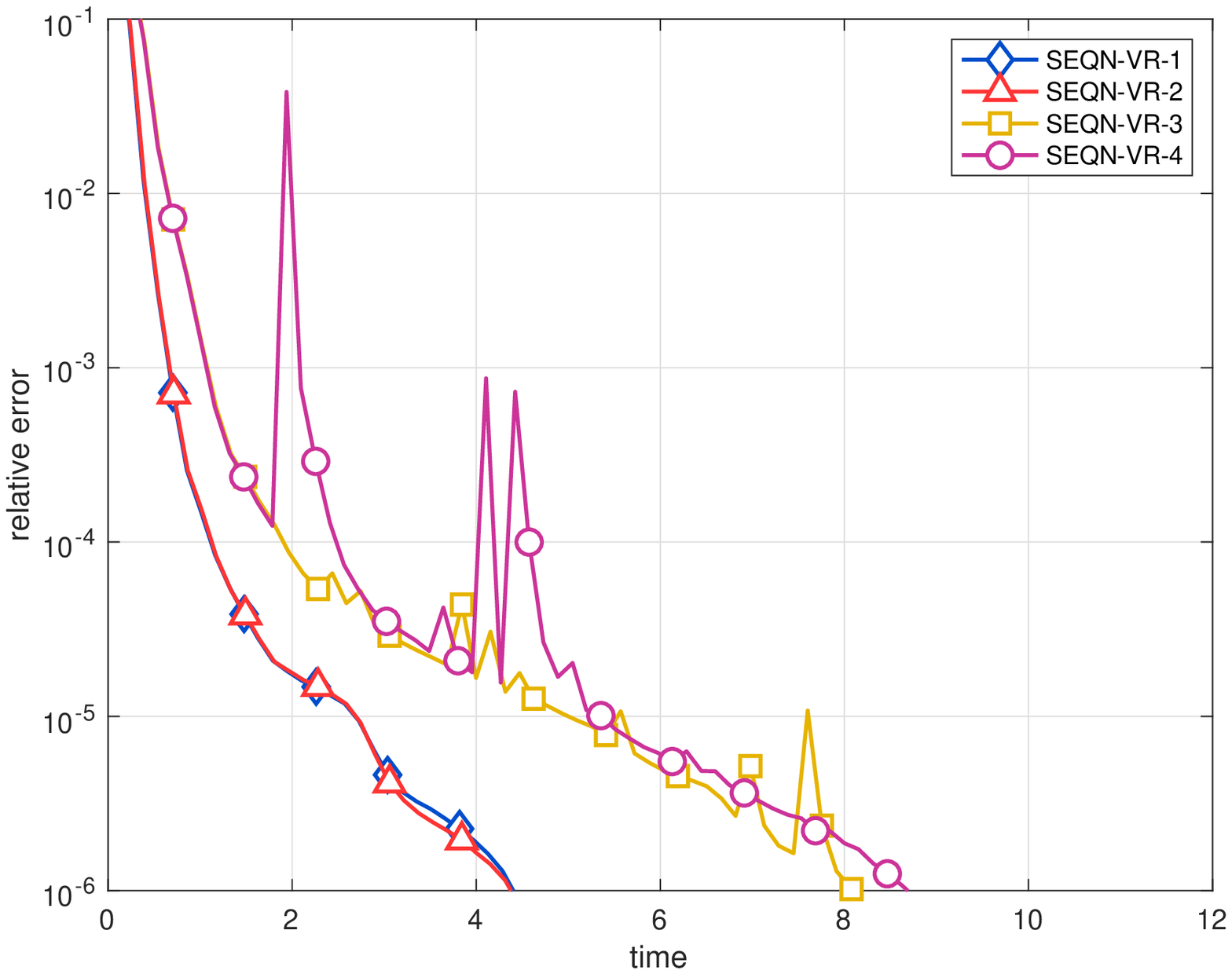}} \\
		\end{tabular}
	\caption{Comparison of different variants of SEQN.}
	\label{figure:variants}
\end{figure}
The results in Figure \ref{figure:variants} depict the average of the relative error versus the cpu-time over 10 independent runs. The performance of SEQN-VR-1 is slightly better than SEQN-VR-2. The two methods SEQN-VR-3 and SEQN-VR-4 are generally outperformed by SEQN-VR-1 and show small oscillations. In the next sections, we use SEQN-VR-1 as our main algorithm for comparisons.



\subsection{Additional Subspace Strategy} 

Next, we propose an additional subspace strategy to further robustify and improve the convergence of SEQN for problems with $n > N$. 
Inspired by the subspace optimization scheme developed in \cite{WenYinGolZha10}, we switch to a subspace phase if the current iterate is close to an optimal solution and solve the following subproblem:
\begin{equation}
\label{subproblem}
\text{find $x \in \Rn$ with} \quad F^\Lambda(x) = 0,\quad \text{s.t.} \quad x_i = x^{|j}_{i}, \quad \forall~i \in \mathcal{O}_{|j}, 
\end{equation}
where $\mathcal O_{|j} \subseteq [n]$ is a given index set. We again consider an approximate version of problem \eqref{subproblem} using the stochastic residual $F^\Lambda_v(x)$. The subspace phase is entered if the iterate $x^{|j}$ is nearly optimal and satisfies $\|F^{\Lambda_{|j}}_{v^{|j}}(x^{|j}) \|_{\infty} < \varepsilon_1$ for some $\varepsilon_1 > 0$. In this case, the set $\mathcal{O}_{|j}$ is chosen as
\[ \mathcal{O}_{|j} = \mathcal{O}(x^{|j};\varepsilon_2) := \{ i \in [n]: | x^{|j}_i| < \varepsilon_2 \}, \quad \varepsilon_2 > 0. \]
%
We keep the index set $ \mathcal{O}_{|j}$ fixed while solving the subproblem \eqref{subproblem} and its stochastic variant. If $|\mathcal{O}_{|j}|$ is large, the size of the problem and the number of active variables can be reduced significantly. We use SEQN-VR with stochastic L-BFGS updates (see section \ref{sto-full-qn}) and the initial point $x^{|j}$ to solve the stochastic subspace subproblem. A new set of curvature pairs is created to construct the L-BFGS directions. Let us notice that this subspace scheme can be interpreted as performing special coordinate-type quasi-Newton steps with $[W_{|r}]_{\mathcal O_{|j} \mathcal O_{|j}} = 0$ and $[v_{+}^{|r}]_{\mathcal O_{|j}} = 0$ for an inner iteration $r$, if the corresponding step sizes are chosen suitably.    
In our experiments, we only apply this strategy when the number of features is larger than the number of data points.

Figure \ref{figure:subpro} demonstrates the effect of the subspace strategy for $\mathtt{rcv1}$ and $\mathtt{news20}$. Clearly, the performance of the adjusted algorithm improves considerably. The version of SEQN-VR with subspace correction is up to ten times faster than the original SEQN-VR method. We leave the subspace phase when the stochastic residual is reduced sufficiently, i.e., if $\|F^{\Lambda_{|r}}_{v^{|r}}(x^{|r})\|/\lambda_{|r}\leq \min\{ 5\cdot10^{-7},0.01\cdot\|F^{\Lambda_{|j}}_{v^{|j}}(x^{|j})\|/\lambda_{|j}\}$, or if the number of inner iterations $r$ exceeds a certain value.

%
\begin{figure}[t]
	\centering
	\setlength{\belowcaptionskip}{-6pt}
	\begin{tabular}{cc}
		\hspace{-1.5ex}
		\subfloat[$\mathtt{rcv1}$]{
			\includegraphics[width=5.5cm]{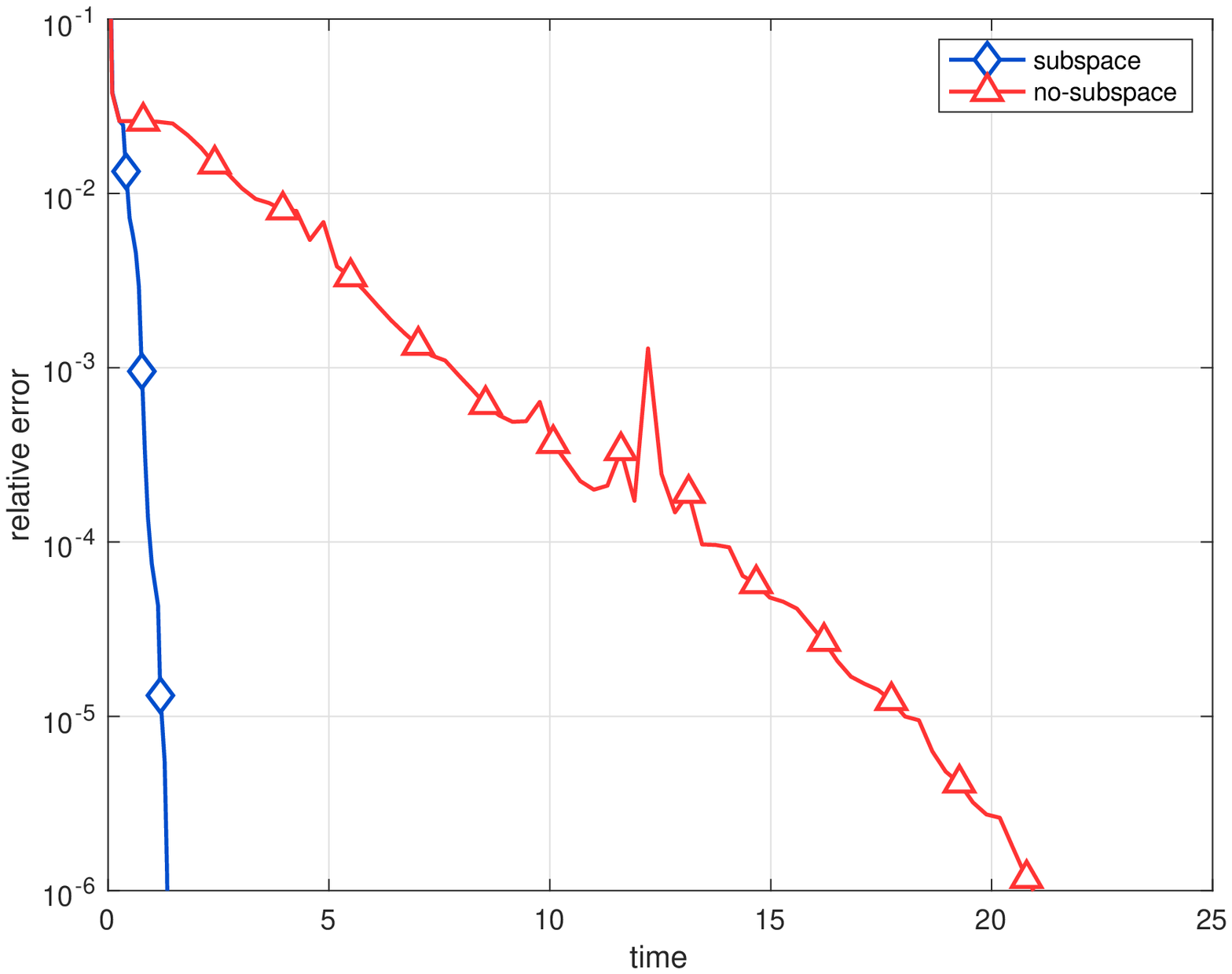}} &
			\subfloat[$\mathtt{news20}$]{
			\includegraphics[width=5.5cm]{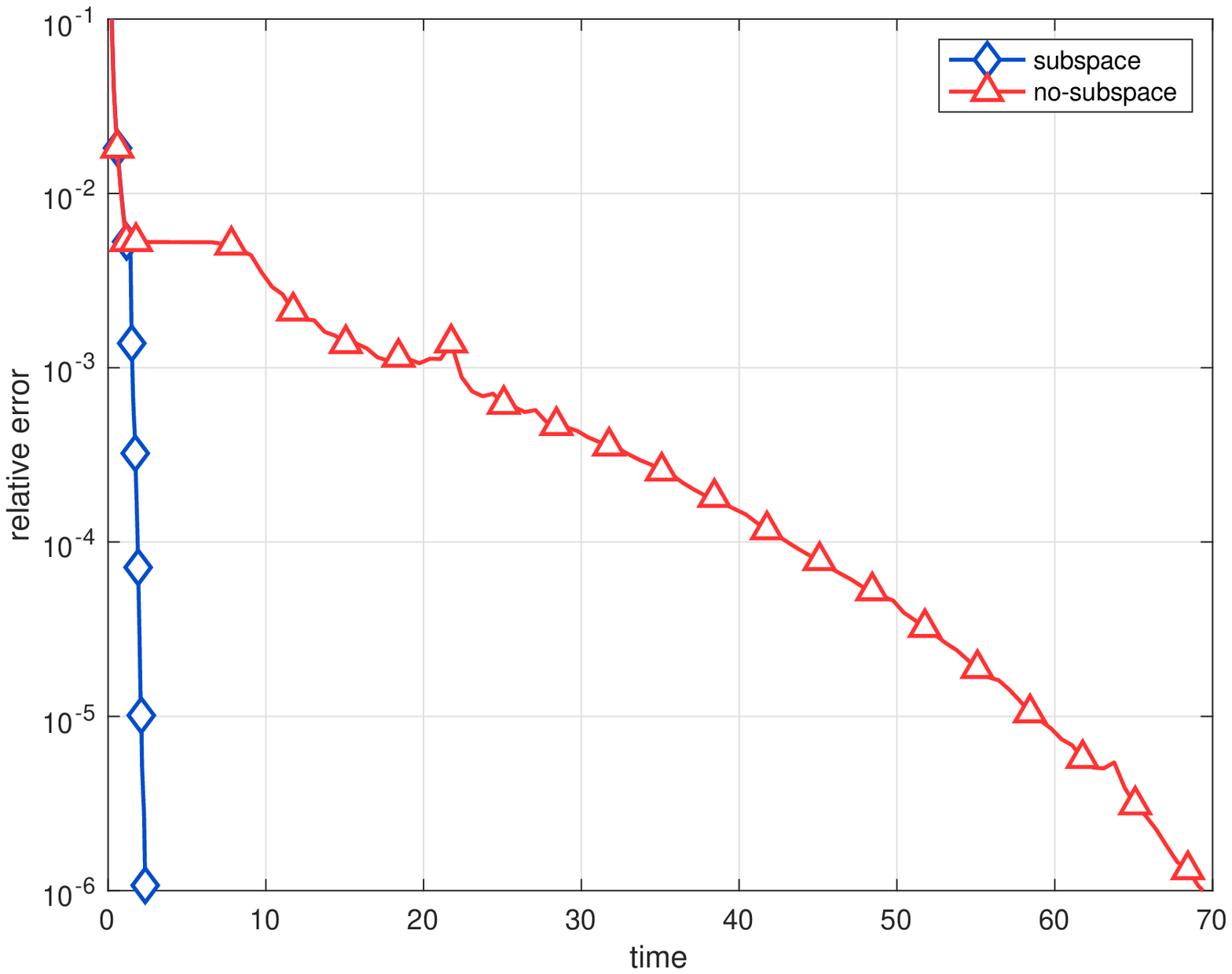}} \\
	\end{tabular}
	\caption{Performance with and without the subspace strategy while other parameters are fixed.}
	\label{figure:subpro}
\end{figure}


\subsection{Numerical Comparison} Next, we describe the implementational details of the different algorithms utilized in the numerical comparison.
\begin{itemize}[leftmargin=4ex]
\item LIBLINEAR, \cite{REF08a}, is a well-known software package for solving logistic regression problems. LIBLINEAR is based on the improved GLMNET method which is a proximal Newton-type algorithm. The source code is available at the LIBLINEAR website \cite{webliblinear}. We use the default parameters. \vspace{.5ex} 
\item \textrm{OW-LQN}, \cite{AndGao07}, is an extension of the L-BFGS method for solving the nonsmooth $\ell_1$-regularized optimization problem \eqref{eq:LR}. The original source code for OW-LQN uses the matrix market file format. In order to allow a fair comparison, we reimplement OW-LQN within the LIBLINEAR code framework. 
\item \textrm{Prox-SVRG}, \cite{XiaZha14}, is the vanilla stochastic proximal gradient method with variance reduction. 
The mini-batch size of the stochastic gradient is set to $1$ and the number of inner iterations is set to $\lfloor 1.5N\rfloor$. The step size is chosen as $1/L_f$ and we use a \textit{lazy update mechanism} for $\ell_1$-problems as suggested in \cite{konevcny2016mini,liu2018fast}. The code is implemented within the LIBLINEAR framework.\vspace{.5ex} 
%
%
\item \textrm{SEQN-VR} denotes the method SEQN-VR-1 that was presented in section \ref{subsec:comparison-variants}. We also incorporate the subspace strategy if the number of features $n$ of a dataset is larger than $N$. The code is implemented based on the LIBLINEAR framework.
\end{itemize}
It is worth mentioning that LIBLINEAR and OW-LQN are  specifically designed for $\ell_1$-problems, while Prox-SVRG and the proposed SEQN-VR method can be applied to other nonsmooth and nonconvex problems.

A summary of our computational results is shown in Table \ref{testpart3}. We report the cpu-time, the number of epochs, the total number of nonzeros and the relative error of the final output of each method. Here, one epoch denotes one full pass through the data and the relative error is defined as
 \begin{equation*}
\text{rel\_err}:=(\psi(x) - \psi^*)/\max\{1,|\psi^*|\},
\end{equation*} 
where the optimal function value $\psi^*$ is obtained by LIBLINEAR with $\mathtt{tol} = 10^{-10}$. 

The different methods terminate when the relative error is smaller than $10^{-6}$ or the number of iterations exceeds a certain threshold. As shown in Table \ref{testpart3}, SEQN-VR achieves the best performance in ``cpu-time'' and ``epochs'' for almost all datasets. It is especially efficient for datasets with a small number of features such as $\mathtt{higgs}$ and $\mathtt{susy}$. In the large-scale problems $\mathtt{kdda}$ and $\mathtt{url}$, the solution generated by SEQN-VR seems to be denser than the solution of LIBLINEAR and OW-LQN, but SEQN-VR still converges significantly faster in terms of cpu-time and epochs. 

In the Figures \ref{figure:relative_time} and \ref{figure:accu_time}, a plot of the change of the relative error and the testing accuracy with respect to the cpu-time is depicted. Again SEQN-VR outperforms most of the other methods and transition to fast local convergence can be observed for most of the tested datasets.
LIBLINEAR and OW-LQN seem to perform worse than the stochastic methods on the large-scale datasets. In Figure \ref{figure:accu_time}, the relationships between the testing accuracy and the training time is illustrated. In summary, the stochastic methods perform better than the deterministic methods on large-scale problems and the additional higher order information used in SEQN-VR allows to accelerate and improve the first order scheme Prox-SVRG.  
%
\begin{table}
\small
\centering
\begin{tabular}{|c| C{0.7cm} C{0.78cm} C{1.08cm} c| C{0.7cm} C{0.78cm} C{1.18cm} c |}
\hline
      solver &  rel\_err &     epochs&       time&       nnz&  rel\_err &     epochs&      time&       nnz \\ \hline 
& \multicolumn{4}{c|}{$\mathtt{cina}$}  &\multicolumn{4}{c|}{$\mathtt{url}$} \\ \hline 
 LIBLINEAR &     1e-07 &      4735 &       2.4\!\;s &       111 &     6e-07  &     55499 &   19936.4\!\;s &     19447 \\   
   SEQN-VR &     9e-07 &       \celliv{129} & \celliv{0.3\!\;s} &       112  &   1e-06 &    \celliv{5331} &   \celliv{11620.8\!\;s} &     21465  \\   
 Prox-SVRG &     1e-06 &       166 &       1.0\!\;s &       111 &     5e-06  &     25001 &   67483.6\!\;s &     22037 \\   
    OW-LQN &     1e-06 &      3481 &       7.5\!\;s &       112 &     1e-04  &     58186 &   82098.6\!\;s &     21203 \\  \hline 
& \multicolumn{4}{c|}{$\mathtt{a9a}$}  &\multicolumn{4}{c|}{$\mathtt{susy}$} \\ \hline 
 LIBLINEAR &     8e-08 &      7239 &       4.7\!\;s &        96 &     4e-07  &     10575 &    1615.4\!\;s &        18 \\   
   SEQN-VR &     9e-07 &        51 &       \celliv{0.1\!\;s} &        99 &     4e-07  &        \celliv{20} &      \celliv{13.2\!\;s} &        18 \\   
 Prox-SVRG &     4e-07 &        \celliv{41} &       0.3\!\;s &       104 &     4e-07  &        41 &      49.5\!\;s &        18 \\   
    OW-LQN &     1e-06 &      3516 &      11.6\!\;s &        99 &     1e-06  &     10409 &    6853.4\!\;s &        18 \\  \hline 
& \multicolumn{4}{c|}{$\mathtt{ijcnn1}$}  &\multicolumn{4}{c|}{$\mathtt{higgs}$} \\ \hline 
 LIBLINEAR &     5e-07 &       149 &       0.2\!\;s &        22 &     3e-07  &      1521 &     756.0\!\;s &        27 \\   
   SEQN-VR &     3e-07 &        \celliv{17} &       \celliv{0.1\!\;s} &        22 &     1e-07  &        \celliv{14} &      \celliv{25.6\!\;s} &        28 \\   
 Prox-SVRG &     5e-07 &        51 &       0.3\!\;s &        22 &     9e-07  &        74 &     264.8\!\;s &        28 \\   
    OW-LQN &     3e-07 &        34 &       0.2\!\;s &        22 &     1e-06  &       852 &    1684.2\!\;s &        28 \\  \hline 
& \multicolumn{4}{c|}{$\mathtt{covtype}$}  &\multicolumn{4}{c|}{$\mathtt{news20}$} \\ \hline 
 LIBLINEAR &     9e-07 &     63509 &     559.2\!\;s &        52 &     1e-07  &        89 &       2.5\!\;s &       417 \\   
   SEQN-VR &     9e-07 &        95 &       \celliv{4.4\!\;s} &        53 &     7e-07  &        \celliv{71} &       \celliv{2.4\!\;s} &       419 \\   
 Prox-SVRG &     6e-07 &        \celliv{86} &       7.7\!\;s &        52 &     8e-07  &        81 &       7.5\!\;s &       418 \\   
    OW-LQN &     2e-06 &     56823 &    3022.6\!\;s &        52 &     1e-06  &       137 &      19.1\!\;s &       417 \\  \hline 
& \multicolumn{4}{c|}{$\mathtt{rcv1}$}  &\multicolumn{4}{c|}{$\mathtt{kdda}$} \\ \hline 
 LIBLINEAR &     5e-07 &        \celliv{57} &       \celliv{0.2\!\;s} &       560 &     9e-07  &     58945 &   77437.0\!\;s &    854569 \\   
   SEQN-VR &     5e-07 &        73 &       1.4\!\;s &       563 &     1e-06  &       \celliv{477} &    \celliv{6043.6\!\;s} &    910031 \\   
 Prox-SVRG &     8e-07 &        76 &       1.7\!\;s &       560 &     1e-06  &      1549 &   11800.6\!\;s &    908249 \\   
    OW-LQN &     9e-07 &        67 &       0.9\!\;s &       560 &     1e-06  &     46632 &  120067.4\!\;s &    905523 \\  \hline 
\end{tabular}
\caption{Numerical results: $\ell_1$-logistic regression. For each dataset, the best performance with respect to number of epochs and cpu-time is shaded with a grey color.}
\label{testpart3}
\end{table}

\begin{figure}
\centering
\setlength{\belowcaptionskip}{-6pt}
\begin{tabular}{cc}
\subfloat[$\mathtt{cina}$]{
\includegraphics[height=3.75cm]{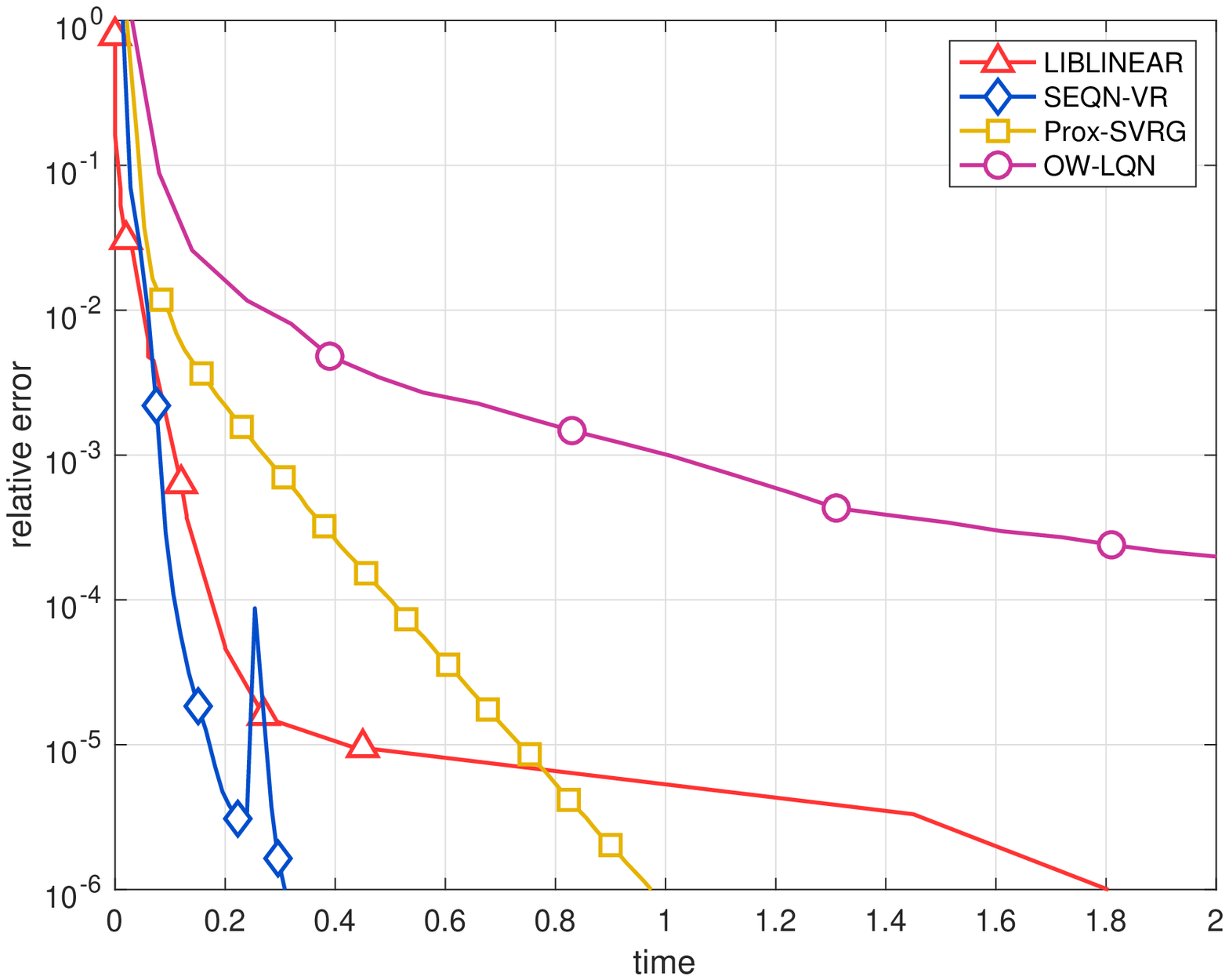}} &
\subfloat[$\mathtt{a9a}$]{
\includegraphics[height=3.75cm]{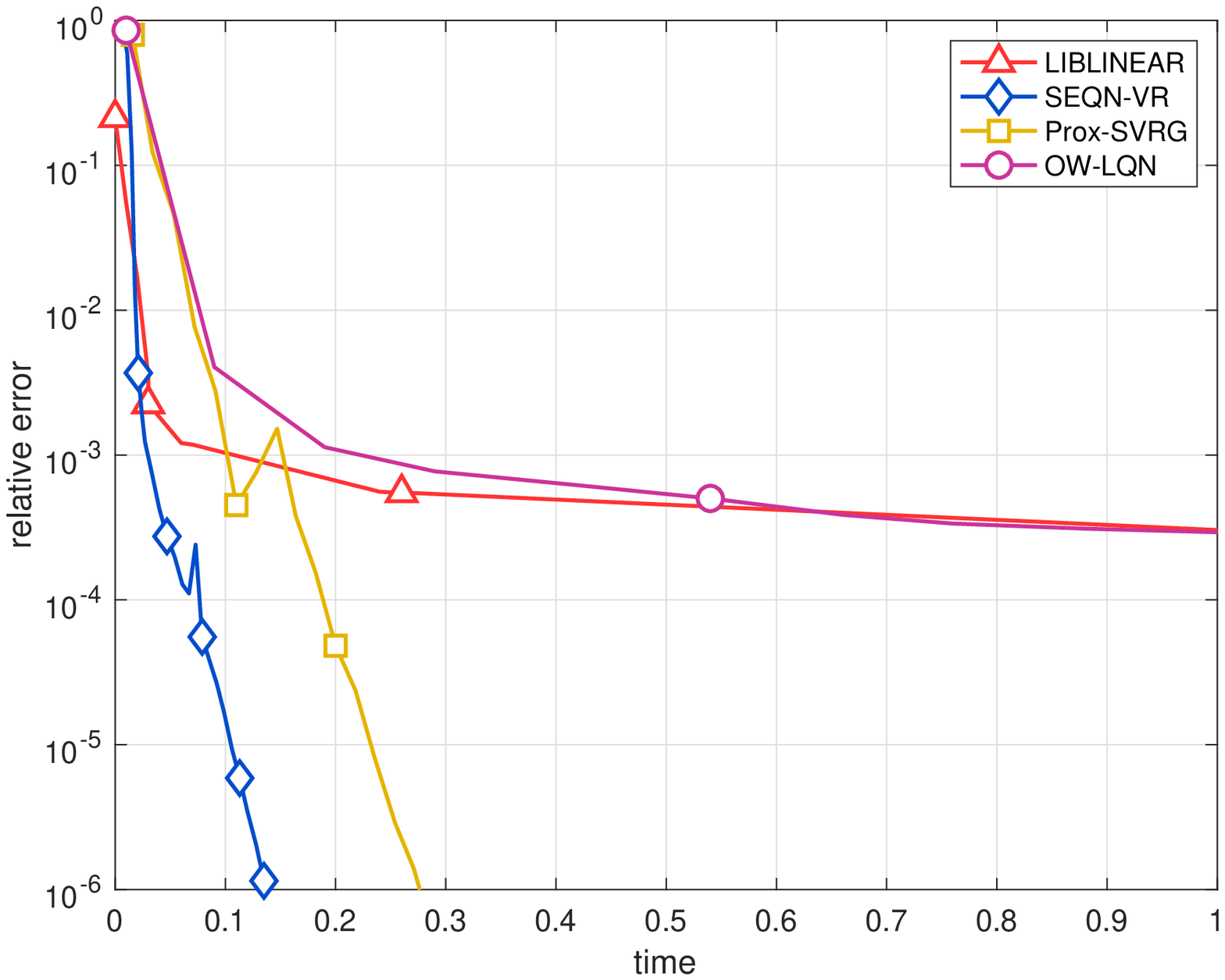}}\\
\subfloat[$\mathtt{ijcnn1}$]{
\includegraphics[height=3.75cm]{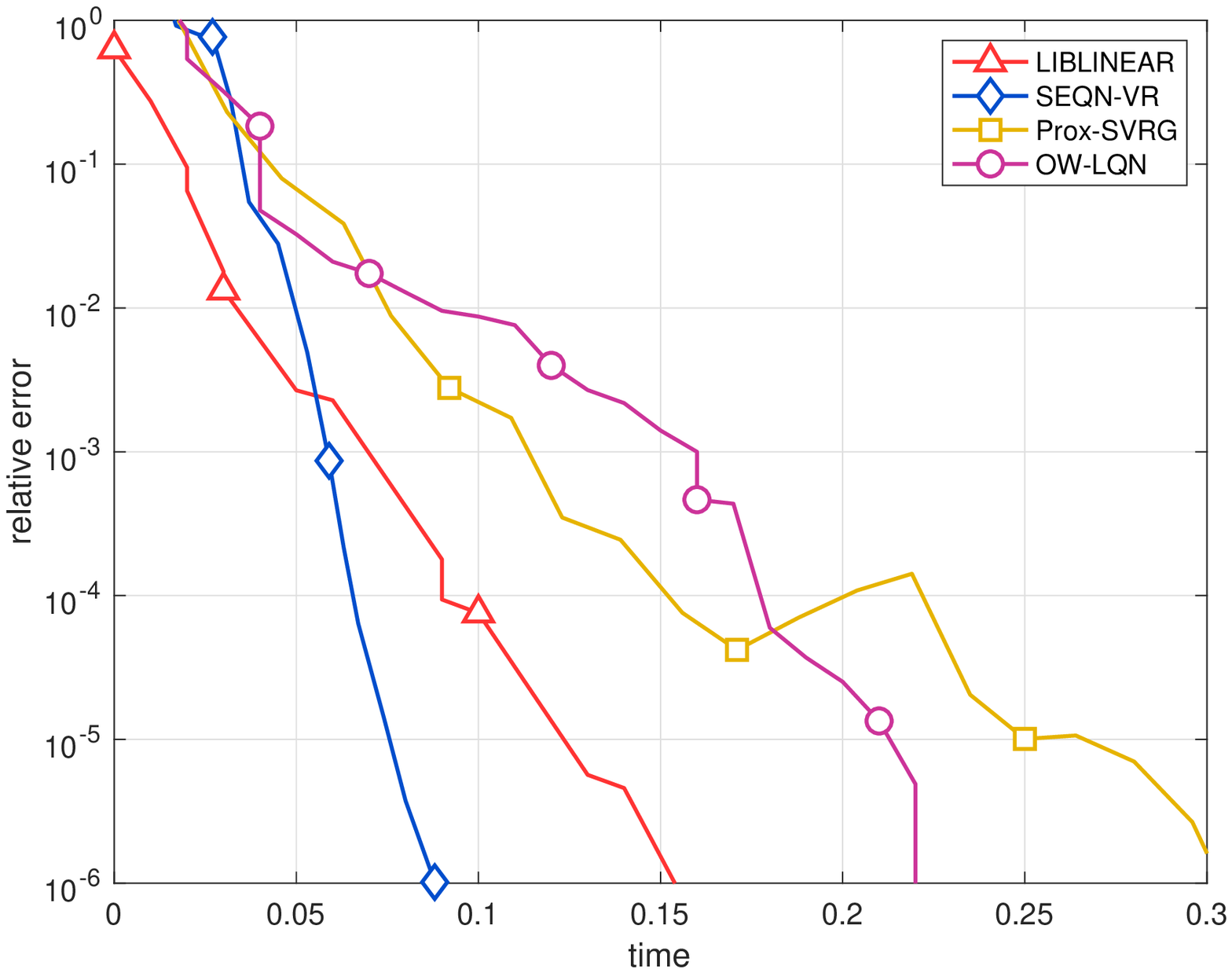}} &
\subfloat[$\mathtt{covtype}$]{
\includegraphics[height=3.75cm]{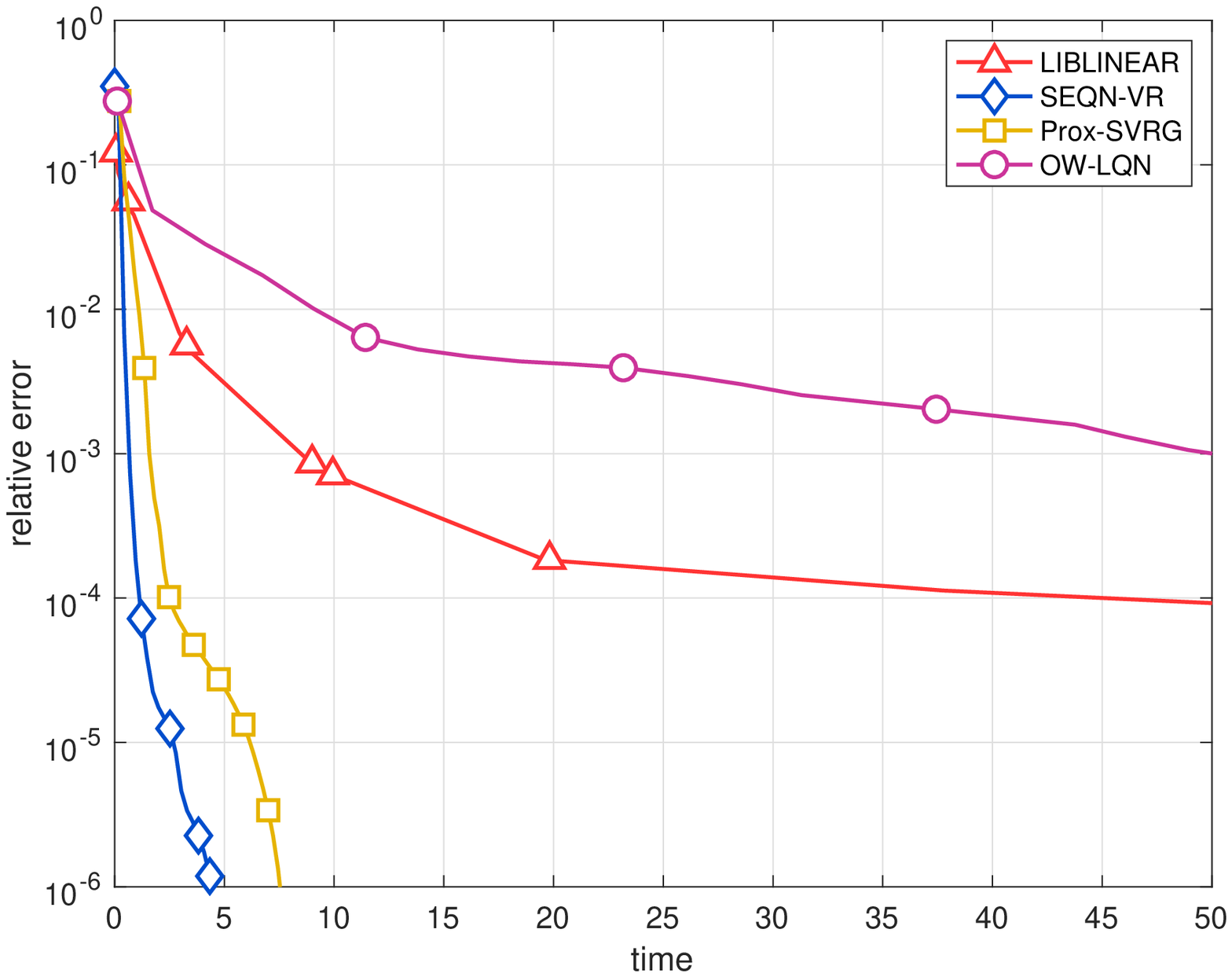}} \\
\subfloat[$\mathtt{url}$]{
\includegraphics[height=3.75cm]{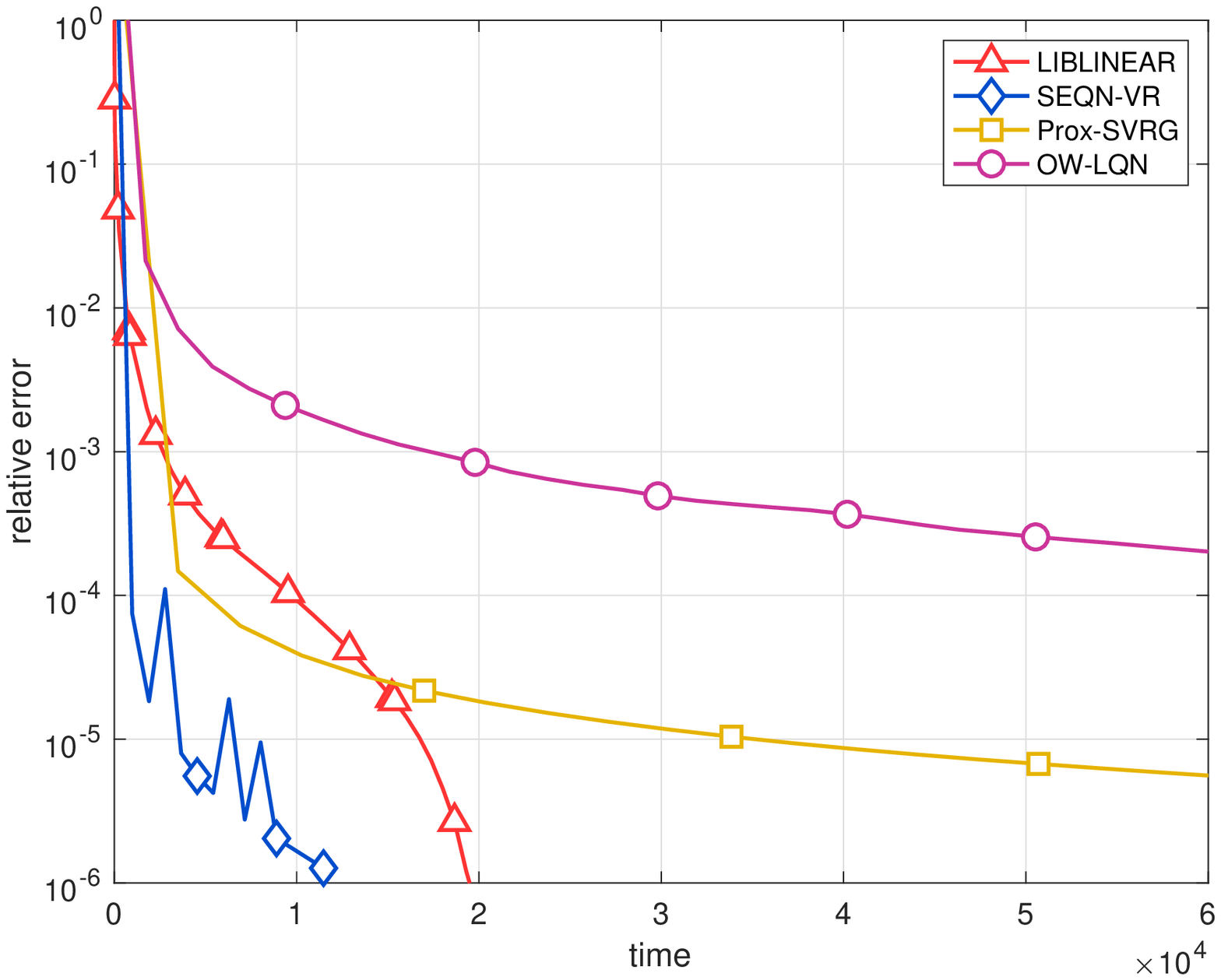}} &
\subfloat[$\mathtt{susy}$]{
\includegraphics[height=3.75cm]{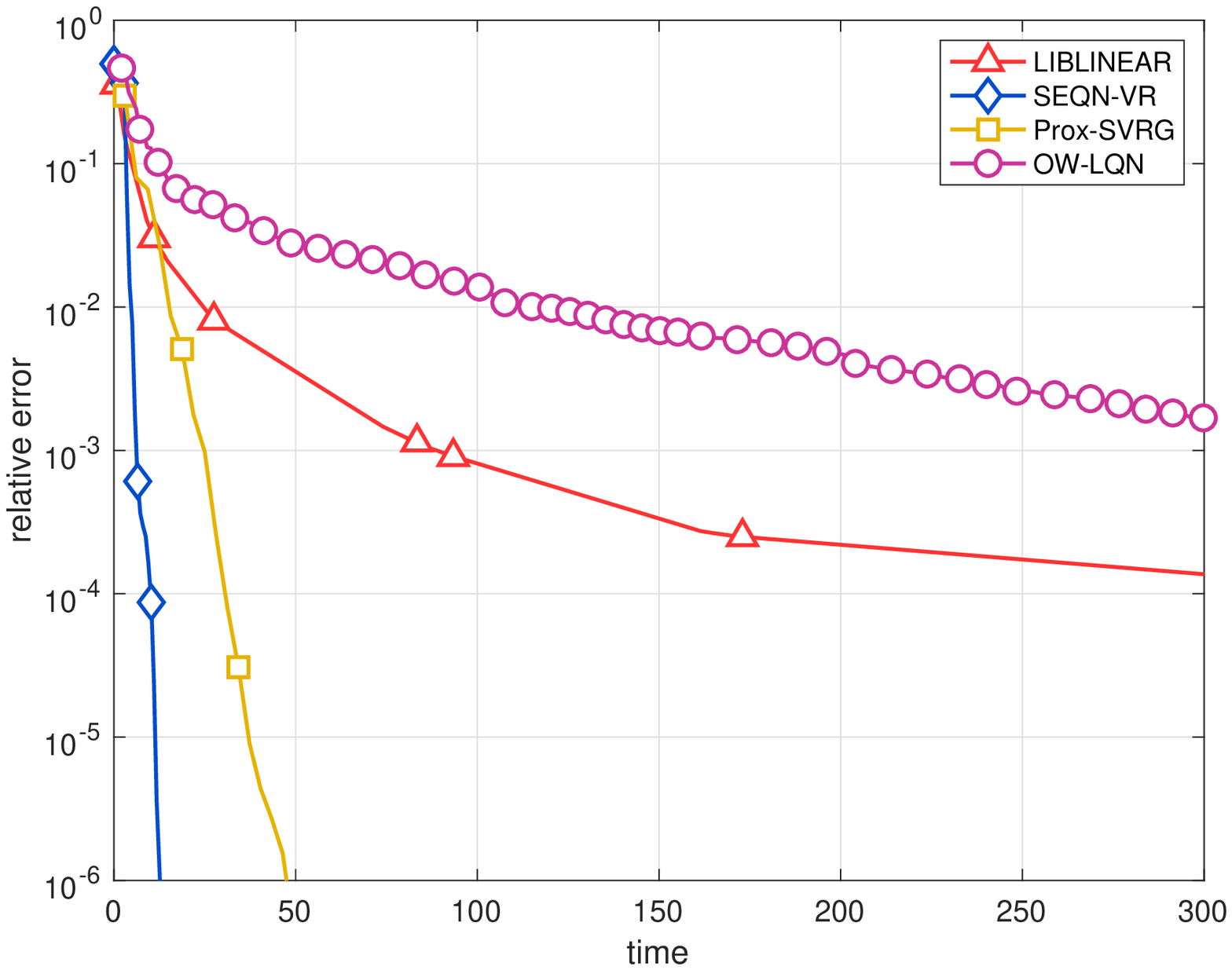}} \\
\subfloat[$\mathtt{higgs}$]{
\includegraphics[height=3.75cm]{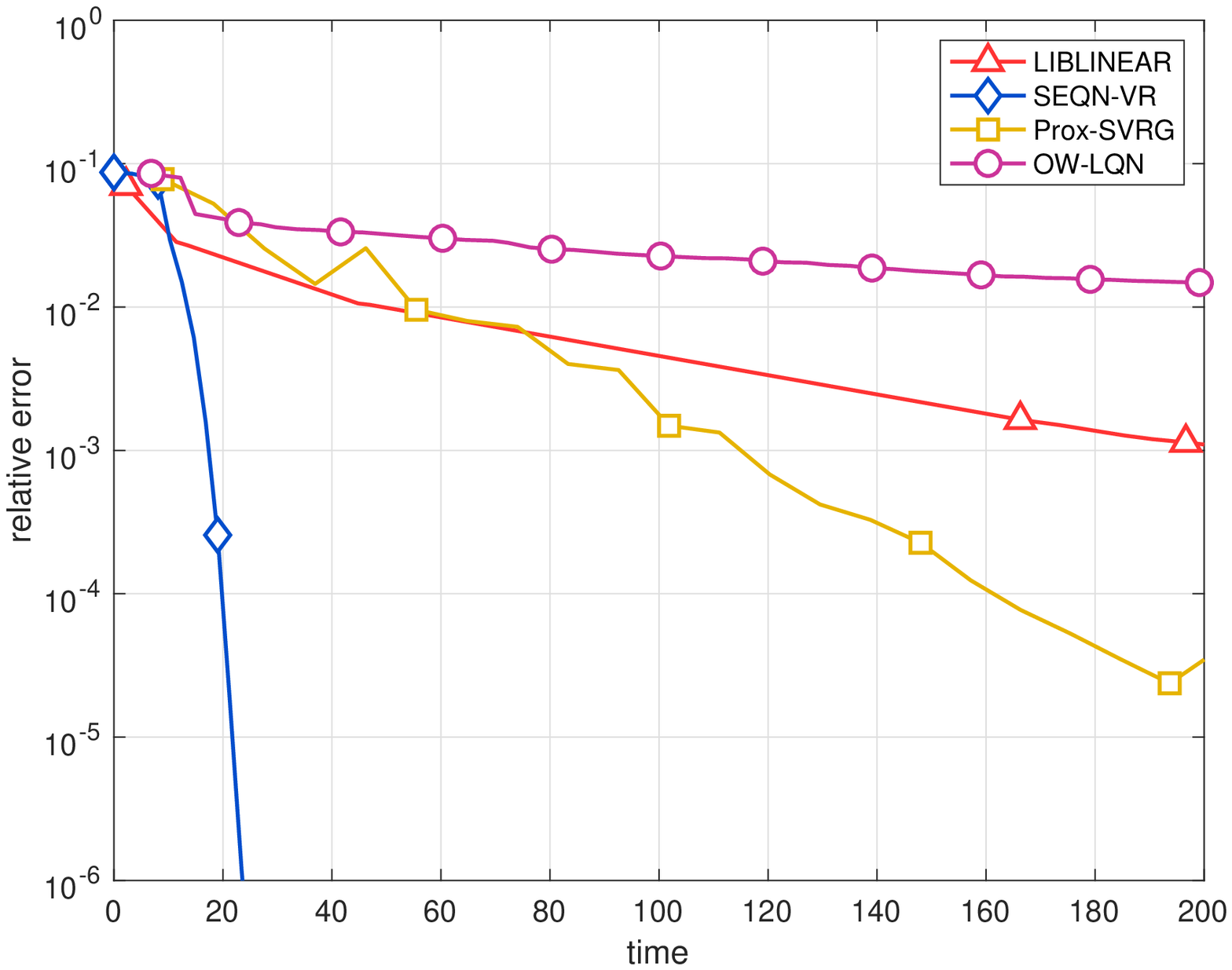}} &
\subfloat[$\mathtt{news20}$]{
\includegraphics[height=3.75cm]{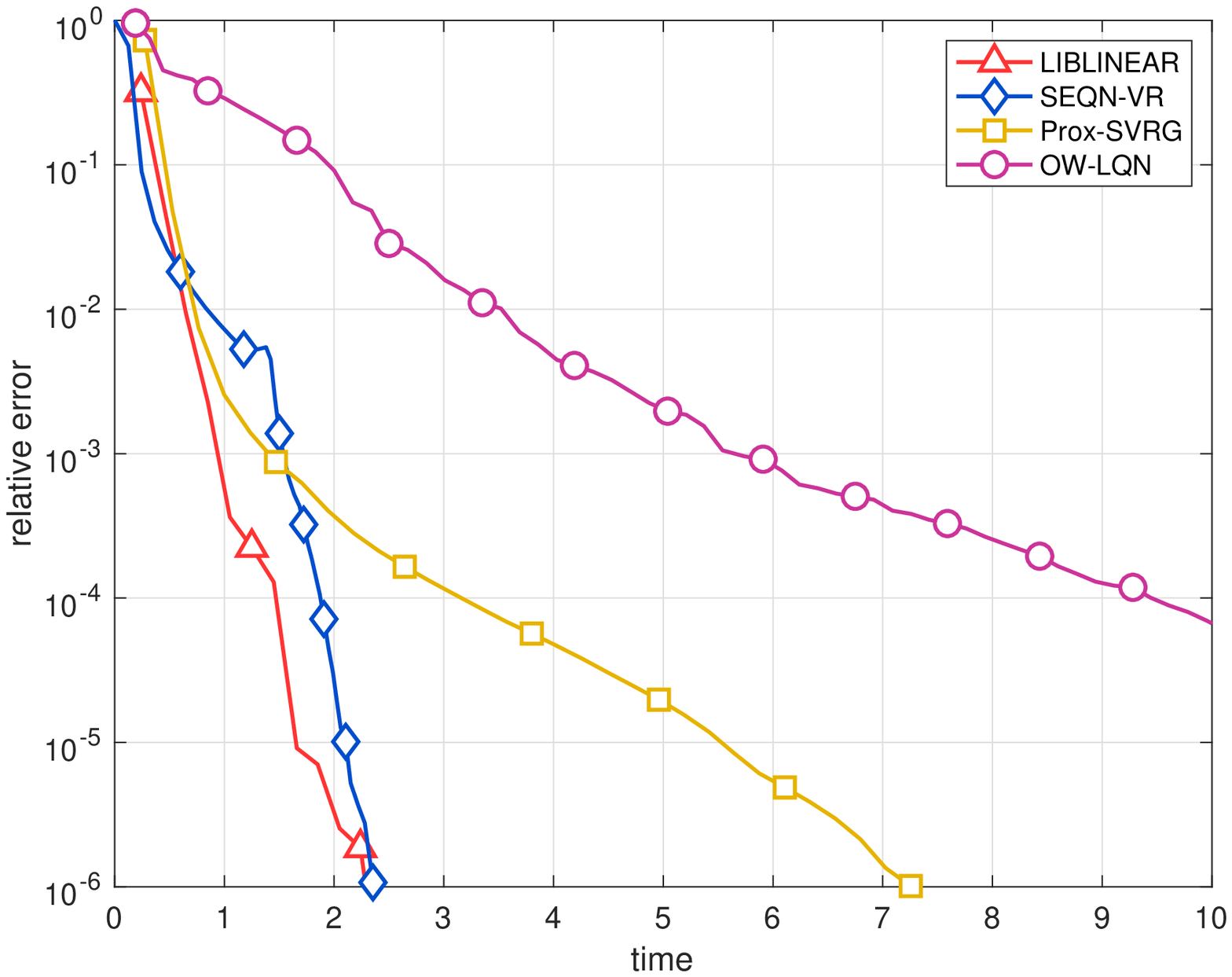}} 
\end{tabular}
\caption{Iteration history of the relative difference to the optimal function values versus the training time on $\ell_1$-logistic regression.}
\label{figure:relative_time}
\end{figure}

\begin{figure}
\centering
\setlength{\belowcaptionskip}{-6pt}
\begin{tabular}{cc}
\subfloat[$\mathtt{cina}$]{
\includegraphics[height=3.75cm]{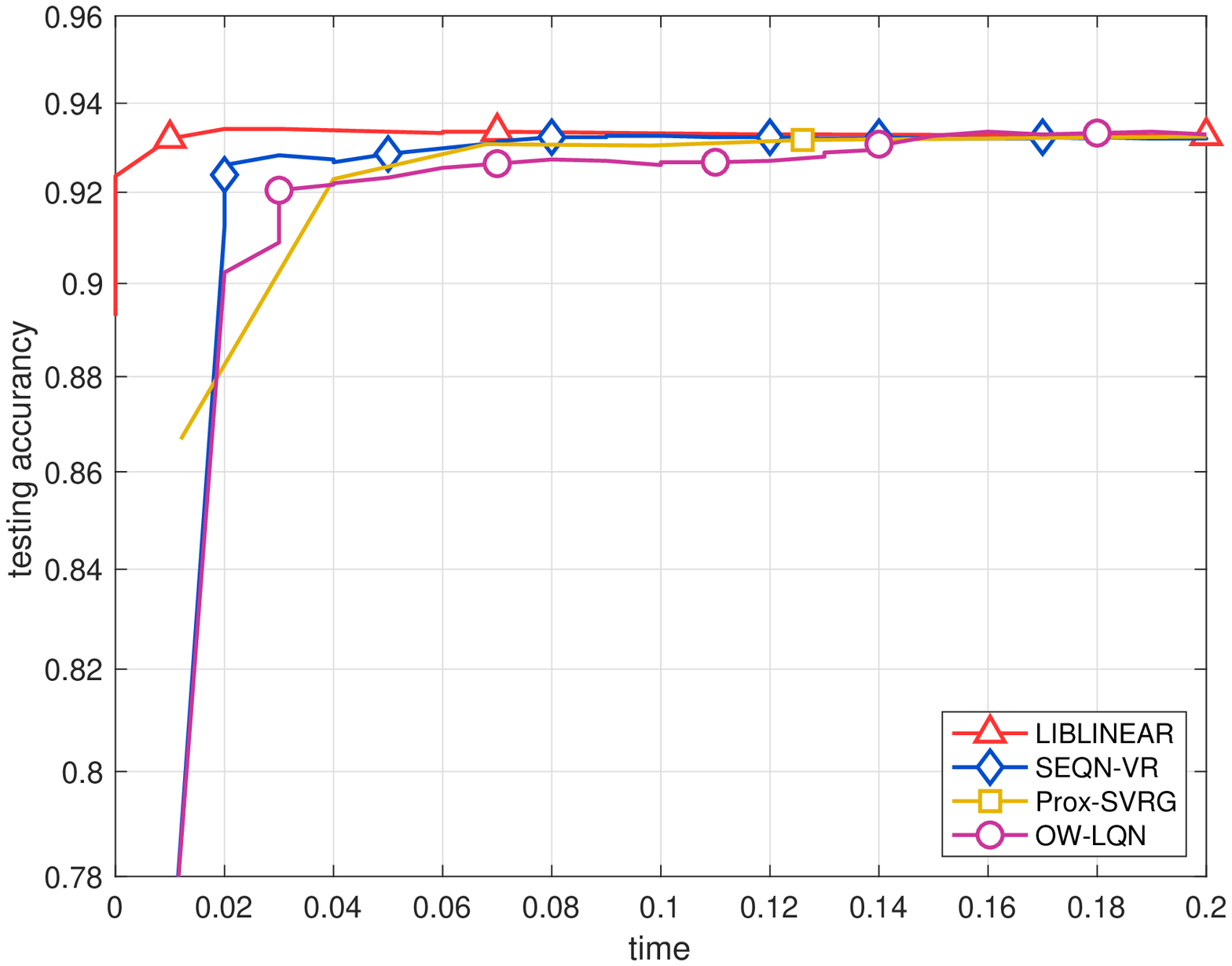}} &
\subfloat[$\mathtt{a9a}$]{
\includegraphics[height=3.75cm]{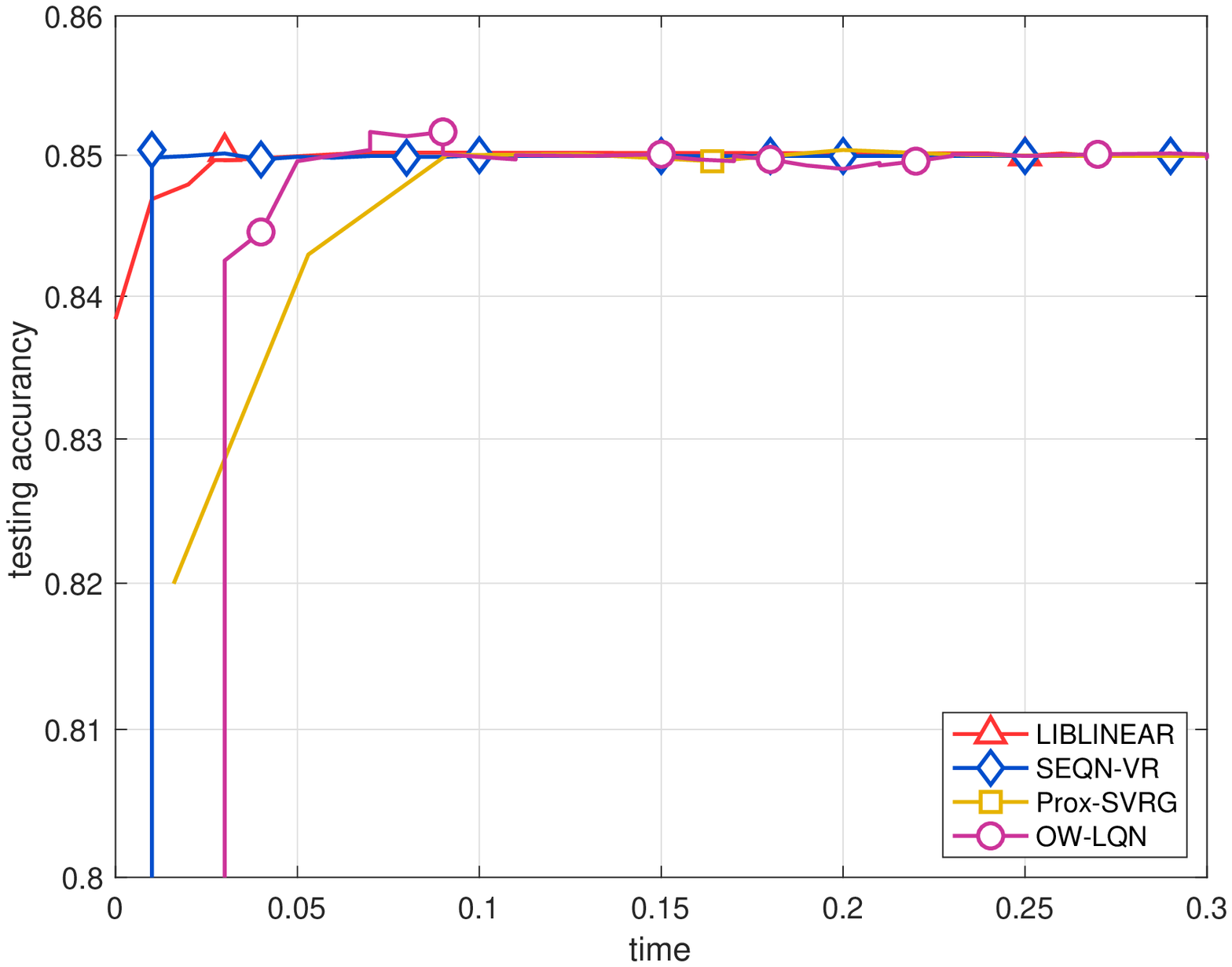}}\\
\subfloat[$\mathtt{ijcnn1}$]{
\includegraphics[height=3.75cm]{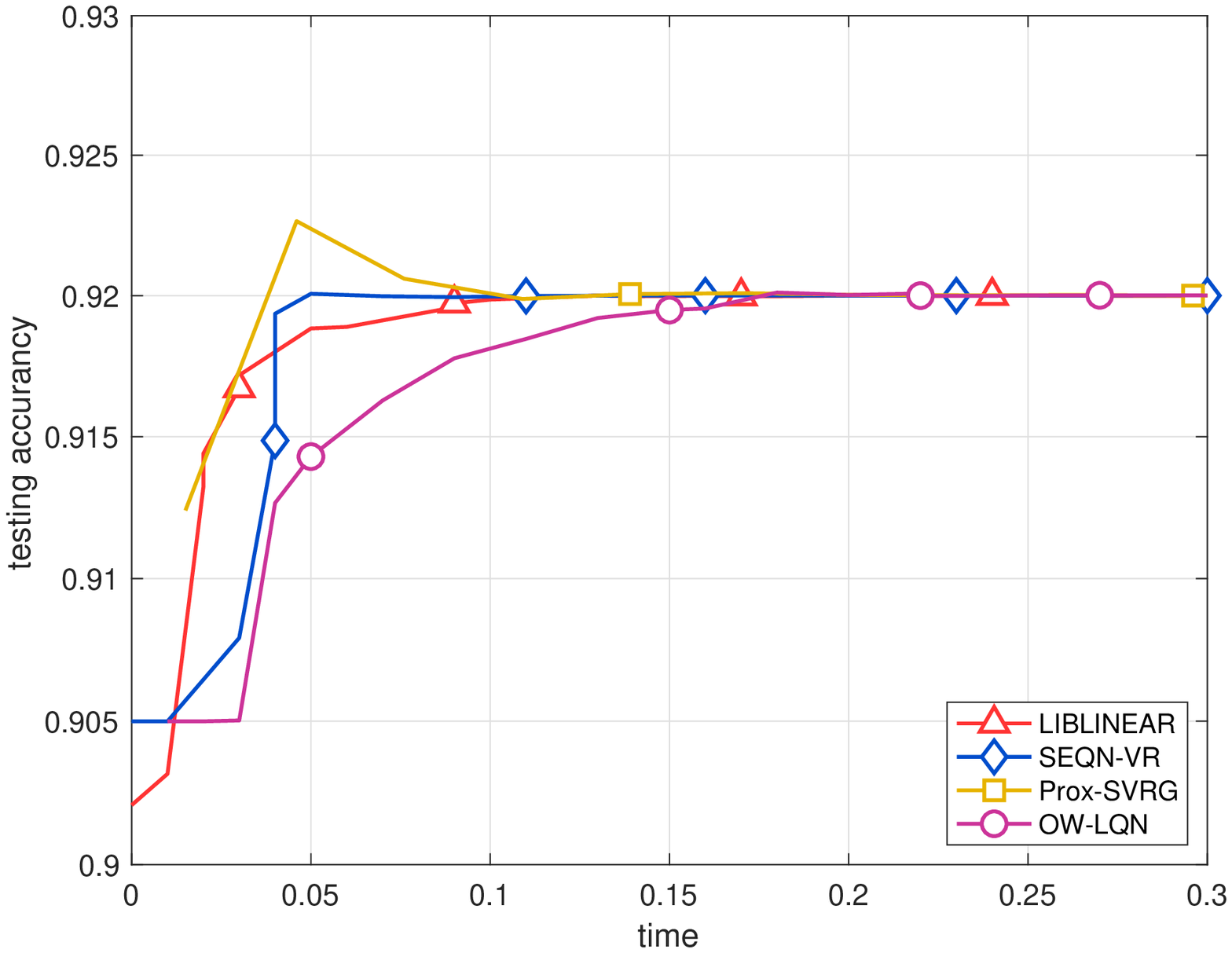}} &
\subfloat[$\mathtt{covtype}$]{
\includegraphics[height=3.75cm]{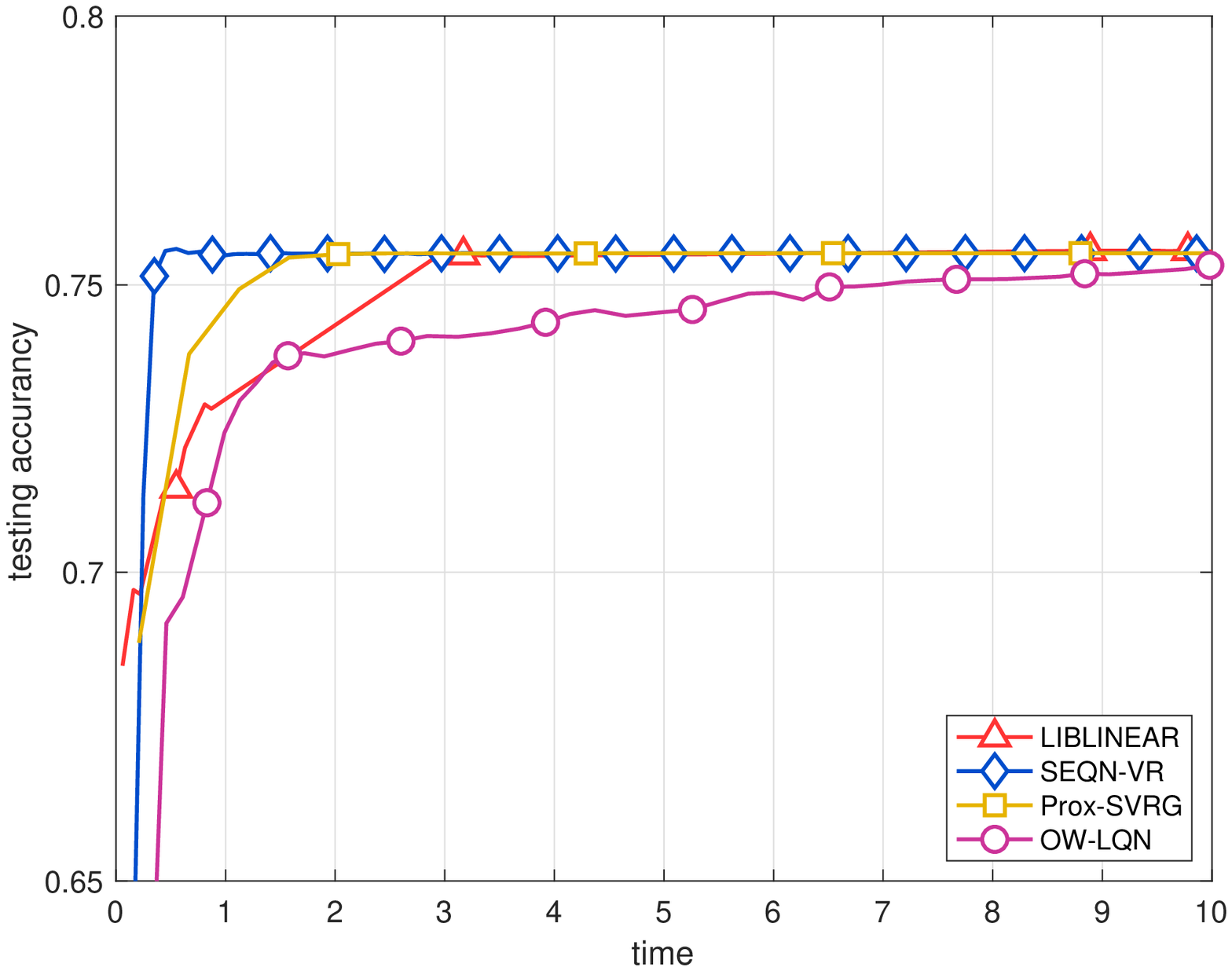}} \\
\subfloat[$\mathtt{url}$]{
\includegraphics[height=3.75cm]{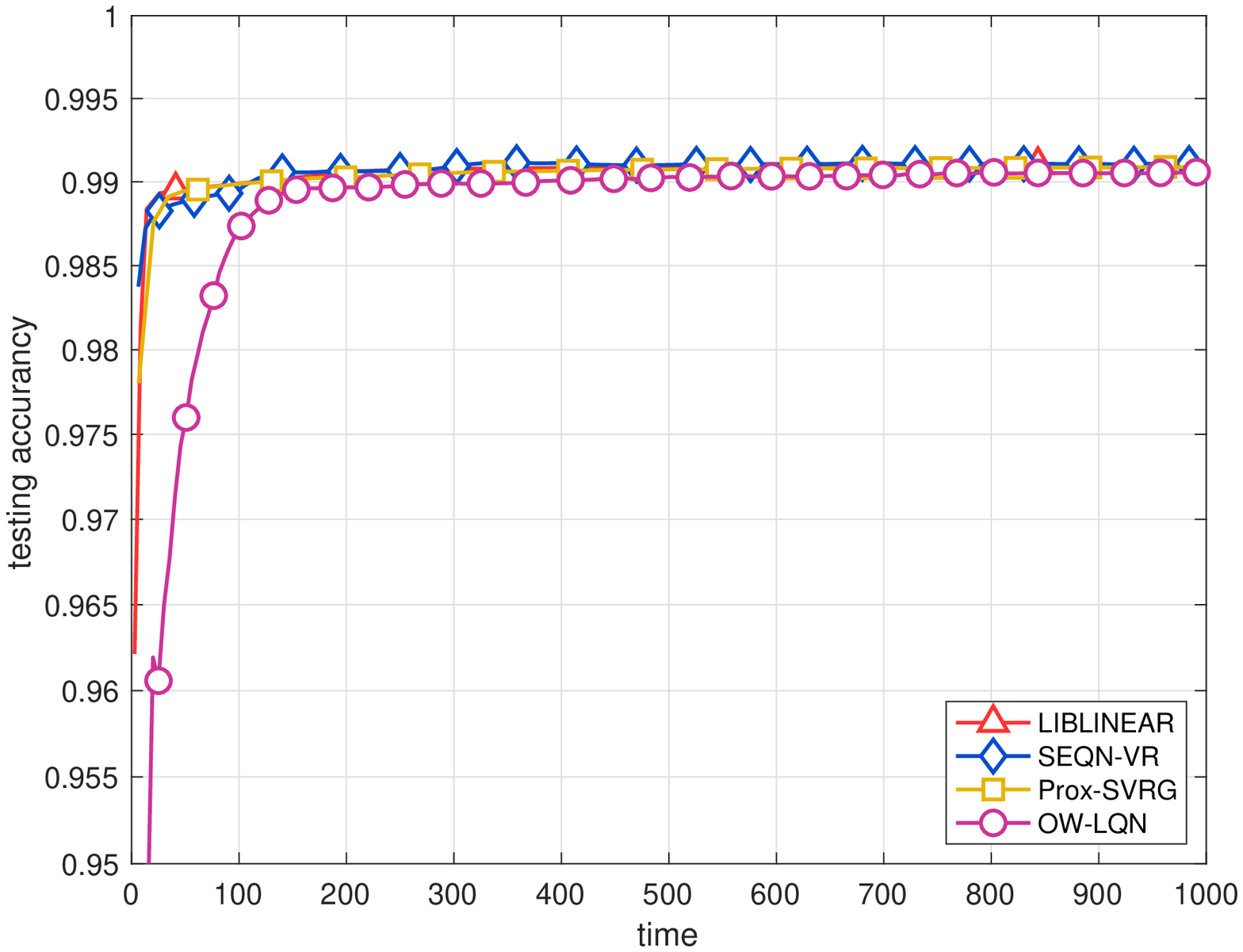}} &
\subfloat[$\mathtt{susy}$]{
\includegraphics[height=3.75cm]{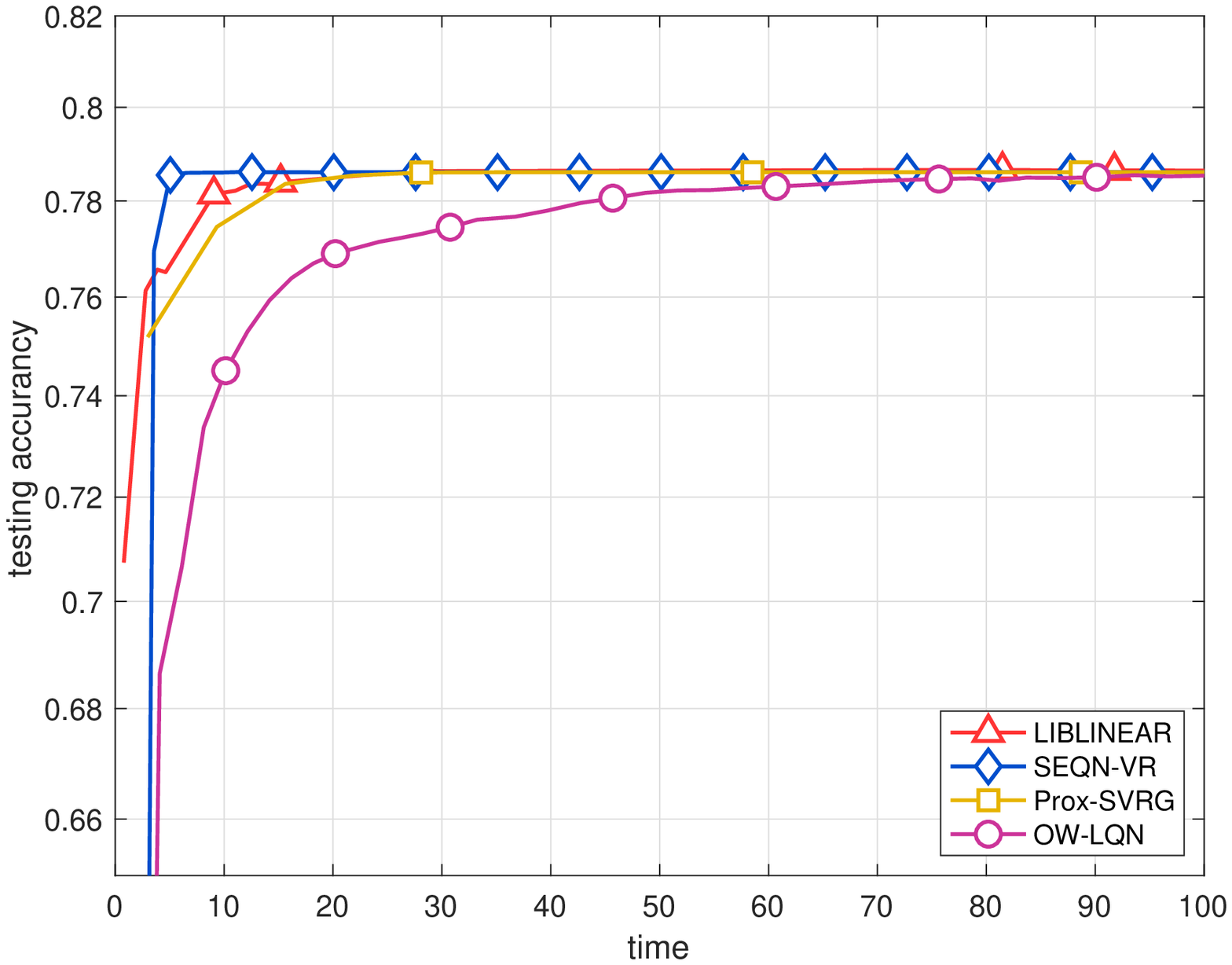}} \\
\subfloat[$\mathtt{higgs}$]{
\includegraphics[height=3.75cm]{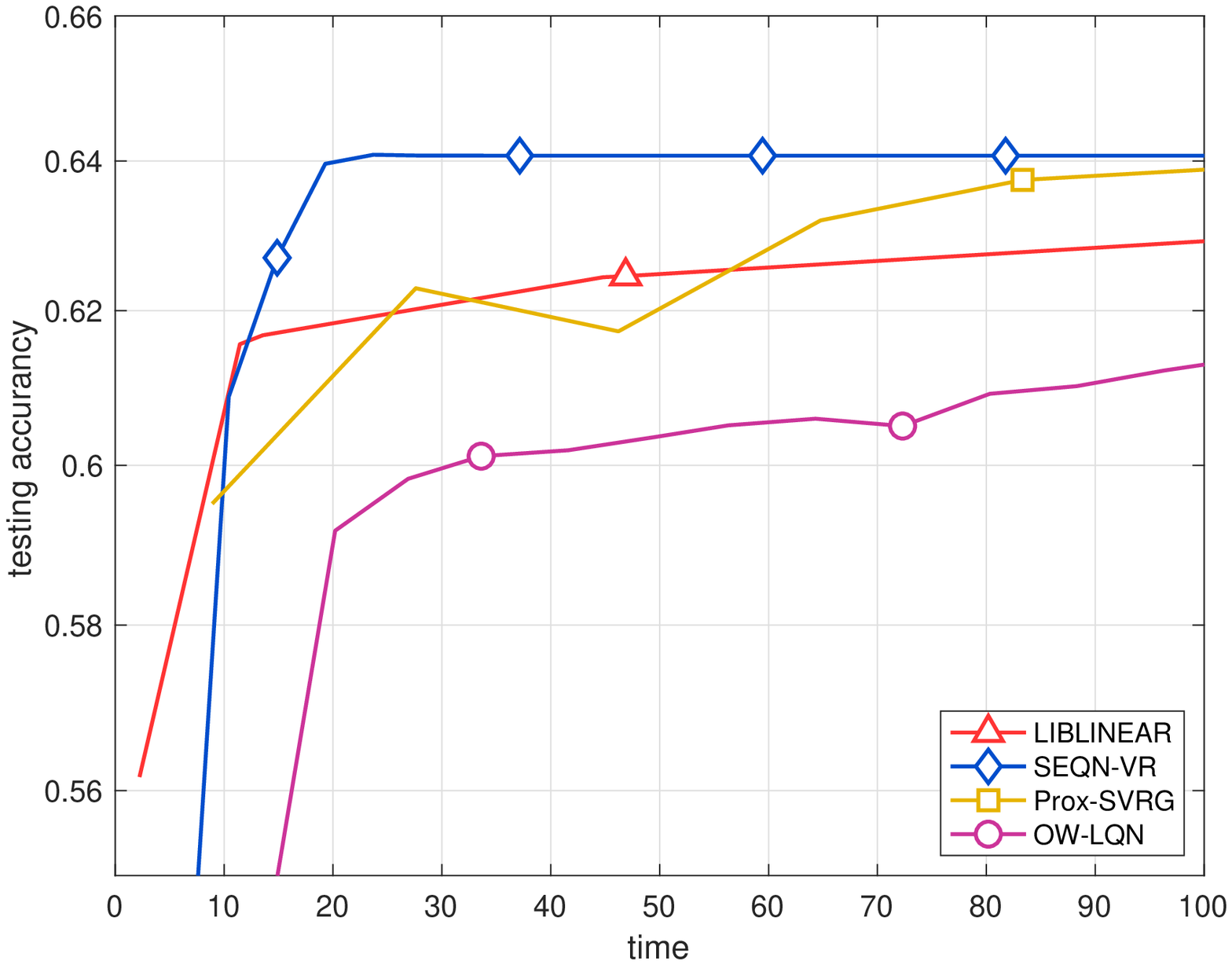}} &
\subfloat[$\mathtt{news20}$]{
\includegraphics[height=3.75cm]{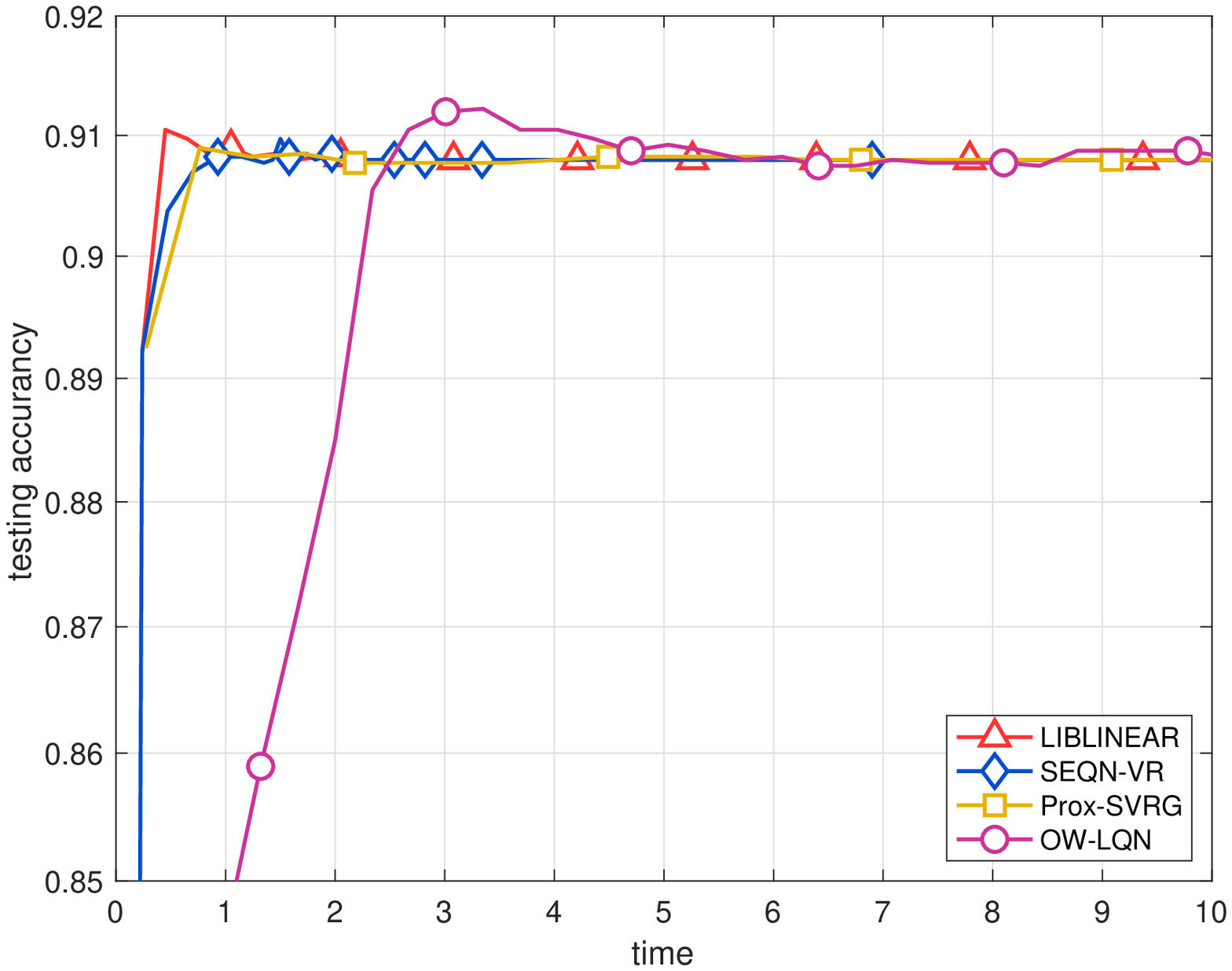}}\\
\end{tabular}
\caption{Plot of the testing accuracy versus the training time on $\ell_1$-logistic regression.}
\label{figure:accu_time}
\end{figure}


\section{Preliminary Numerical Results: Deep Learning}
\vspace{-2ex}
Although deep learning is a very successful methodology for many machine learning applications and tasks, it typically leads to highly challenging optimization problems due to the inherently large number of network parameters. A natural idea is to use the $\ell_1$-regularization to prune neural networks. In this section, we study the performance of SEQN on deep neural networks and present preliminary numerical comparisons.

\subsection{Datasets and Experimental Setting}
We test SEQN on both small and large neural networks. $\mathtt{ConvNet}$ is a small 4-layer network consisting of two convolution layers followed by two fully connected layers. For large neural networks, we consider $\mathtt{VGG}$-$\mathtt{16}$ \cite{simonyan2014very} with batch normalization and $\mathtt{ResNet}$-$\mathtt{18}$ \cite{he2016deep}. They are widely used in computer vision and other applications.
We test our proposed algorithm on the $\mathtt{Cifar}$-$\mathtt{10}$ dataset which consists of 50000 training images and 10000 testing images. The different number of parameters of the neural networks are summarized in Table \ref{table:cifar-num-param}. All the experiments are performed in PyTorch \cite{paszke2017automatic}. As loss function, we use the $\ell_1$-regularized cross entropy function:
\[
\min_{x}~\psi(x) := -\frac{1}{N} \sum_{i=1}^N \log\left(\frac{\exp(h(x,a_i)[b_i])}{\sum_j \exp(h(x,a_i)[j])}\right) + \mu \|x\|_1, \quad \mu > 0,
\]
where all the weight matrices of the neural network are concatenated in one vector $x$, $(a_i,b_i)$, $i \in [N]$, denotes the training dataset, and $h$ is the composition of all matrix operations and activation functions.
  We choose $\mu = 10^{-4}$ for $\mathtt{ConvNet}$ and $\mathtt{VGG}$-$\mathtt{16}$ and $\mu = 2 \cdot 10^{-5}$ for $\mathtt{ResNet}$-$\mathtt{18}$. 

%
\subsection{Numerical Results}
We compare SEQN with Prox-SGD. Both methods utilize the mini-batch-type oracles 
\be \label{eq:exp-oracle} v^k = \nabla f_{\cS_k}(x^k) = \frac{1}{|\cS_k|} \sum_{i \in \cS_k} \nabla f_i(x^k), \quad v^k_+ = v_z^k = \nabla f_{\cS_k}(z^k), \ee
 where $\cS_k$ is again randomly and uniformly sampled from $[N]$. The parameters of Prox-SGD are standard. The batch size is set to $|\cS_k| = 128$
and the initial learning rate of Prox-SGD is determined by a grid search. It is 0.01 for $\mathtt{VGG}$-$\mathtt{16}$ and 0.1 for both $\mathtt{ResNet}$-$\mathtt{18}$ and $\mathtt{ConvNet}$. The learning rate is decreased by a factor of 0.1 after 150 epochs. We implement the initial version of SEQN (Algorithm \ref{alg:seqn}) in this section. We set $\alpha_k = \beta_k = 1$ and $\Lambda_{k,+} = \lambda_{k}^{-1}I$, $\Lambda_k$ =  0.5$\cdot\lambda_k^{-1} I$ where $\lambda_k$ is the learning rate which is adjusted as in Prox-SGD. We also adopt the batch progressive strategy and utilize Prox-SGD as a warm start for SEQN. The parameter $\mu$ is set to a larger value in the initial stage and is gradually decreased until it matches the predefined values. 
We run 200 epochs in each experiment and shuffle the samples after every complete pass over the full dataset. Here, the generation of a sample set $\cS_k$ increases the number of epochs by $|\cS_k| /N$. Let us also notice that the computational costs of a single iteration of SEQN are typically twice as high as the costs of an iteration of Prox-SGD (if the same mini-batch size is used) due to the additional evaluation of $v^k_+ = v^k_z$ in the quasi-Newton updates. We report the training accuracy, training loss, and testing accuracy for the three different network architectures in Figure \ref{test_neural}.

\begin{table}[t]
\centering
\begin{tabular}{|c|c|}
\hline
 & $\mathtt{Cifar}$-$\mathtt{10}$ \\
 \hline
\# Total parametes of $\mathtt{ConvNet}$ &413882\\
\# Total parametes of $\mathtt{VGG}$-$\mathtt{16}$ &15253578\\
\# Total parameters of $\mathtt{ResNet}$-$\mathtt{18}$ &11173962\\
\hline
\end{tabular}
\caption{Total numbers of the neural network parameters for $\mathtt{Cifar}$-$\mathtt{10}$}
\label{table:cifar-num-param}
\end{table}

\begin{table}[t]
\small
\centering
\begin{tabular}{|c|cccc|cccc|}
\hline
\multicolumn{9}{|c|}{$\mathtt{VGG}$-$\mathtt{16}$}\\
\hline
 {}&\multicolumn{4}{c|}{SEQN} &\multicolumn{4}{c|}{Prox-SGD} \\
 \hline
Epochs &50& 100& 150& 200&50& 100& 150& 200\\
Testing acc. &91.23& 91.21& 91.51& 91.82 &86.98& 89.44& 89.37& 92.17\\
Train. loss &0.92& 0.80& 0.74& 0.73 &1.56& 0.97& 0.85& 0.61\\
Train. acc. &96.08& 97.77& 98.87& 99.12 &90.75& 93.89& 95.30& 99.57\\
Train. time &865.8& 1771.5& 2516.9& 3197.0 &830.1& 1659.9& 2489.2& 3318.2\\
\hline 
Best test. acc.&\multicolumn{4}{c|}{91.98} &\multicolumn{4}{c|}{92.52} \\
\hline
\end{tabular}
\caption{Summary of the computational results on $\mathtt{VGG}$-$\mathtt{16}$}
\label{vggnet-table}
\end{table}
\begin{table}
\small
\centering
\begin{tabular}{|c|cccc|cccc|}
\hline
\multicolumn{9}{|c|}{$\mathtt{ResNet}$-$\mathtt{18}$}\\
\hline
&\multicolumn{4}{c|}{SEQN} &\multicolumn{4}{c|}{Prox-SGD} \\
 \hline
Epochs &70& 100& 140& 200&70& 100& 140& 200\\
Testing acc. &92.14& 92.35& 92.80& 92.79 &90.80& 91.21& 92.03& 93.48\\
Train. loss &0.32& 0.28& 0.25& 0.24 &0.50& 0.47& 0.43& 0.30\\
Train. acc. &98.87& 99.27& 99.85& 99.91 &97.10& 97.42& 98.19& 99.99\\
Train. time &2562.6& 3091.8& 3660.2& 4438.4 &1631.3& 2326.0& 3261.3& 4658.5\\\hline 
Best test. acc.&\multicolumn{4}{c|}{93.02} &\multicolumn{4}{c|}{93.65} \\
\hline
\end{tabular}
\caption{Summary of the computational results on $\mathtt{ResNet}$-$\mathtt{18}$}
\label{Resnet-table}
\end{table}

\begin{figure}[t]
\centering
\setlength{\belowcaptionskip}{-6pt}
\subfloat[$\mathtt{ConvNet}$: Train. accuracy]{\includegraphics[width=3.8cm,trim={0.5cm 0 1cm 1cm},clip]{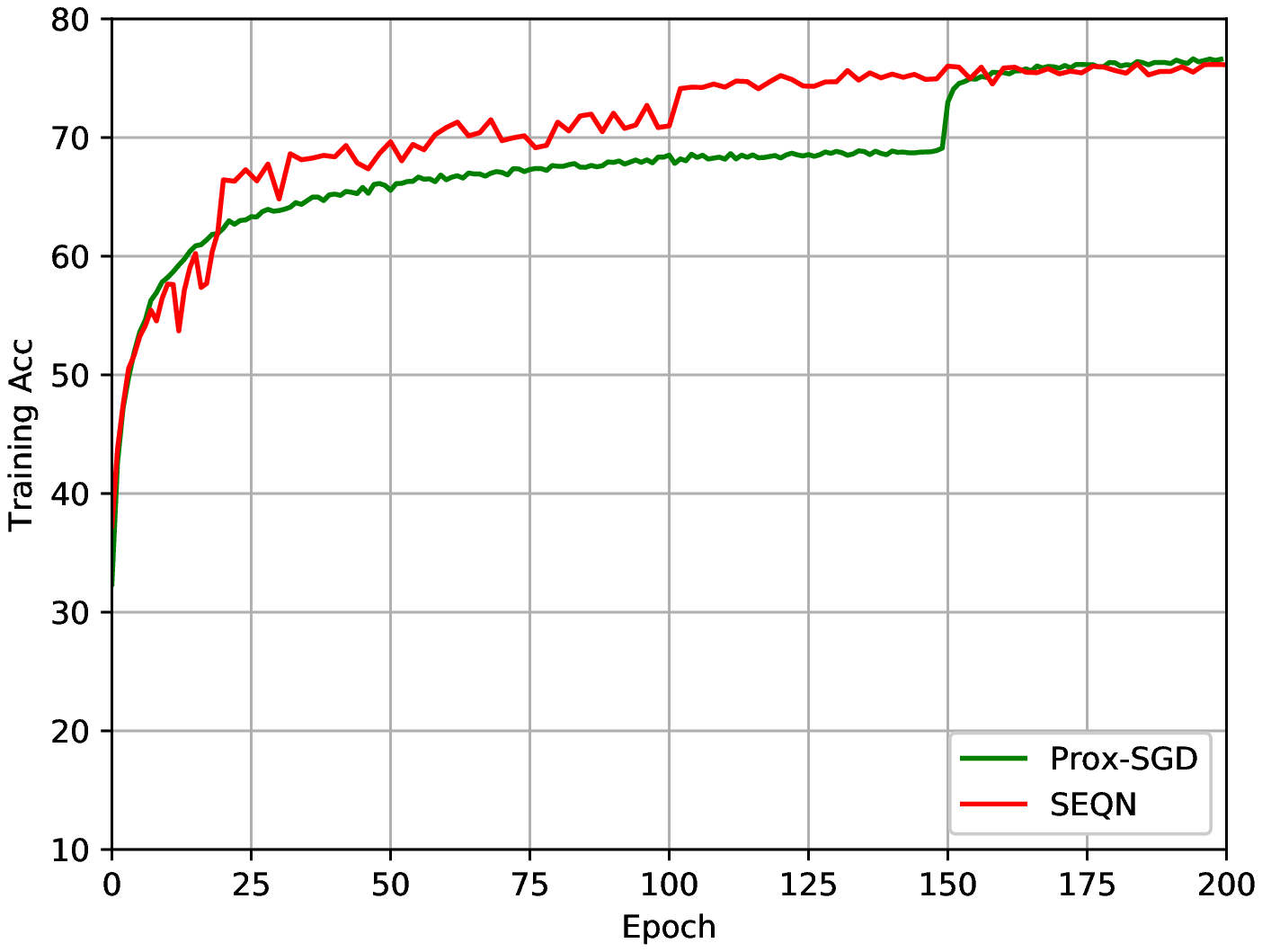}} \hspace{.5ex}
\subfloat[$\mathtt{VGG}$-$\mathtt{16}$: Train. accuracy]{\includegraphics[width=3.8cm,trim={0.5cm 0 1cm 1cm},clip]{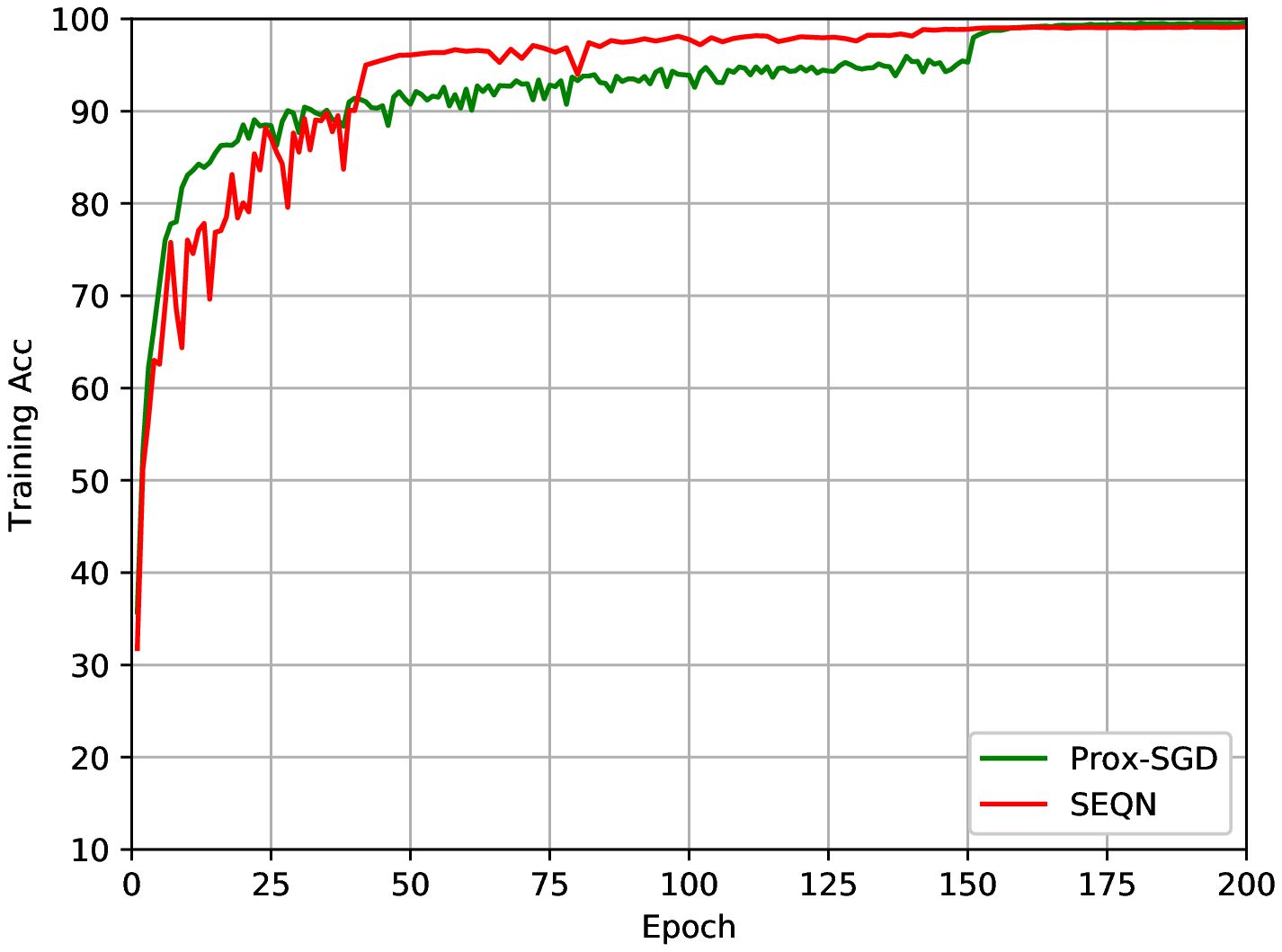}} \hspace{.5ex} 
\subfloat[$\mathtt{ResNet}$-$\mathtt{18}$: Train. accuracy]{\includegraphics[width=3.8cm,trim={0.5cm 0 1cm 1cm},clip]{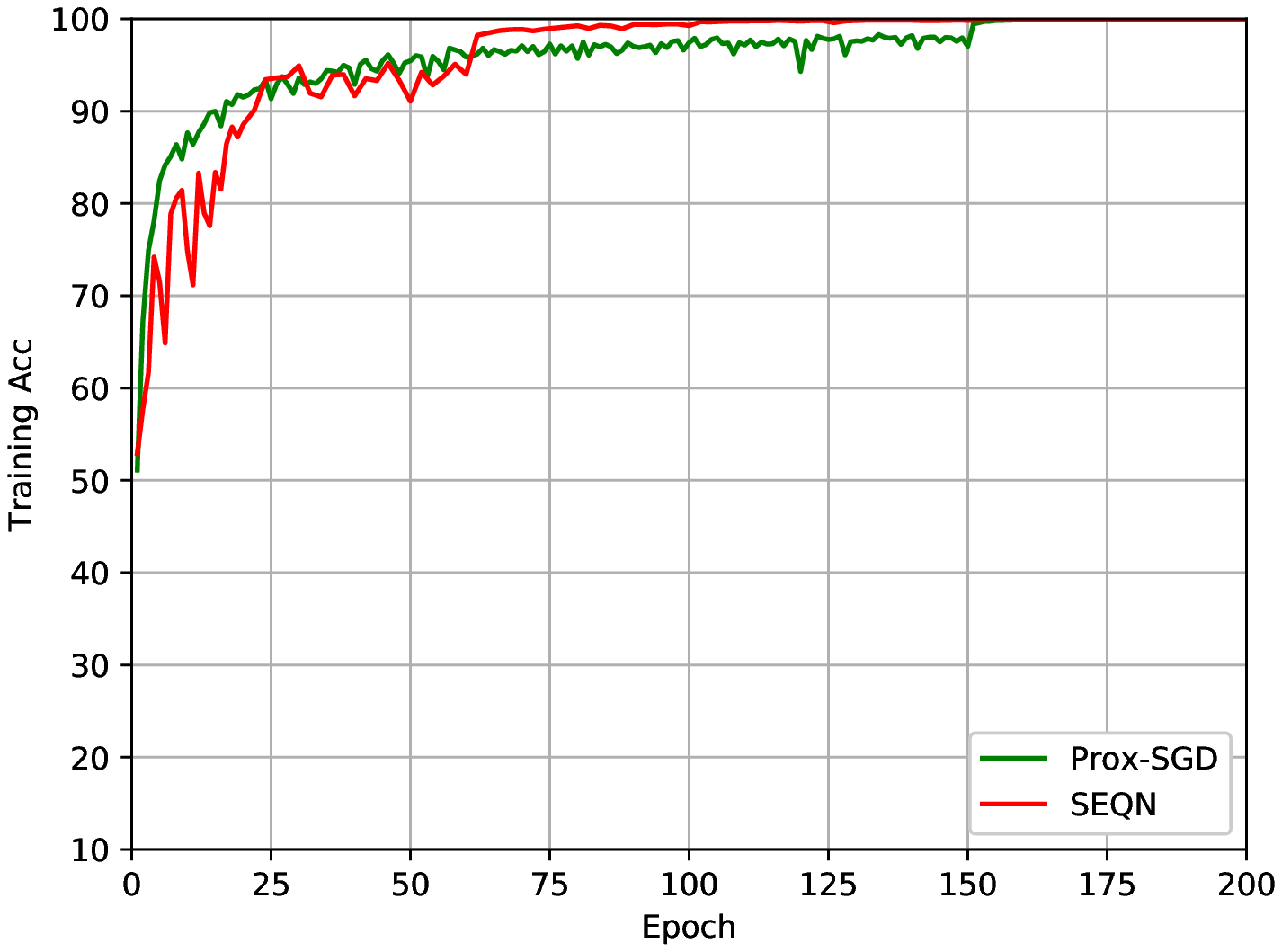}} \\
\subfloat[$\mathtt{ConvNet}$: Train. loss]{\includegraphics[width=3.8cm,trim={0.5cm 0 1cm 1cm},clip]{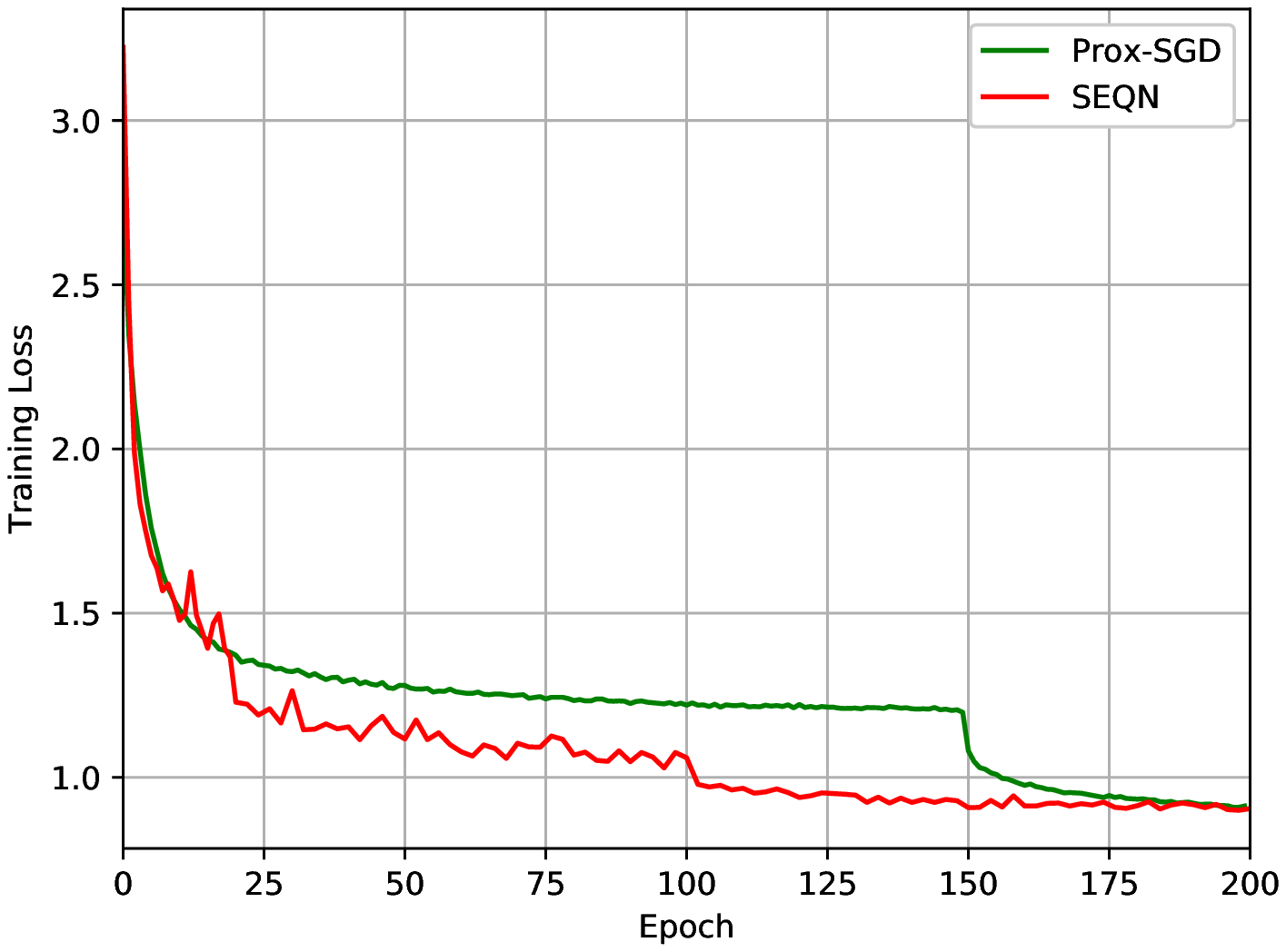}} \hspace{.5ex} 
\subfloat[$\mathtt{VGG}$-$\mathtt{16}$: Train. loss]{\includegraphics[width=3.8cm,trim={0.5cm 0 1cm 1cm},clip]{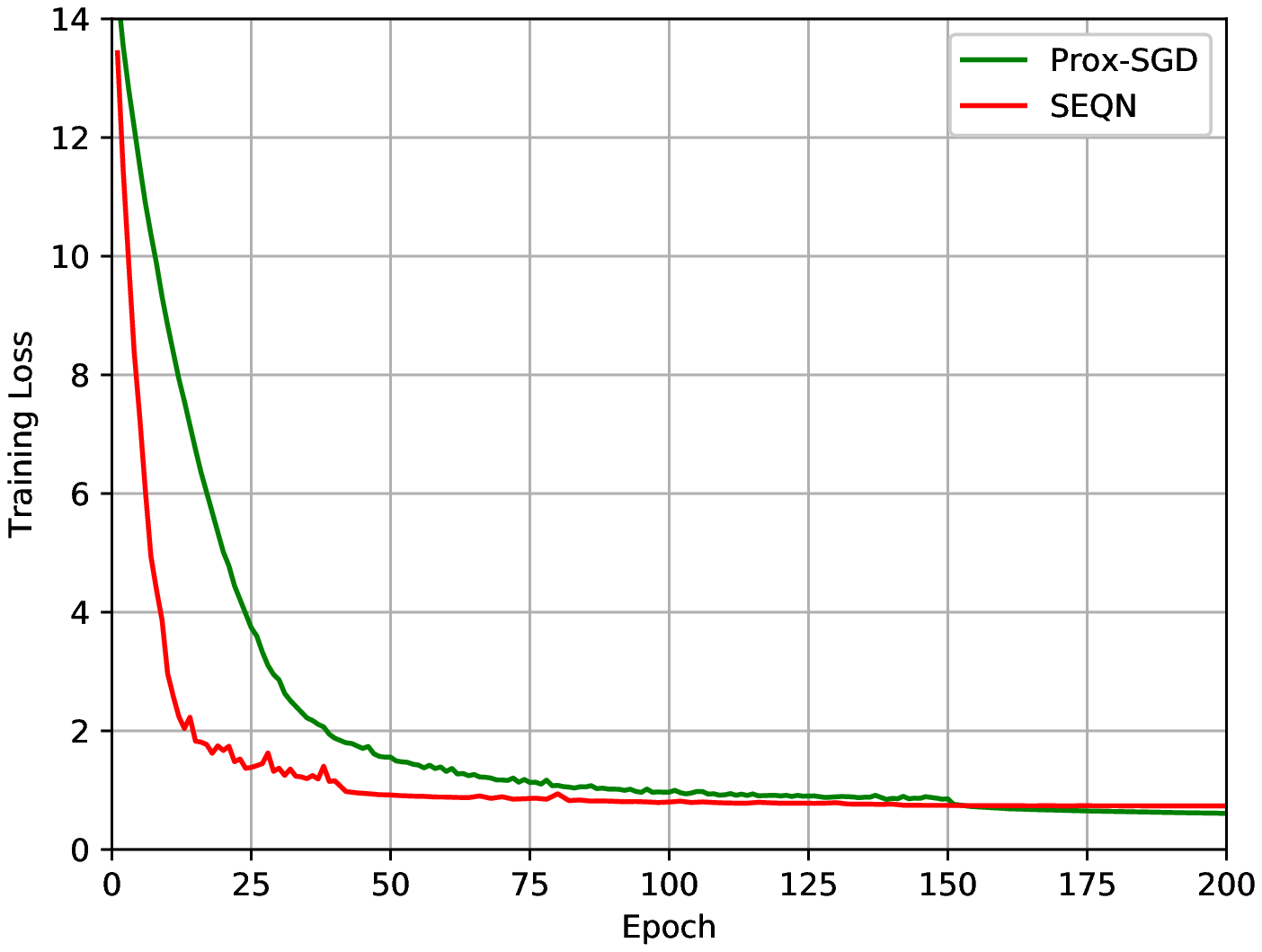}} \hspace{.5ex}
\subfloat[$\mathtt{ResNet}$-$\mathtt{18}$: Train. loss]{\includegraphics[width=3.8cm,trim={0.5cm 0 1cm 1cm},clip]{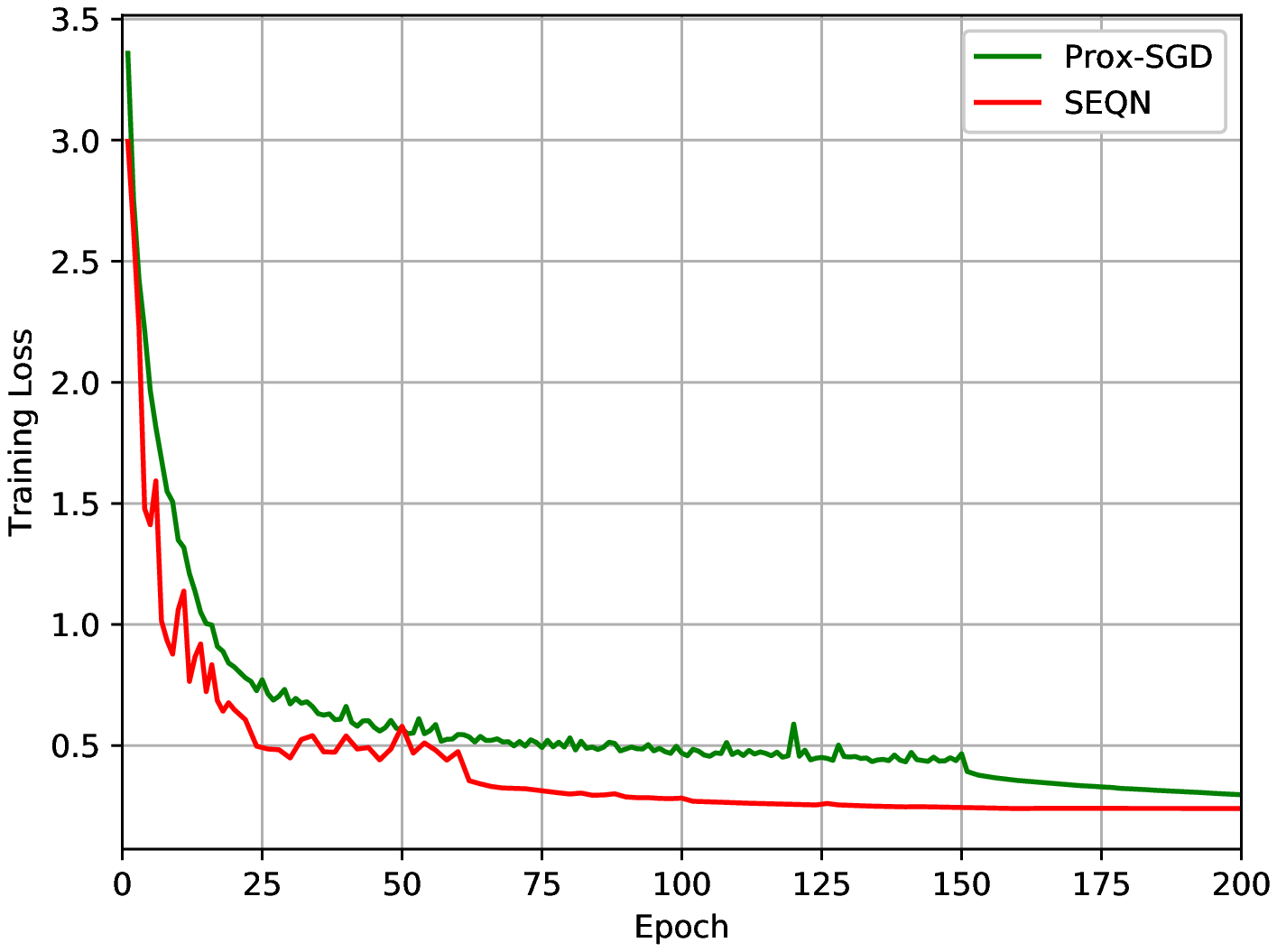}} \\
\subfloat[$\mathtt{ConvNet}$: Testing accuracy]{\includegraphics[width=3.8cm,trim={0.5cm 0 1cm 1cm},clip]{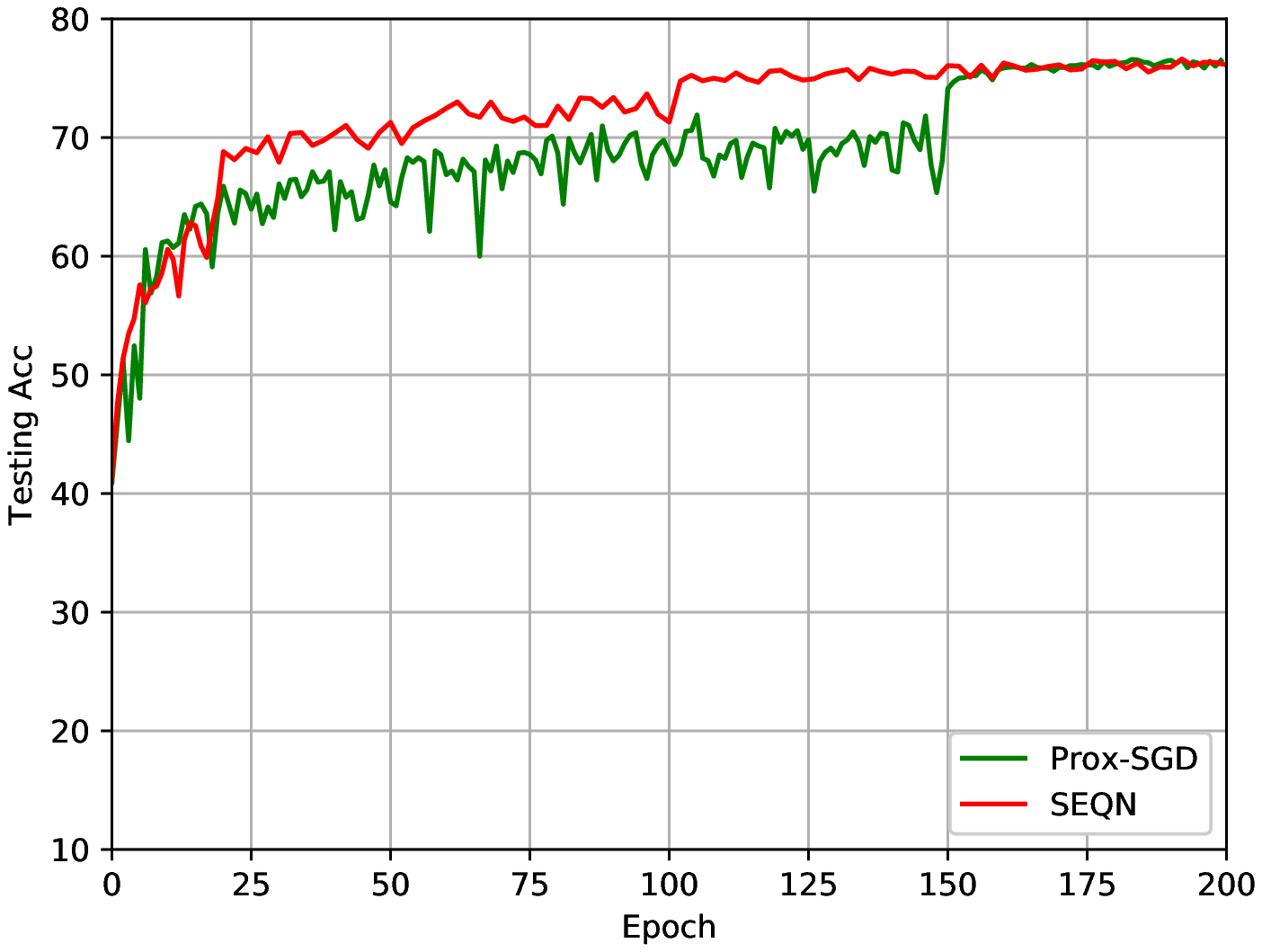}} \hspace{.5ex}
\subfloat[$\mathtt{VGG}$-$\mathtt{16}$: Testing accuracy]{\includegraphics[width=3.8cm,trim={0.5cm 0 1cm 1cm},clip]{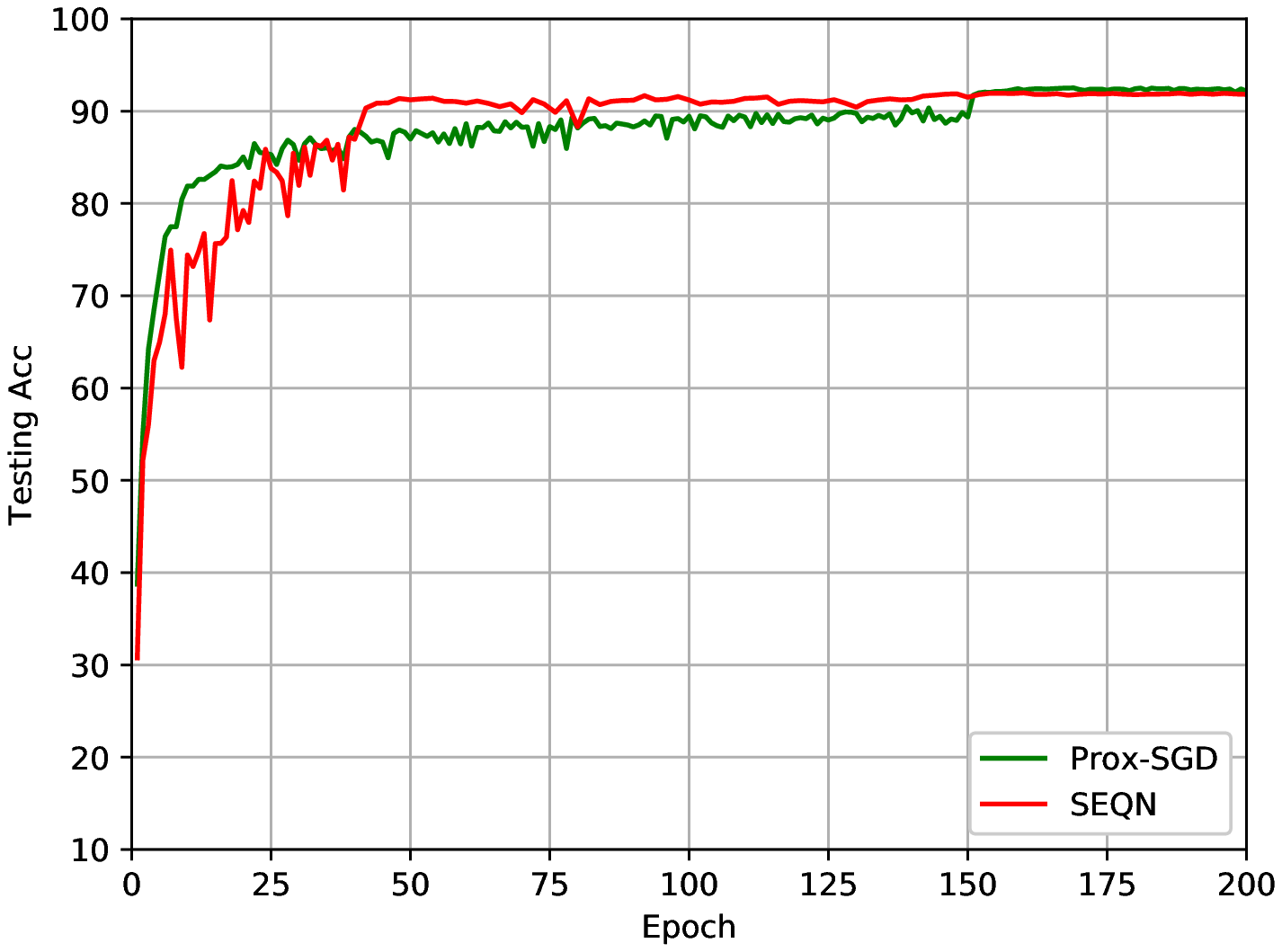}} \hspace{.5ex}
\subfloat[$\mathtt{ResNet}$-$\mathtt{18}$: Testing accuracy]{\includegraphics[width=3.8cm,trim={0.5cm 0 1cm 1cm},clip]{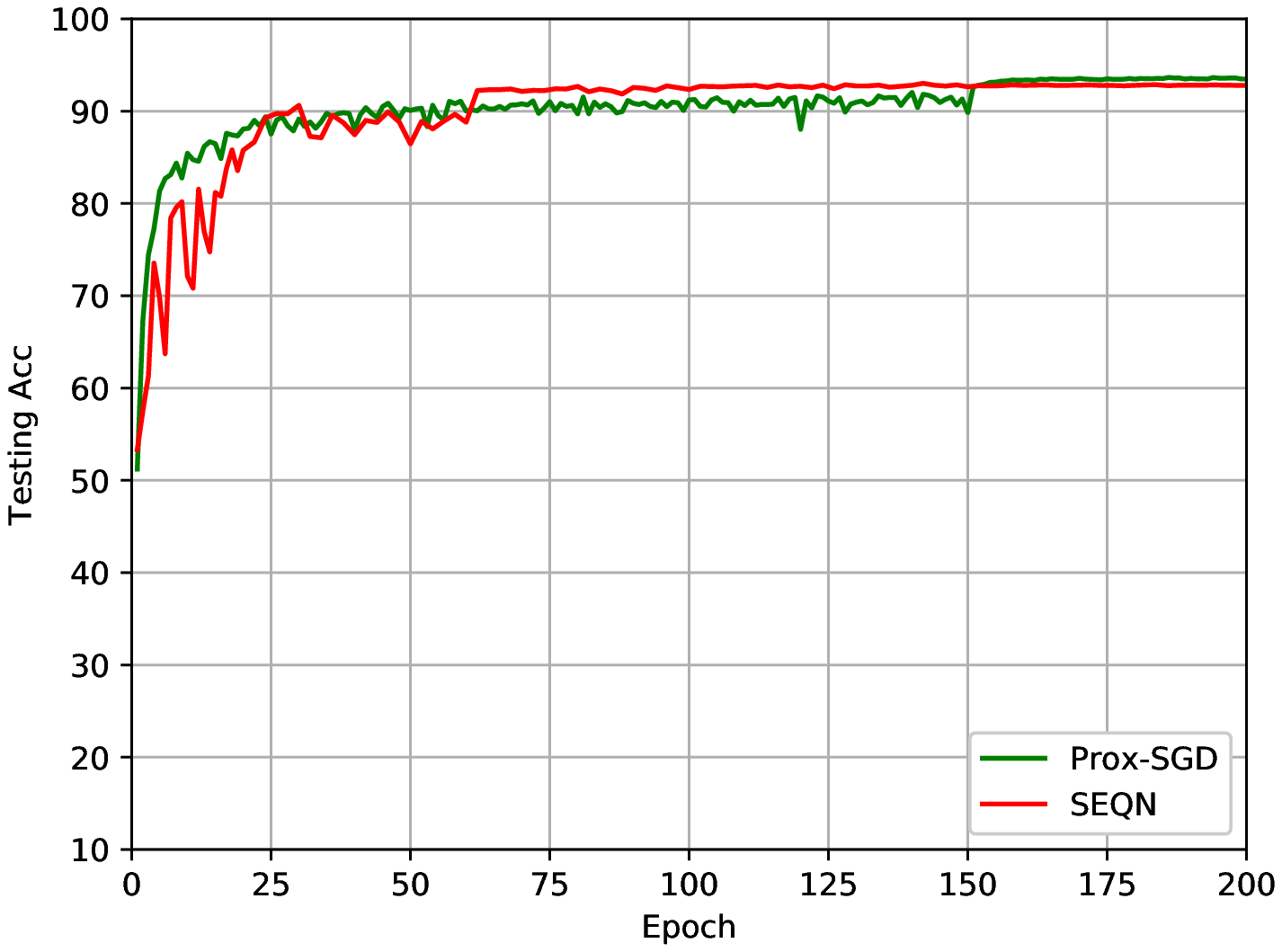}} \\
\caption{Numerical results for neural networks on the dataset $\mathtt{Cifar}$-$\mathtt{10}$.}
\label{test_neural}
\end{figure}

From the figure, it can be seen that SEQN performs comparable to Prox-SGD. In $\mathtt{ConvNet}$, SEQN appears to outperform Prox-SGD in terms of all three criteria. In the early stage, SEQN behaves similar to Prox-SGD, but it then jumps to a much better training and testing accuracy around the 25th epoch. In the final stage, SEQN attains a similar and good testing accuracy as Prox-SGD. The results not only illustrate faster convergence in the training process but also the good generalization of SEQN.
 
The results for $\mathtt{VGG}$-$\mathtt{16}$ and $\mathtt{ResNet}$-$\mathtt{18}$ are shown in the second and third column of Figure~\ref{test_neural}, respectively, and more detailed statistics are given in the two Tables \ref{vggnet-table}-\ref{Resnet-table}. In $\mathtt{VGG}$-$\mathtt{16}$, SEQN manages to decrease the training loss significantly -- compared to Prox-SGD -- but has a lower training and testing accuracy. However, the testing accuracy of SEQN reaches $91.23\%$ at the 50th epoch after $14$ minutes while the testing accuracy of Prox-SGD is still $89.37\%$ at the 150th epoch after approximately $41$ minutes. Although the best testing accuracy of SEQN is slightly worse than that of Prox-SGD, the favorable training loss and training/testing accuracy after a relatively low number of epochs is a clear advantage of SEQN. The performance on $\mathtt{ResNet}$-$\mathtt{18}$ is similar. The testing accuracy of SEQN reaches $92.14\%$ at the 70th epoch after $43$ minutes while the testing accuracy of Prox-SGD is still $92.03\%$ at the 140th epoch after $54$ minutes. We should point out that the computational costs of SEQN can be reduced if the implementation of the algebraic operations is further improved. 


\section{Conclusion} \label{sec:conclusion}
In this paper, we present a novel stochastic extra-step quasi-Newton method for solving nonsmooth nonconvex composite problems. We assume  that cheap stochastic approximation techniques are used to access objective function and gradient information. The proposed two-step scheme can be interpreted as a hybrid method combining higher order-type steps with an extra proximal gradient step. Our framework covers several existing algorithms and a large variety of stochastic approximations, such as mini-batch-type oracles or variance reduction techniques, can be utilized. We establish global convergence to stationary points in expectation and almost surely, if the step sizes are chosen appropriately to balance the variance of the stochastic gradients. The numerical experiments on sparse logistic regression strongly indicate the advantages of incorporating stochastic quasi-Newton directions. The proposed algorithm is also tested on challenging nonconvex deep learning problems. Our preliminary results demonstrate the potential of stochastic higher order information in sparse deep learning and a comparable performance to state-of-the-art stochastic first order methods is achieved.

\section{Appendix: Proofs of Auxiliary Results}


\subsection{Proof of Lemma \ref{lemma:point-diff}} \label{sec:app-A1}

\begin{proof}
As in \cite[Lemma 3.2]{DavDru18-2}, applying the definition and characterization \eqref{eq:prox-opt} of the proximal operator, it follows 
\begin{align*} \bar x = \proxt{\Theta}{\psi}(x)  \, 
& \iff \, \bar x \in x - \Theta^{-1}[\nabla f(\bar x) + \partial \vp(\bar x)] \\ 
& \iff \, \bar x\in \bar x - \Lambda_{+}^{-1} [ \nabla f(\bar x) + \Theta (\bar x - x)] - \Lambda_{+}^{-1} \partial \vp(\bar x) \\ 
& \iff \, \bar x = \proxt{\Lambda_{+}}{\vp}(\bar x - \Lambda_{+}^{-1}\nabla f(\bar x) - \Lambda_{+}^{-1}\Theta[\bar x - x]). \end{align*} 
Now, setting $z := x + \beta d$ and $q := \alpha d + \Lambda_{+}^{-1}(\nabla f(x) - \nabla f(z))$, we have $p = \Lambda_{+}^{-1}[\nabla f(\bar x) - \nabla f(x)] + q$ and $\|q\| \leq (\alpha + L_f \beta\lambda_{+}) \|d\| $. Furthermore, using Young's inequality, the nonexpansiveness of the proximity operator and the Lipschitz continuity of $\nabla f$, we obtain
\begingroup
\allowdisplaybreaks
\begin{align*} 
&\hspace{-2ex}\| p_+ - \bar x \|^2 \\
&=\|\proxt{\Lambda_{+}}{\vp}(\bar x - \Lambda_{+}^{-1}\nabla f(\bar x) - \Lambda_{+}^{-1}\Theta[\bar x - x])- \proxt{\Lambda_+}{\vp}(x+\alpha d - \Lambda^{-1}_+v_+)\|^2\\
&\leq \|(I - \Lambda_{+}^{-1}\Theta)[\bar x - x] - p + \Lambda_{+}^{-1}(v_+ - \nabla f(z))\|^2 \\ 
& = \|(I - \Lambda_{+}^{-1}\Theta)[\bar x - x] - \Lambda_{+}^{-1} [\nabla f(\bar x) - \nabla f(x)]\|^2 \\ 
& \hspace{2ex} - 2\iprod{(I-\Lambda_{+}^{-1}\Theta)[\bar x - x]}{q} - 2 \lambda_{+} \iprod{\nabla f(\bar x) - \nabla f(x)}{q} + \|q\|^2 \\
& \hspace{2ex} + 2 \lambda_{+} \iprod{(I-\Lambda_{+}^{-1}\Theta)[\bar x - x] -p}{v_+ - \nabla f(z)} + \lambda_{+}^2 \|\nabla f(z) - v_+\|^2 \\ 
&\leq \left[ (1+\rho_1)\left[1 - \frac{\lambda_{+}}{\theta}\right]^2 + 2\lambda_{+}\left[1 - \frac{\lambda_{+}}{\theta}\right]L_f +(1+\rho_2)L_f^2\lambda_{+}^2 \right] \|\bar x - x\|^2 \\ 
& \hspace{2ex} + \left[ 1+\frac{1}{\rho_{1}}+\frac{1}{\rho_{2}} \right]  \mu^2\|d\|^2+ \lambda_{+}^2 \|\nabla f(z) - v_+\|^2 \\
&\hspace{2ex}+2 \lambda_{+}  \iprod{(I-\Lambda_+^{-1}\Theta)[\bar x - x] -p}{v_+ - \nabla f(z)},
\end{align*}
\endgroup
where $\mu = \alpha + L_f \beta \lambda_+$. This establishes the statement in Lemma \ref{lemma:point-diff}.
\end{proof}

\subsection{Proof of Corollary  \ref{corollary:complexity}}
\begin{proof} We need to verify that the choice of $\lambda$ and $\lambda_+$ satisfies the constraints derived in Theorem \ref{theorem:complex} and in \eqref{eq:def-lamplus-app}, respectively. Notice that we can set $\bar \rho = 0$ and we can work with $(\lambda_+^m)^{-1}$ instead of $(2\lambda_+^m)^{-1}$ in \eqref{eq:def-lamplus-app}. Due to the definition of $b$ and $b_+$, we have $\tau_m  = \sqrt{2}$ for all $m$ and thus, it follows
\begin{align*} L_f - \frac{1}{\lambda_+} + \frac{K(K-1)}{2} \theta(\lambda_+) & \leq L - \frac{1}{\lambda_+} +  \frac{L^2 K^2}{2} \left[ \frac{1}{b_+} + \frac{1}{b} \right] \lambda_+  = 0. \end{align*}
Furthermore, with the choice $\lambda_+ = \gamma L^{-1}$ it holds that
 \begin{align*} \mathcal L_k^m &= \left[ \frac{(\alpha_k^m)^2}{\lambda_+} + 2L_f \alpha_k^m \beta_k^m + L_f^2 (\beta^m_k)^2 \lambda_+ + \frac{L^2(\beta_k^m)^2}{K} \lambda_+ \right] (\nu^m_k)^2 + \frac{1}{\lambda_+} \\ & \leq \left[\frac{1}{\gamma} + 2 + \gamma + \frac{\gamma}{K}\right] L  \bar \nu^2 + \frac{L}{\gamma} \leq \frac{L}{\gamma}(1+3\bar\nu^2) \end{align*}
and we have $\mathcal L^m_k \leq \frac{1}{\lambda}$ for all $k$ and $m$. Hence, the estimate \eqref{eq:cor-complex} directly follows from \eqref{eq:esti-for-complex}. 

Since each iteration in the inner loop of Algorithm \ref{alg:inex-qnm} requires $b+b_+$  gradient component evaluations (IFO), the total amount of evaluations in a single outer iteration is given by $N + K(b + b_+)$ and the corresponding total number of gradient component evaluations after $M$ outer iterations is
\[ MK \cdot \left[ \frac{N}{K} + b+b_+ \right] = MK \cdot \left [\frac{N}{K} + 2K^2 \right]. \] 
Moreover, since the bound for $\Exp[\|F^I({\sf X})\|^2]$ in \eqref{eq:cor-complex} is proportional to $(MK)^{-1}$, the IFO complexity of Algorithm \ref{alg:inex-qnm} for reaching an $\veps$-accurate stationary point with $\Exp[\|F^I({\sf X})\|^2] \leq \veps$ is $\mathcal O((N/K + 2K^2)/\veps)$. Minimizing this expression with respect to $K$ yields $K \sim N^{1/3}$ and the IFO complexity $\mathcal O(N^{2/3}/\veps)$. 
\end{proof}

\bibliographystyle{spmpsci}
\bibliography{ssn_ex_biblio}

\end{document}